\documentclass{amsart}
\usepackage[margin=2.8cm]{geometry}
\usepackage{amssymb}
\usepackage{amsmath}
\usepackage{graphicx} 
\usepackage[mathscr]{euscript}
\usepackage{enumerate}
\usepackage{xspace}
\usepackage{color}
\usepackage{enumitem}
\usepackage{amsthm}
\usepackage{dsfont}
\usepackage{mathrsfs}
\usepackage[scr=boondoxo]{mathalfa}
\usepackage{bbm}
\usepackage[colorlinks=true, linkcolor = blue, citecolor = blue]{hyperref}
\usepackage[numbers]{natbib}
\usepackage[backgroundcolor=white,bordercolor=red]{todonotes}
\DeclareMathAlphabet{\mathpzc}{OT1}{pzc}{m}{it}
\usepackage[nameinlink]{cleveref}

\allowdisplaybreaks

\begin{document}

\theoremstyle{plain}
\newtheorem{theorem}{Theorem}[section]
\newtheorem{lemma}[theorem]{Lemma}
\newtheorem{condition}[theorem]{Condition}
\newtheorem{proposition}[theorem]{Proposition}
\newtheorem{corollary}[theorem]{Corollary}
\newtheorem{Ass}[theorem]{Assumption}

\theoremstyle{definition}
\newtheorem{definition}[theorem]{Definition}
\newtheorem{remark}[theorem]{Remark}
\newtheorem{discussion}[theorem]{Discussion}
\newtheorem{intuition}[theorem]{Intuition}
\newtheorem{example}[theorem]{Example}
\newtheorem{SA}[theorem]{Standing Assumption}
\newtheorem*{observation}{Observations}
\newtheorem*{RL}{Comments on Related Literature}

\renewcommand{\chapterautorefname}{Chapter} 
\renewcommand{\sectionautorefname}{Section} 

\crefname{lemma}{lemma}{lemmas}
\Crefname{lemma}{Lemma}{Lemmata}
\crefname{corollary}{corollary}{corollaries}
\Crefname{corollary}{Corollary}{Corollaries}

\newcommand{\m}{\mathfrak{m}}
\newcommand{\mo}{\mathfrak{m}^\circ}
\newcommand{\mon}{\mathfrak{m}^{\circ, n}}
\newcommand{\M}{\tilde{\mathfrak{m}}}
\newcommand{\Mo}{\tilde{\mathfrak{m}}^\circ}
\newcommand{\Mon}{\tilde{\mathfrak{m}}^{\circ, n}}
\newcommand{\s}{\mathfrak{s}}
\renewcommand{\S}{\tilde{\mathfrak{s}}}
\newcommand{\n}{\widetilde{\mathfrak{m}}}
\renewcommand{\v}{\mathfrak{v}}
\renewcommand{\u}{\mathfrak{u}}
\newcommand{\q}{\mathfrak{q}}
\def\stackrelboth#1#2#3{\mathrel{\mathop{#2}\limits^{#1}_{#3}}}
\renewcommand{\l}{\mathscr{l}}
\newcommand{\E}{{\mathds{E}}}
\renewcommand{\P}{\mathds{P}}
\newcommand{\Po}{\mathds{P}^\circ}
\newcommand{\Pon}{\mathds{P}^{\circ, n}}
\newcommand{\Q}{\tilde{\mathds{P}}}
\newcommand{\Qo}{\tilde{\mathds{P}}^\circ}
\newcommand{\Qon}{\tilde{\mathds{P}}^{\circ, n}}
\newcommand{\W}{\mathds{W}}
\newcommand{\QQ}{\mathds{Q}}
\newcommand{\tQQ}{\tilde{\mathds{Q}}}
\newcommand{\on}{\operatorname}
\newcommand{\oU}{U}
\newcommand{\of}{[\hspace{-0.06cm}[}
\newcommand{\gs}{]\hspace{-0.06cm}]}
\newcommand{\ofr}{(\hspace{-0.09cm}(}
\newcommand{\gsr}{)\hspace{-0.09cm})}
\renewcommand{\emptyset}{\varnothing}

\renewcommand{\theequation}{\thesection.\arabic{equation}}
\numberwithin{equation}{section}

\newcommand{\1}{\mathds{1}}
\renewcommand{\epsilon}{\varepsilon}
\newcommand{\X}{\mathsf{X}}
\newcommand{\Y}{\mathsf{Y}}
\newcommand{\Z}{\mathsf{Z}}
\newcommand{\B}{\mathsf{B}}
\newcommand{\f}{\mathfrak{f}}
\newcommand{\g}{\mathfrak{g}}
\newcommand{\e}{\mathfrak{e}}
\renewcommand{\t}{T}
\newcommand{\A}{J_{\on{sep}}}

\title[Separating Times for Diffusions]{Separating Times for One-Dimensional General Diffusions}

\author[D. Criens]{David Criens}
\address{D. Criens - Albert-Ludwigs-University of Freiburg, Ernst-Zermelo-Str. 1, 79104 Freiburg, Germany.}
\email{david.criens@stochastik.uni-freiburg.de}

\author[M. Urusov]{Mikhail Urusov}
\address{M. Urusov - University of Duisburg-Essen, Thea-Leymann-Str. 9, 45127 Essen, Germany.}
\email{mikhail.urusov@uni-due.de}


\makeatletter
\@namedef{subjclassname@2020}{\textup{2020} Mathematics Subject Classification}
\makeatother


\thanks{We are very grateful to Paul Jenkins for pointing out a gap in a previous version of this paper.
Moreover, we thank Zhesheng Liu and Mihail Zervos for sharpening our thinking on the difference between diffusion and semimartingale local times.
Last but not least, we thank the anonymous referee for numerous constructive comments and suggestions that helped us improve the manuscript.
DC acknowledges financial support from the DFG project No. SCHM 2160/15-1.}

\date{\today}

\begin{abstract}
The separating time for two probability measures on a filtered space is an extended stopping time which captures the phase transition between equivalence and singularity. More specifically, two probability measures are equivalent before their separating time and singular afterwards. In this paper, we investigate the separating time for two laws of general one-dimensional regular continuous strong Markov processes, so-called general diffusions, which are parameterized via scale functions and speed measures. Our main result is a representation of the corresponding separating time as
(loosely speaking) a hitting time
of a deterministic set which is characterized via speed and scale.
As hitting times are fairly easy to understand, our result gives access to explicit and easy-to-check sufficient and necessary conditions for two laws of general diffusions to be (locally) absolutely continuous and/or singular.
Most of the related literature treats the case of stochastic differential equations.
In our setting we encounter several novel features, which are due to general speed and scale on the one hand, and to the fact that we do not exclude (instantaneous or sticky) reflection on the other hand.
These new features are discussed in a variety of examples.
As an application of our main theorem, we investigate the no arbitrage concept no free lunch with vanishing risk (NFLVR) for a single asset financial market whose (discounted) asset is modeled as a general diffusion which is bounded from below (e.g., non-negative).
More precisely, we derive deterministic criteria for NFLVR and we identify the (unique) equivalent local martingale measure as the law of a certain general diffusion on natural scale.
\end{abstract}

\maketitle

\frenchspacing
\pagestyle{myheadings}

\section{Introduction}
The purpose of this paper is to give explicit and easy-to-check sufficient and necessary conditions for two laws of general one-dimensional regular continuous strong Markov processes
(so-called \emph{general diffusions}) to be (locally) absolutely continuous and/or singular. 
In contrast to the (strict) subclass of diffusions that can be characterized via stochastic differential equations (SDEs), and which are called \emph{It\^o diffusions} in this paper, general diffusions are, in general, not semimartingales. The law of a general diffusion can be characterized via two deterministic objects, namely its \emph{scale function} and its \emph{speed measure}. 

Our interest in general diffusions is motivated by applications to mathematical finance. It is well-known that It\^o diffusion models with non-vanishing volatility coefficients cannot reflect occupation time effects such as stickiness of the price curve around some price. However, such effects can, for instance, be observed when a company got a takeover offer. 
\begin{figure}
	\label{fig: hansen}
	\includegraphics[width=12cm, height=5.5cm]{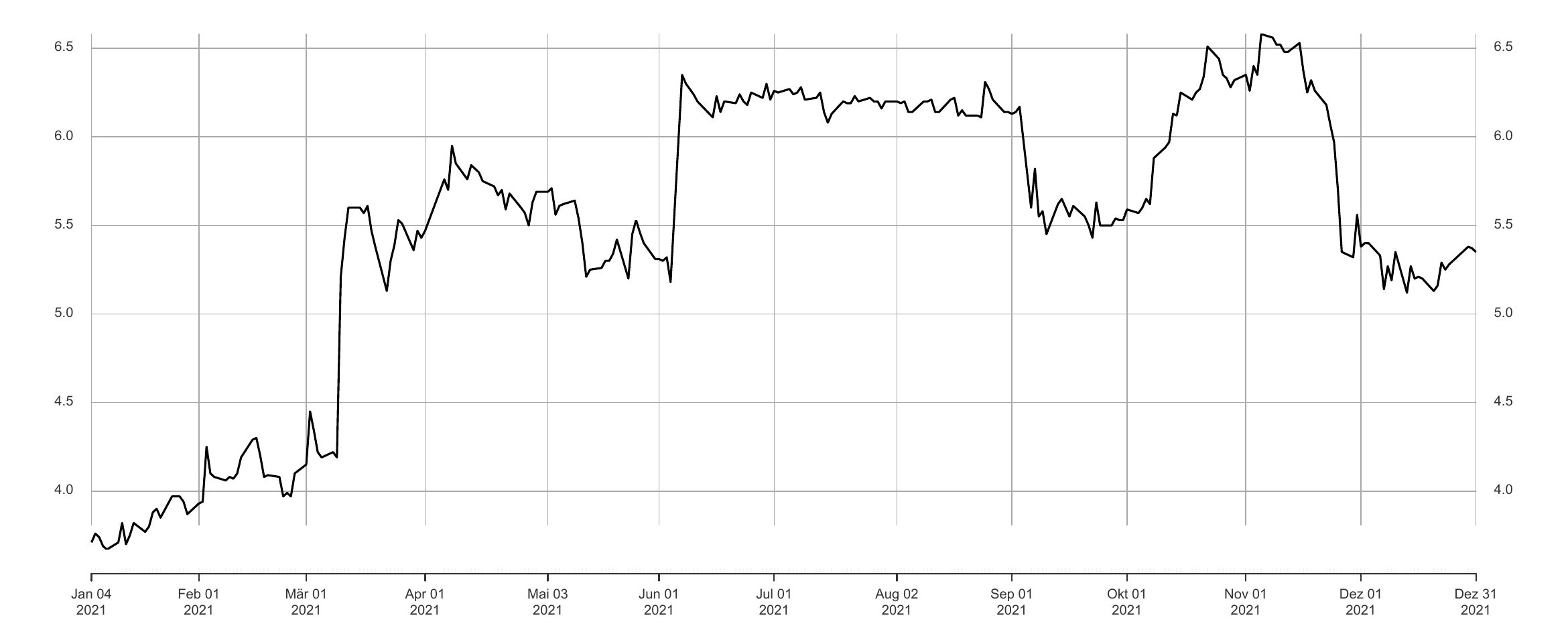} \caption{Share Price \emph{Hansen Technologies Limited (HSN.AX)}}
\end{figure}
To give an explicit example, we draw the reader's attention to Figure \ref{fig: hansen}, which shows the share price of Hansen Technologies Limited. On June~7, 2021, BGH Capital made an offer to acquire Hansen Technologies Limited at a price of AUD 6.5 in cash per share. The acquisition was cancelled by BGH Capital on September 6, 2021.\footnote{\tiny https://www.marketscreener.com/quote/stock/HANSEN-TECHNOLOGIES-LIMIT-6500487/news/BGH-Capital-Fund-I-managed-by-BGH-Capital-cancelled-the-acquisition-of-Hansen-Technologies-Limited-36355691/ (retrieved on January 31, 2022, 17:45 CET)} Figure~\ref{fig: hansen} shows a sticky behavior of the price curve while the takeover offer was active. 
Contrary to It\^o diffusions with non-vanishing volatility, general diffusions are able to capture such occupation time effects. 

Further, general diffusions can be used to model non-negative volatility processes which are allowed to vanish and reflect off the origin. For instance, there is empirical evidence (see, e.g., \cite[Section 6.3.1]{clark}) that for the classical Heston model the Feller condition is often violated, which means its volatility process can reach the origin and reflect from it.
In general, it might not always be possible to capture such properties with It\^o diffusions or semimartingales. 
The class of general diffusions is more flexible when it comes to modeling fine local time effects.

Let us now turn to our main results. Studying (local) absolute continuity and singularity of probability measures on a path space has a very long tradition, see \cite[pp. 631--634]{JS} and \cite[pp. 314--315]{LS} for bibliographic notes. A classical approach to tackle such questions is by studying certain candidate density processes. In this paper we use a somehow different idea based on the concept of \emph{separating times}.
For two given probability measures on a filtered space there exists (\cite{cherUru}) a so-called separating time which captures the phase transition between equivalence and singularity. More precisely, the two probabilities are equivalent before their separating time and singular afterwards. It is possible to describe all standard notions such as (local) absolute continuity, equivalence and singularity via the separating time. Broadly speaking, in case the separating time is known in a tractable form, one can answer any type of absolute continuity or singularity question. In particular, the separating time even encodes the information \emph{when} absolute continuity is lost.

Our main result is a representation of the separating time for two laws of general diffusions as ``something like''
the hitting time of a deterministic set which is described in terms of the scale functions and the speed measures.
As easy corollaries, we obtain necessary and sufficient conditions for two laws of general diffusions to be (locally) absolutely continuous and/or singular.
In contrast to such results for It\^o diffusions, in our setting we encounter several novel features.
Some of them are due to general scale functions and speed measures.
Other ones are due to the fact that we allow boundary points of the state space to be accessible and, in this case, allow both absorption and (instantaneous or sticky) reflection.
These new features are illustrated by a variety of examples.
As hitting times are fairly easy to understand, our result appears to be very useful
for applications.

To illustrate such an application, we consider a single asset financial market where the discounted asset price is modeled as a general diffusion bounded from below (e.g., non-negative).
We establish the existence of a unique candidate for an equivalent local martingale measure (ELMM), which is a certain general diffusion on natural scale, and,  in terms of scale and speed, we give \emph{precise deterministic criteria} for the candidate to be an ELMM. In financial terms, our result provides a precise deterministic description of the no arbitrage concept \emph{no free lunch with vanishing risk} (NFLVR) as introduced in \cite{DS}.

We comment now on related literature. 
The separating time for two solutions to SDEs under the Engelbert--Schmidt conditions was computed in \cite{cherUru} (see also \cite{MU12}). 
A characterization of local absolute continuity of two general diffusions with open state space was proved in \cite{orey}. Both of these results follow directly from our main theorem.

Local absolute continuity is closely related to the martingale property of a
certain
local martingale, which can be viewed as a candidate density process.
In this connection, we mention the paper \cite{Kotani2006}, which provides necessary and sufficient conditions for the martingale property of diffusions on natural scale.
These conditions are expressed solely in terms of the speed measure of the diffusion.
For the same general framework as considered in this paper, a characterization of the martingale property of certain non-negative continuous local martingales was recently proved in \cite{desmettre}. 
In Remark~\ref{rem: comment desm} we comment in more detail on the relation to this work.

Among other things, the existence and absence of NFLVR for one-dimensional It\^o diffusion markets has been studied in \cite{MU12b}. Our work extends those from \cite{MU12b} (for NFLVR) to a general diffusion framework. In particular, when compared to \cite{MU12b}, we make an interesting structural observation. Namely, in the It\^o diffusion framework from \cite{MU12b} the scale functions and the speed measures are assumed to be absolutely continuous w.r.t. the Lebesgue measure. We prove that the absolute continuity of the scale functions (but not of the speed measures) is a \emph{necessary} condition for NFLVR. At first glance, this appears to be quite surprising.

In general, (local) absolute continuity and singularity have also been studied for frameworks which are multidimensional, which include jumps and which are even non-Markovian. For instance, It\^o diffusions with and without jumps were studied in \cite{CFY,criens21}, (non-Markovian) continuous It\^o processes were studied in \cite{criens20,ruf15}, and more general semimartingales have been considered in \cite{criensglau}. Of course, this short list is quite incomplete and we refer to the references of these papers for more literature. 
We contribute to this list, as our paper seems to be the first to cover the full class of general diffusions. 

This article is structured as follows.
In Section~\ref{sec: sep time} we recall the concept of separating times.
An introduction to our canonical  diffusion framework is given in  Section~\ref{sec: diff basics}.
Our main result is presented in Section~\ref{sec: main} and examples are discussed in Section~\ref{sec: examples}.
The financial application is given in Section~\ref{sec: appl}.
Finally, the proof of our main result is presented in Section~\ref{sec: pf}.
We also added three appendices to make our paper as self-contained as possible.
In Appendix~\ref{app: A} we discuss the martingale problem for general diffusions, which we also think to be of independent interest.
Further technical facts about general diffusions, semimartingales and differentiation of measures are collected in Appendicies \ref{app: B} and~\ref{app: C}.

\section{Separating Times for General Diffusions}
This section is divided into four parts.
In the first two parts we recall the concept of separating times and we introduce our canonical diffusion setting.
In the third part we present our main results and finally we discuss some corollaries and examples.

\subsection{Separating Times: Encoding Absolute Continuity and Singularity}\label{sec: sep time}

We take a filtered space \((\Omega, \mathcal{F}, (\mathcal{F}_t)_{t\in[0,\infty)})\) with a right-continuous filtration.
Recall that, for a stopping time $\tau$, the $\sigma$-field $\mathcal F_\tau$ is defined by the formula
\[
\mathcal F_\tau \triangleq \big\{B\in\mathcal F \colon B\cap\{\tau\le t\}\in\mathcal F_t \text{ for all } t\in[0,\infty) \big\},
\]
in particular, $\mathcal F_\infty=\mathcal F$.
Let \(\delta\) be a point outside \([0, \infty]\) and endow \([0, \infty] \cup \{\delta\}\) with the ordering \(\delta > t\) for all~\(t \in [0, \infty]\). 

\begin{definition}
	A map \(S \colon \Omega \to [0, \infty] \cup \{\delta\}\) is called an \emph{extended stopping time} if \(\{S \leq t\} \in \mathcal{F}_{t}\) for all \(t \in [0, \infty]\).
\end{definition}

Let \(\P\) and \(\Q\) be two probability measures on \((\Omega, \mathcal{F})\).
The following result is a restatement of \cite[Theorem~2.1]{cherUru}.

\begin{theorem}\label{th:020323a1}
	There exists a \(\P, \Q\)-a.s. unique extended stopping time \(S\) such that for any stopping time \(\tau\) we have 
	\begin{align} \label{eq: sep time}
	\P \sim \Q \text{ on } \mathcal{F}_{\tau} \cap \{\tau < S\}, \qquad \P \perp \Q \text{ on } \mathcal{F}_\tau \cap \{\tau \geq S\}.
	\end{align}
	This extended stopping time is called the \emph{separating time for \(\P\) and \(\Q\)}.
\end{theorem}

We have already mentioned in the Introduction that separating times encode (local) absolute continuity and singularity properties. This claim is verified by the following proposition, which is a (sightly extended) restatement of \cite[Lemma~2.1]{cherUru}.

\begin{proposition}\label{prop: AC Sing}
	Let \(S\) be the separating time for \(\P\) and \(\Q\) and let $t\in[0,\infty]$. Then, the following hold:
	\begin{enumerate}[label=\textup{(\roman*)}]
		\item
		\(\P \ll \Q\) if and only if \(\P\)-a.s. \(S = \delta\).
		\item
		\(\P \sim \Q\) if and only if \(\P, \Q\)-a.s. \(S = \delta\).
		\item
		\(\P \ll_\textup{loc} \Q\) if and only if \(\P\)-a.s. \(S \geq \infty\).
		\item
		\(\P \sim_\textup{loc} \Q\) if and only if \(\P, \Q\)-a.s. \(S \geq \infty\).
		\item
		$\P\ll\Q$ on $\mathcal F_t$ if and only if \(\P\)-a.s. \(S>t\).
		\item
		$\P\sim\Q$ on $\mathcal F_t$ if and only if \(\P, \Q\)-a.s. \(S>t\).
		\item
		\(\P \perp \Q\) on \(\mathcal{F}_t\) if and only if \(\P, \Q\)-a.s. \(S \le t\) if and only if \(\P\)-a.s. \(S \le t\).
	\end{enumerate}
\end{proposition}

The following generalization of Proposition~\ref{prop: AC Sing} is sometimes also useful. Again, its proof is straightforward.

\begin{proposition}\label{prop:121224a1}
Let \(S\) be the separating time for \(\P\) and \(\Q\), let $\tau$ be a stopping time and $B\in\mathcal F_\tau$.
Then, the following hold:
\begin{enumerate}[label=\textup{(\roman*)}]
\item
$\P\ll\Q$ on $\mathcal F_\tau\cap B$ if and only if \(S>\tau\) \(\P\)-a.s. on $B$.
\item
$\P\sim\Q$ on $\mathcal F_\tau\cap B$ if and only if \(S>\tau\) \(\P, \Q\)-a.s. on $B$.
\item
The following are equivalent:
\begin{enumerate}[label=\textup{(\alph*)}]
\item
\(\P \perp \Q\) on \(\mathcal F_\tau\cap B\).
\item
\(S \le\tau\) \(\P, \Q\)-a.s. on $B$.
\item
\(S \le\tau\) \(\P\)-a.s. on $B$.
\end{enumerate}
\end{enumerate}
\end{proposition}

For several classes of probability measures the separating times are explicitly known. For instance, this is the case if \(\P\) and \(\Q\) are laws of one-dimensional It\^o diffusions under the Engelbert--Schmidt conditions (\cite{cherUru, MU12}) or Levy processes (\cite{cherUru}).
In Section \ref{sec: main} we prove an explicit characterization of the separating times for general one-dimensional regular continuous strong Markov processes. Our result generalizes those from \cite{cherUru, MU12} for It\^o diffusions in several directions. For instance, we allow sticky points in the interior and also (both instantaneous and slow) reflection at the boundaries.

\subsection{A Recap on the Canonical Diffusion Framework} \label{sec: diff basics}
The purpose of this section is to introduce our general setting and to fix some notation related to the theory of one-dimensional regular continuous strong Markov processes (general diffusions). 
Before we start this program, let us mention some of the most classical references. The definitive account of the theory is given in the seminal monograph \cite{itokean74} by It\^o and McKean. More introductory treatments can be found in \cite[Chapter~16]{breiman1968probability}, \cite[Chapter~2]{freedman}, \cite[Chapter~33]{kallenberg}, \cite[Section~VII.3]{RY} and \cite[Section~V.7]{RW2}. For an introduction with many useful examples, we also refer to
\cite[Chapter~15]{KarlinTaylor1981}. An overview on formulae is given in Chapter~II of the ``Handbook'' \cite{BorodinSalminen2002} by Borodin and Salminen.

Let \(J \subset [- \infty, \infty]\) be a bounded or unbounded, closed, open or half-open interval. We denote the closure of \(J\) in the extended reals \([- \infty, \infty]\) by \(\on{cl}(J)\), its interior by \(J^\circ\) and its boundary by \(\partial J\ (= \on{cl}(J) \setminus J^\circ)\).
Further, we define \(\Omega\triangleq C(\mathbb R_+;[-\infty,\infty])\), the space of continuous functions \(\mathbb{R}_+ \to [- \infty, \infty]\). The coordinate process on \(\Omega\) is denoted by \(\X\), i.e., \(\X_t (\omega) = \omega(t)\) for \(t \in \mathbb{R}_+\) and \(\omega \in \Omega\). 
We also set \(\mathcal{F} \triangleq \sigma (\X_s, s \geq 0)\) and \(\mathcal{F}_t \triangleq \bigcap_{s > t}\sigma (\X_r, r \leq s)\) for \(t \in \mathbb{R}_+\). Except in Section~\ref{sec: appl}, where we discuss an application to mathematical finance, \((\Omega, \mathcal{F}, (\mathcal{F}_t)_{t \geq 0})\) will be the underlying filtered space.

In accordance with \cite[Definition~V.45.1]{RW2},
we call a map \((J \ni x \mapsto \P_x)\), from \(J\) into the set of probability measures on \((\Omega, \mathcal{F})\), a \emph{general diffusion} (with state space $J$) if the following hold:
\begin{enumerate}
	\item[\textup{(i)}] $\P_x(\X_0 = x) = 1$ and $\P_x(C(\mathbb R_+;J))=1$ for all $x\in J$;
		\item[\textup{(ii)}] the map \(x \mapsto \P_x(B)\) is measurable for all \(B \in \mathcal{F}\);
				\item[\textup{(iii)}] the \emph{strong Markov property} holds, i.e., for any stopping time \(\tau\) and any \(x \in J\), the kernel \(\P_{\X_\tau}\) is the regular conditional \(\P_x\)-distribution of \((\X_{t + \tau})_{t \geq 0}\) given \(\mathcal{F}_{\tau}\) on \(\{\tau < \infty\}\), i.e., \(\P_x\)-a.s. on~\(\{\tau < \infty\}\)
				\[  
				\P_x ( \X_{\cdot + \tau} \in d\omega | \mathcal{F}_\tau) = \P_{\X_\tau} (d \omega).
				\]
\end{enumerate}
For abbreviation, we call ``general diffusions'' simply ``diffusions" in this paper.
A diffusion \((x \mapsto \P_x)\) is called \emph{regular} if, for all \(x \in J^\circ\) and~\(y \in J\),
\begin{equation}\label{eq:121224a0}
\P_x(T_y < \infty) > 0,
\end{equation}
where\footnote{Later we use the notation~\eqref{eq:020323a1} for arbitrary $y\in[-\infty,\infty]$ regardless of whether $y$ belongs to $J$.}
\begin{equation}\label{eq:020323a1}
T_y \triangleq \inf (s \geq 0 \colon \X_s = y)
\end{equation}
with the usual convention $\inf(\emptyset) \triangleq \infty$.
In this article we will only work with regular diffusions.

Next, we recall the important concepts of \emph{scale} and \emph{speed}. There exists a strictly increasing continuous function \(\s \colon J \to \mathbb{R}\), which is unique up to increasing affine transformations, such that, for any interval \(I = (a, b)\) with \([a, b] \subset J\), we have
\begin{equation}\label{eq:260223a1}
\P_x(T_b < T_a)
= 1-\P_x(T_a < T_b)
=
\frac{\s(x) - \s(a)}{\s(b) - \s(a)}, \quad x \in I, 
\end{equation}
see \cite[Theorem 16.27]{breiman1968probability}.
Any such function
\(\s\) is called a \emph{scale function}.
In case \(\s (x) = x\), we say that the diffusion is on \emph{natural scale}.
For an interval \(I = (a, b)\) with \([a, b] \subset J\), we set
\begin{align} \label{eq: def G}
G_I (x, y) \triangleq \begin{cases}
\displaystyle\frac{2\hspace{0.03cm} (\s(x\wedge y) - \s(a))(\s(b) - \s(x\vee y))}{\s(b) - \s(a)},& x, y \in I, \\
0,&\text{otherwise},
\end{cases}
\end{align}
which is often called {\em Green's function}.
There exists a unique locally finite measure \(\m\) on \((J^\circ, \mathcal{B}(J^\circ))\) such that, for any interval \(I = (a, b)\) with~\([a, b] \subset J\), 
\begin{equation}\label{eq:090322a1}
\E_x \big[ T_a \wedge T_b\big] = \int G_I(x, y) \m(dy), \quad x \in I,
\end{equation}
see \cite[Theorem~VII.3.6]{RY}
(notice that $\m$ is determined uniquely given the version of $\s$;
if we replace $\s$ with $k\s+l$, $k>0$, $l\in\mathbb R$, then $\m$ is replaced with $\m/k$).
The measure \(\m\) is called the \emph{speed measure} of \((x \mapsto \P_x)\).\footnote{We remark that our speed measure is one half of the speed measure
in \cite{BorodinSalminen2002,freedman,RY}.
Our normalization is consistent
with \cite{breiman1968probability,KarlinTaylor1981,RW2}.
For instance, our speed measure of a Brownian motion is the Lebesgue measure (when we take the scale function $\s(x)=x$).}
Scale and speed determine the potential operator of the diffusion killed when exiting an interval.
We take a reference point \(c \in J^\circ\) and define, for \(x \in J^\circ\),
\begin{equation*} 
\begin{split}
\u (x) &\triangleq \begin{cases} \displaystyle \int_{(c, x]} \m((c, z]) \s(dz),& \text{for } x \geq c,\\\vspace{-
	0.4cm}\\
\displaystyle \int_{(x, c]} \m((z, c]) \s (dz),&\text{for } x \leq c,\end{cases}  
\\
\v (x) &\triangleq \begin{cases} \displaystyle \int_{(c, x]} \s((c,y]) \m(dy),&\text{for } x \geq c,\\ \vspace{-0.4cm}
\\
\displaystyle\int_{(x, c]} \s((y,c]) \m(dy),&\text{for } x \leq c,\end{cases}
\end{split}
\end{equation*}
where we identify $\s$ with the locally finite measure on $(J^\circ,\mathcal B(J^\circ))$ defined via $\s((x,y])\triangleq\s(y)-\s(x)$, for $x<y$ in $J^\circ$.
For \(b\in \partial J\), we write \(\u (b) \triangleq \lim_{J^\circ\ni x \to b} \u(x)\) and \(\v(b) \triangleq \lim_{J^\circ\ni x \to b} \v(x)\).
A boundary point \(b\in \partial J\) is called 
\begin{align*}
\emph{regular}&\quad\text{ if } \u(b) < \infty \text{ and } \v(b) < \infty,\\  
\emph{exit}&\quad\text{ if } \u(b) < \infty \text{ and } \v(b) = \infty,\\  
\emph{entrance}&\quad\text{ if } \u(b) = \infty \text{ and } \v(b) < \infty,\\  
\emph{natural}&\quad\text{ if } \u(b) = \infty \text{ and } \v(b) = \infty.
\end{align*}
These definitions are independent of the choice of the reference point \(c \in J^\circ\).
Regular and exit boundaries are called \emph{accessible}, and entrance and natural boundaries are called \emph{inaccessible}. As already indicated by their names, inaccessible boundaries are not in the state space \(J\), while accessible ones are, see
\cite[Proposition 16.43]{breiman1968probability}.\footnote{Concerning the boundary classification, our nomenclature can be shown to be the same as in \cite{breiman1968probability,KarlinTaylor1981}.
On the contrary, it is worth noting that \cite{BorodinSalminen2002} uses a slightly different terminology.
}

For $b\in\partial J$, we set $\s(b)\triangleq\lim_{J^\circ\ni x\to b}\s(x)$ $(\in[-\infty,\infty])$,
and, for a reference point $c\in J^\circ$, we also use the notation
\begin{equation}\label{eq:271022a0}
\ofr b, c \gsr  \triangleq \begin{cases} (b, c), & b < c,\\ (c, b), & c < b.\end{cases}
\end{equation}
Straightforward calculations show that, for $b\in\partial J$,
\begin{align}
\u(b)<\infty
\;\;\Longrightarrow&\;\;
|\s(b)|<\infty,
\label{eq:101022a1}\\
\v(b)<\infty
\;\;\Longrightarrow&\;\;
\m(\ofr b, c \gsr)<\infty,
\label{eq:101022a2}\\
b\text{ is regular}
\;\;\Longleftrightarrow&\;\;
|\s(b)|<\infty
\text{ and }
\m(\ofr b, c\gsr)<\infty,
\label{eq:101022a3}
\end{align}
where $c\in J^\circ$ is arbitrary
(as $\m$ is a locally finite measure on $(J^\circ,\mathcal B(J^\circ))$,
the finiteness of $\m(\ofr b, c \gsr)$ does not depend on $c\in J^\circ$).
For the sake of comparing the above statements with each other,
we emphasize that the converse implications in \eqref{eq:101022a1} and~\eqref{eq:101022a2} are false, whereas \eqref{eq:101022a3} is indeed an equivalence.
Sometimes the following characterization for $b\in\partial J$ to be accessible
(i.e., for $\u(b)<\infty$)
is convenient:
\begin{equation}\label{eq:160223a1}
b\text{ is accessible}
\;\;\Longleftrightarrow\;\;
|\s(b)|<\infty
\text{ and }
\int_{(\hspace{-0.06cm}( b,c)\hspace{-0.06cm})}
|\s(z)-\s(b)|\,\m(dz)<\infty
\end{equation}
for all (equivalently, for some) $c\in J^\circ$.
The characterization~\eqref{eq:160223a1} follows from straightforward calculations.

The behavior of the diffusion at exit, entrance and natural boundaries is fully specified by \(\s\) and \(\m\).
Regular boundaries are different in this regard.
To see this, consider Brownian motion with state space \([0, \infty)\) and absorption or reflection in the origin (\cite[Section~16.3]{breiman1968probability}). In both cases the speed measure coincides with the Lebesgue measure on \((0, \infty)\) (when we take the scale function $\s(x)=x$) and the origin is regular.
Hence, knowing the speed measure on \(J^\circ = (0, \infty)\) does not suffice to decide whether the origin is absorbing or reflecting. To fix this issue, it is convenient to extend the speed measure \(\m\) to a measure on \((J, \mathcal{B}(J))\).
In the following we explain how this can be done for the case \(J = [0, \infty)\) and \(\s(0) = 0\).
Define \(\s^* \colon \mathbb{R} \to \mathbb{R}\) by setting \(\s^* (x) \triangleq \s (x)\) and \(\s^*(- x) \triangleq - \s(x)\) for \(x \in \mathbb{R}_+\).
For \(I = [0, a)\) with \(a > 0\), define \(G^*_I\) as \(G_{(-a, a)}\) from \eqref{eq: def G} with \(\s\) replaced by~\(\s^*\), and set
\[
G^\circ_I (x, y) \triangleq G^*_I (x, y) + G^*_I (x, -y), \quad x, y \in J = \mathbb{R}_+.
\]
By \cite[Proposition VII.3.10]{RY}, it is possible to define \(\m(\{0\})\in[0,\infty]\) such that, for any interval \(I = [0, a)\) with \(a > 0\) and any Borel function \(f \colon \mathbb{R}_+ = J \to \mathbb{R}_+\) with \(f(0) > 0\), 
\begin{align} \label{eq: extended speed measure}
\E_x \Big[ \int_0^{\t_a} f(\X_s) ds \Big] = \int_I G^\circ_I (x, y) f(y) \m(dy), \quad x \in I.
\end{align}
Let us convince ourselves that \(\m(\{0\})\) distinguishes absorption and reflection. Taking \(f \equiv \1_{\{0\}}\) in \eqref{eq: extended speed measure} yields 
\[
\E_0 \Big[ \int_0^{\t_a} \1_{\{\X_s = 0\}} ds \Big] = 2\s(a) \m(\{0\}), \quad a > 0.
\]
This formula motivates the following definitions:
a regular boundary point \(b\) is called \emph{absorbing} if \(\m(\{b\}) = \infty\), \emph{slowly reflecting} if \(0 < \m(\{b\}) < \infty\), and \emph{instantaneously reflecting} if \(\m (\{b\}) = 0\).
In what follows, we call a boundary point simply ``reflecting'' if it is ``either instantaneously or slowly reflecting''.

We stress at this point that exit boundaries are also absorbing in the sense that such a boundary point cannot be left by the diffusion. However, in contrast to the regular case, the behavior of an exit boundary is fully characterized by \(\m\) on \((J^\circ, \mathcal{B}(J^\circ))\). To guarantee that \eqref{eq: extended speed measure} holds we use the convention that \(\m (\{b\}) \equiv \infty\) in case \(b\) is an exit boundary.

The scale function and the extended speed measure determine the law of the diffusion uniquely, see \cite[Corollary 16.73]{breiman1968probability}. Therefore, we call the pair \((\m, \s)\) the \emph{characteristics of the diffusion}. To avoid confusions, we stress that here \(\m\) is the \emph{extended} speed measure, i.e., it is defined on \((J, \mathcal{B}(J))\).

\begin{remark}\label{rem:170322a1}
We will also need a kind of converse to the discussion above:
given a function and a measure
\emph{with properties to be specified precisely in this remark}
they are the scale function and the speed measure of some diffusion.
To this end, we first recall from above that the restriction $\m|_{J^\circ}$ of the speed measure $\m$ to $J^\circ$ is necessarily a locally finite measure on $(J^\circ,\mathcal B(J^\circ))$. Moreover, by virtue of~\eqref{eq:090322a1}, the speed measure also satisfies
\begin{equation}\label{eq:090322a2}
\m([a,b])>0\quad\text{for all }a<b\text{ in }J^\circ.
\end{equation}
Conversely, given an \emph{open} interval $I\subset\mathbb R$, a continuous and strictly increasing function $\s^\circ\colon I\to\mathbb R$ and a locally finite measure $\m^\circ$ on $(I,\mathcal B(I))$ satisfying~\eqref{eq:090322a2}
with $I$ in place of $J^\circ$,
there exists a diffusion with state space $J$, scale function $\s$ and speed measure $\m$ such that $J^\circ=I$, $\s|_{J^\circ}=\s^\circ$ and $\m|_{J^\circ}=\m^\circ$ (see \cite[Theorem~33.9]{kallenberg}).\footnote{The law of such a diffusion can happen to be non-unique, as we can have different boundary behavior (instantaneous or slow reflection, absorption) at regular boundary points.
More precisely, as we have seen above, $\s^\circ$ and $\m^\circ$ alone determine which boundary points of $J^\circ$ are regular, exit or inaccessible.
Accessible boundary points need to be in $J$, so $J$ is uniquely determined by $J^\circ$, $\s^\circ$ and $\m^\circ$.
For an exit boundary point $b\in\partial J$, the only possible boundary behavior is absorption, and we need to set $\m(\{b\})=\infty$ according to the convention above.
For a regular boundary point $b\in\partial J$, however, we can choose any $\m(\{b\})\in[0,\infty]$, i.e., we can decide about the boundary behavior at a regular boundary in addition to the information carried in $\s^\circ$ and $\m^\circ$.}
\end{remark}

\subsection{Main Results} \label{sec: main}
We take two state spaces \(J\) and \(\tilde{J}\) such that \(J^\circ = \tilde{J}^\circ\).
Let \((J \ni x \mapsto \P_x)\) and \((\tilde{J} \ni x \mapsto \Q_x)\) be two regular diffusions with characteristics \((\m, \s)\) and \((\M, \S)\), respectively.
Our goal is to compute the separating time for \(\P_{x_0}\) and \(\Q_{x_0}\) and an arbitrary initial value \(x_0 \in J \cap \tilde{J}\). 

To give some guidance on the structure of this section, in the first part, only in terms of the diffusion characteristics \((\m, \s)\) and \((\M, \S)\), we introduce the concept of {\em separating points} for \(\P_{x_0}\) and \(\Q_{x_0}\) (these are certain points from $\on{cl}(J)$).
Our main Theorem~\ref{theo: main1} then shows that the separating time of \(\P_{x_0}\) and \(\Q_{x_0}\) is ``something like''
the hitting time of the set \(\A\) of separating points, where we need to carefully distinguish between the values \(\infty\) and \(\delta\).
A precise mathematical formulation is given below.

\begin{definition}[\emph{Non-separating} or \emph{good} interior point] \label{def: non-sep int}
	We say that a point \(x \in J^\circ (= \tilde{J}^\circ)\) is
\emph{non-separating} (or \emph{good})
	if there exists an open neighborhood \(U (x) \subset J^\circ\) of \(x\) such that 
		\begin{enumerate}[label=\textup{(\roman*)}]
		\item
		the differential quotient
\begin{equation}\label{eq:101224a1}
\Big(\frac{d^+ \s}{d \S} \Big) (z) \triangleq \lim_{h \searrow 0} \frac{\s (z + h) - \s (z)}{\S (z + h) - \S (z)}
\end{equation}
		exists for all \(z \in U (x)\) and is strictly positive and finite, i.e.,
		\(d^+\s/d\S\) is a function \(U(x) \to (0, \infty)\);
		
		\item
		there exists a Borel function \(\beta \colon U (x) \to \mathbb{R}\) such that 
\begin{equation}\label{eq:101224a2}
\int_{U(x)} \big(\beta (z) \big)^2\, \S (dz) < \infty 
\end{equation}
		and, for all \(y \in U(x)\),			
\begin{equation}\label{eq:050122a1}
\Big(\frac{d^+ \s}{d \S} \Big) (y) - \Big(\frac{d^+ \s}{d \S} \Big) (x) = \int_x^y \beta (z)\, \s (dz);
\end{equation}

				\item
				the differential quotient
\begin{equation}\label{eq:101224a3}
\Big(\frac{d \m}{d\M}\Big)(z) \triangleq \lim_{h \searrow 0} \frac{\m ( (z - h, z + h))}{\M ((z - h, z + h))}
\end{equation}
				exists for all \(z \in U(x)\) and
\begin{equation}\label{eq:111022a1}
\frac{d \m}{d \M} \frac{d^+ \s}{d \S} = 1 
\end{equation}
on \(U (x)\).
\end{enumerate}
\end{definition}

We stress that the integration in~\eqref{eq:050122a1} is indeed meant w.r.t.~$\s$. One could alternatively write the right-hand side of~\eqref{eq:050122a1} as $\int_x^y \bar\beta (z) \S (dz)$, where $\bar\beta=\beta\,d^+\s/d\S$, but the function $\beta$ (rather than $\bar\beta$) is explicitly used many times below, which is the reason for the chosen form of~\eqref{eq:050122a1}.

\begin{remark}\label{rem:260922a1}
Replacing the \emph{right} differential quotient \(d^+\s/d\S\) with the \emph{left} one \(d^-\s/d\S\) everywhere in Definition \ref{def: non-sep int} results in an equivalent definition.
Indeed, if \(x\in J^\circ\) is non-separating according to Definition \ref{def: non-sep int} (i.e., with \(d^+\s/d\S\)), then, due to~\eqref{eq:050122a1}, \(d^+\s/d\S\) is continuous on \(U(x)\) and hence, thanks to \cite[p.~204]{saks}
applied to the function $\s\circ\S^{-1}$ on $\S(U(x))$,
\(d^-\s (z)/d\S\) exists for all \(z\in U(x)\) and coincides with \(d^+\s/d\S\). A similar argumentation also applies to the converse.
\end{remark}

	\begin{intuition} \label{int: interior}
		Let us outline the main ideas that explain that (i)--(iii) from Definition~\ref{def: non-sep int} must hold in case \(\P_{x_0}\) and \(\Q_{x_0}\) are equivalent up to the exit time of a neighborhood of \(x_0\), emphasizing that the following discussion is by no means rigorous nor complete. We refer to Section~\ref{sec: pf} for all the fizzy details. 
		We restrict our attention to the case \(\S = \on{Id}\). In fact this simplification is without loss of generality, as the general case then can be obtained by Lemma~\ref{lem: diff homo}. Take a point \(x_0 \in J^\circ\), an open neighborhood \(V (x_0) \subset J^\circ\) of \(x_0\) and set \(\tau \triangleq \inf (t \geq 0 \colon \X_t \not \in V (x_0))\). We assume that \(\P_{x_0} \sim \Q_{x_0}\) on \(\mathcal{F}_{\tau}\).
		
		Let us try to understand (i) from Definition~\ref{def: non-sep int}. Since the process \(\s (\X_{\cdot \wedge \tau})\) is a \(\P_x\)-martingale (see Lemma~\ref{lem: scale fct}), \(\P_{x_0} \sim \Q_{x_0}\) on \(\mathcal{F}_\tau\) yields that \(\s (\X_{\cdot \wedge \tau})\) is a \(\Q_x\)-semimartingale. But, under \(\Q_{x_0}\), \(\X_{\cdot \wedge \tau}\) is a time-change of Brownian motion (because \(\S = \on{Id}\)) and, as the semimartingale property is invariant under changes of time (see Lemma~\ref{lem: change of time}), also \(\s (W_{\cdot \wedge \tau'})\) is a semimartingale, where \(W\) is a Brownian motion starting at \(x_0\) and \(\tau' \triangleq \inf (t \geq 0 \colon W_t \not \in V (x_0))\). Thanks to the seminal work \cite{CinJPrSha}, for a Borel function \(f \colon \mathbb{R} \to \mathbb{R}\), we know that the process \(f (W)\) is a continuous semimartingale if and only if \(f\) is the difference of two convex functions (a so-called dc function). In Theorem~\ref{lem: diff convex}, we prove a local version of this deep result, which allows us to conclude that \(\s\) is a dc function in a neighborhood of \(x_0\). For now, we assume that \(\s\) is a dc function on \(V (x_0)\) (possibly making $V(x_0)$ a bit smaller),
which, by \cite[Proposition~5.1]{CinJPrSha}, means that it is continuous with a right-continuous right-hand derivative whose second derivative measure has locally finite variation.
In particular, part~(i) follows (except the strict positivity, which we explain below).
		
		To understand (ii), let
		\(\{L^x_t (\X) \colon (t, x) \in \mathbb{R}_+ \times J^\circ\}\)
		be a jointly continuous modification of the
		semimartingale
		local time processes of \(\X\) under \(\Q_{x_0}\)
		(it exists for diffusions on natural scale, see Lemma~\ref{lem: occ formula diff}
		and Remark~\ref{rem:241124a1}~(a)).
		Similarly, let
		\(\{L^x_t (\s (\X)) \colon (t, x) \in \mathbb{R}_+ \times \s (J^\circ)\}\)
		be a jointly continuous modification of the semimartingale
		local time processes of \(\s(\X)\) under \(\P_{x_0}\)
		(which is a diffusion on natural scale, see Lemma~\ref{lem: diff homo}).
		Then, by Lemma~\ref{lem: occ smg}, we must have \(\P_{x_0}, \Q_{x_0}\)-a.s.
		\begin{align} \label{eq: understand (i), (iii)}
		L^{\s(x)}_t (\s (\X)) = \Big(\frac{d^+ \s}{dx} \Big) (x) L^x_t (\X),\quad
		x\in V(x_0), \ t \leq \tau.
		\end{align} 
		As both local times are continuous in the space variable, \(y \mapsto d^+ \s (y) /dx\) needs to be continuous on \(V (x_0)\). Furthermore, inherited from the respective property of the Brownian local time, one can prove that \(\P_{x_0}, \Q_{x_0}\)-a.s.  \(L^{\s(x)}_\tau (\s (\X)), L^x_{\tau} (\X) > 0\) on \(V (x_0)\) (see Lemma~\ref{lem: pos LT}). Hence, \(y \mapsto d^+ \s (y) /dx\) is a strictly positive continuous function on \(V (x_0)\). 
In fact, the above regularity of \(y \mapsto d^+ \s (y) /dx\) can be improved to absolute continuity. To understand this, using the It\^o--Tanaka formula (see Lemma~\ref{lem: occ smg}), one gets that \(\Q_{x_0}\)-a.s., for all \(t < \tau\),
		\begin{align} \label{eq: intu 1; ito}
		d \s (\X_t) = \Big( \frac{d^+ \s}{dx} \Big) (\X_t) \,  d \X_t + \tfrac{1}{2}\, d \, \int L^{x}_t (\X) \s'' (dx).
		\end{align}
		From now on, suppose that \(\X\) is a Brownian motion till \(\tau\) under \(\Q_{x_0}\), which can be made rigorous by a change of time. Then, Girsanov's theorem (see Lemma~\ref{lem: Girs}) yields that \(\Q_{x_0}\)-a.s.
		\[
		d \, \int L^x_t (\X) \,\s'' (dx) \ll dt \quad \text{on } [0, \tau].
		\] 
		But, as \(\Q_{x_0}\)-a.s., for all \(t  \in [0, \tau]\), 
		\(
		t = \int L_t^x (\X) \, dx,
		\) 
		by the occupation time formula (see Lemma~\ref{lem: occ smg}),
	one might find it convincing that
		\(\Q_{x_0}\)-a.s. 
		\[
		L^x_t (\X) \, \s'' (dx) \ll L^x_t (\X) \, dx, \quad t \in [0, \tau],
		\] 
		which yields \(\s''(dx) \ll dx\), as \(\Q_{x_0}\)-a.s. \(L^x_\tau (\X) > 0\) for all \(x \in V (x_0)\) (see Lemma~\ref{lem: pos LT}). 
		As a consequence, up to increasing affine transformations, we have 
		\[
		\s (x) = \int^x \exp \Big( \int^y \beta (z) dz \Big) \, dy, \quad x \in V (x_0),
		\] 
		which means that on \(V (x_0)\) the scale function has the form of a scale function of an It\^o diffusion (which is given through an SDE).
In particular, we get~	\eqref{eq:050122a1} with (a priori) some integrable $\beta$.
		The final part of~(ii) deals with the square-integrability of \(\beta\), i.e.,
\(\int_{V (x_0)} \beta (z)^2 dz < \infty\),
where we recall that \(\S = \on{Id}\). 
		To understand this condition, notice that \eqref{eq: intu 1; ito} reformulates to 
		\begin{align*}
		d \s (\X_t) &= \s' (\X_t)  d \X_t + \tfrac{1}{2} \, d \, \int L^{x}_t (\X) \s' (x) \beta (x) dx
		\\&= \s' (\X_t) d\X_t + \tfrac{1}{2}\, \s' (\X_t) \beta (\X_t) d\langle \X, \X\rangle_t,
		\end{align*} 
		where we use the occupation time formula (see Lemma~\ref{lem: occ smg}) and \(\s'\) denotes the derivative of \(\s\).
		Now, Girsanov's theorem (see Lemma~\ref{lem: Girs}) yields that \(\Q_{x_0}\)-a.s.
\(\int_0^{\tau} \beta (\X_s)^2\,d\langle\X,\X\rangle_s < \infty.\)
		Because \(\Q_{x_0}\)-a.s.
		\[
		\int_0^{\tau} \beta (\X_s)^2\,d\langle\X,\X\rangle_s = \int_{V (x_0)} \beta (x)^2 L^x_\tau (\X) dx, \quad L^x_\tau (\X) > 0, \ x \in V (x_0), 
		\] 
		we conclude that
\(\int_{V (x_0)} \beta (z)^2 dz < \infty\) (possibly making $V(X_0)$ a bit smaller), as needed.
		
		Finally, we comment on part (iii). 
		Thanks to the occupation time formula for diffusions (see Lemma~\ref{lem: occ formula diff}),
		for $z\in J^\circ$,
		we have \(\P_{x_0}\)-a.s. 
		\[
		\frac{\int_0^\tau \1_{\{ z - h\, < \, \X_s \, <\, z + h\}} ds}{\m ( (z - h, z + h))} \to L^{\s (z)}_\tau (\s (\X)), \quad h \searrow 0, 
		\] 
		and \(\Q_{x_0}\)-a.s. 
		\[
		\frac{\int_0^\tau \1_{\{ z - h \, <\, \X_s\, <\, z + h\}} ds}{\M ( (z - h, z + h))} \to L^z_\tau (\X), \quad h \searrow 0.
		\] 
		Hence, as \(\P_{x_0}\)-a.s. \(L^{\s (x)}_\tau (\s (\X)) > 0\) on $V(x_0)$, using also \(\P_{x_0} \sim \Q_{x_0}\) on \(\mathcal{F}_\tau\) and \eqref{eq: understand (i), (iii)}, we get that \(\P_{x_0}\)-a.s. 
		\[
		d \m / d \M (x) = L^x_\tau (\X) / L^{\s (x)}_\tau (\s (\X)) = 1 / ( d^+ \s (x) / dx),
		\quad x\in V(x_0),
		\]  
		which yields~(iii).
	\end{intuition}

The next result provides a measure-theoretical view on Definition~\ref{def: non-sep int} and a measure-theoretical meaning of the quantities appearing there.
We recall that
$\s$ and $\S$ are identified with locally finite measures on $(J^\circ,\mathcal B(J^\circ))$ defined via $\s((x,y])\triangleq\s(y)-\s(x)$ and $\S((x,y])\triangleq\S(y)-\S(x)$, where $x<y$ are in $J^\circ$.

\begin{lemma}[Equivalent definition of a non-separating interior point]\label{lem:130222a1}
\quad 
\begin{enumerate}
\item[\textup{(i)}]
A point $x\in J^\circ (=\tilde J^\circ)$ is non-separating if and only if there exists an open neighborhood $U(x)\subset J^\circ$ of $x$ such that
\begin{enumerate}[label=\textup{(\alph*)}]
\item
$\s\ll\S$ on $\mathcal B(U(x))$ and
$\m\ll\M$ on $\mathcal B(U(x))$;

\item
there exists a version of the Radon--Nikodym derivative $\partial\s/\partial\S$ and a Borel function \(\beta \colon U (x) \to \mathbb{R}\) such that
\begin{gather*}
\Big(\frac{\partial\s}{\partial\S} \Big)(y)
>0
\quad\text{for all }y\in U(x),
\\[1mm]
\Big(\frac{\partial\s}{\partial\S}\Big)(y)-\Big(\frac{\partial\s}{\partial\S}\Big)(x)
=\int_x^y \beta(z)\,\s(dz)
\quad\text{for all }y\in U(x),
\\[1mm]
\int_{U(x)}\big(\beta(z)\big)^2\,\S(dz)
<\infty;
\end{gather*}

\item
for the version of $\partial\s/\partial\S$ described in~(b) we have
$$
\frac{\partial\m}{\partial\M}
\frac{\partial\s}{\partial\S}
=1\;\;\M\text{-a.e. on }U(x),
$$
where $\partial\m/\partial\M$ denotes (any version of) the Radon--Nikodym derivative of~$\m$ w.r.t.~$\M$.
\end{enumerate}
\item[\textup{(ii)}]
Furthermore, let $x\in J^\circ (=\tilde J^\circ)$ be a non-separating point.
Then, the differential quotient $d^+\s/d\S$ from Definition~\ref{def: non-sep int} equals the version $\partial\s/\partial\S$ of the Radon--Nikodym derivative described in~(b)
and the differential quotient $d\m/d\M$ from Definition~\ref{def: non-sep int} is a version of the Radon--Nikodym derivative $\partial\m/\partial\M$.
Finally, the function $\beta$ from part (ii)~of Definition~\ref{def: non-sep int} satisfies 
\begin{align} \label{eq: beta formula}
	\beta (z) = \lim_{h \searrow 0} \frac{d^+ \s (z + h)/d \S - d^+ \s (z - h) / d \S}{\s (z + h) -\s (z - h)}
\end{align} 
for \(\s\), \(\S\)-a.a. \(z \in U (x)\). 
\end{enumerate}
\end{lemma}

\begin{remark}\label{rem:130222a1}
As an immediate consequence of Lemma~\ref{lem:130222a1}, we make the following note.
In a neighborhood of a non-separating point we must have $\s\sim\S$ and $\m\sim\M$ with continuous (!) Radon--Nikodym derivatives.
\end{remark}

\begin{proof}[Proof of Lemma~\ref{lem:130222a1}]
The facts that (a)--(c) of Lemma~\ref{lem:130222a1} imply (i)--(iii) of Definition~\ref{def: non-sep int}
and that the differential quotients from Definition~\ref{def: non-sep int} are versions of the Radon-Nikodym derivatives as described in Lemma~\ref{lem:130222a1}~(ii)
are straightforward
due to the continuity of the version $\partial\s/\partial\S$ of the Radon--Nikodym derivative described in~(b) and the continuity of the version $\partial\m/\partial\M\triangleq1/(\partial\s/\partial\S)$ of the Radon--Nikodym derivative of $\m$ w.r.t. $\M$.
	We now prove that the formula \eqref{eq: beta formula} holds for \(\s\)-a.a. \(z \in U (x)\). Let \(\nu\) be the signed measure defined by 
	\[
	\nu ( (a, b] ) \triangleq d^+ \s / d \S (b) - d^+ \s / d \S (a), \quad a < b, \, a, b \in U (x). 
	\]
	Then, we get from \eqref{eq:050122a1} and a monotone class argument that 
	\[
	\nu (B) = \int_B \beta (y) \s (dy), \quad B \in \mathcal{B} (U (x)). 
	\] 
	As a consequence, \(\nu \ll \s\) on \(\mathcal{B} (U (x))\) with Radon--Nikodym density \(\partial \nu / \partial \s = \beta\). Finally, Theorem~\ref{th:260922a1} yields that \eqref{eq: beta formula} holds for \(\s\)-a.a. \(z \in U (x)\).
Finally, as $\s\sim\S$ on $\mathcal B(U(x))$, \eqref{eq: beta formula}~also holds $\S$-a.e. on $U(x)$.

It remains to prove that (i)--(iii) of Definition~\ref{def: non-sep int} imply (a)--(c) of Lemma~\ref{lem:130222a1}.
Let $x\in J^\circ$ and suppose that there exists an open neighborhood $U(x)\subset J^\circ$ of $x$ such that (i)--(iii) of Definition~\ref{def: non-sep int} hold true.
Remark~\ref{rem:260922a1} implies that
$$
\Big(\frac{d^+ \s}{d \S} \Big) (z) = \lim_{h \searrow 0} \frac{\s (z + h) - \s (z-h)}{\S (z + h) - \S (z-h)}
$$
for all $z\in U(x)$.
Using the notation from Section~\ref{subsec:DiffMeas}, this means that $d^+\s/d\S=D^{\on{sym}}_{\S}(\s)$ on $U(x)$.
By Theorem~\ref{th:260922a1}, $d^+\s/d\S$ is a version of the generalized Radon--Nikodym derivative $\partial\s/\partial\S$
(the notion is defined in Section~\ref{subsec:GenDen}).
As $d^+\s/d\S$ is $(0,\infty)$-valued, Lemma~\ref{lem:260922a1} implies that $\s\sim\S$ on $\mathcal B(U(x))$
(alternatively, this follows from Corollaries \ref{cor:260922a1} and~\ref{cor:260922a2});
in particular, the generalized Radon--Nikodym derivative $\partial\s/\partial\S$ is the standard one.
In a similar way, (iii)~of Definition~\ref{def: non-sep int} and Theorem~\ref{th:260922a1} yield that $\m\sim\M$ on $\mathcal B(U(x))$ and that $d\m/d\M$ is a version of the Radon--Nikodym derivative $\partial\m/\partial\M$.
The remaining claims are now obvious.
\end{proof}

\begin{definition}[\emph{Half-good} boundary point]\label{def:130222a1}
Take \(b \in \partial J (= \partial\tilde{J})\). We say that \(b\) is \emph{half-good} if there exists a non-empty open interval \(B \subsetneq J^\circ (= \tilde{J}^\circ)\) with \(b\) as endpoint such that
\begin{enumerate}[label=\textup{(\roman*)}]
\item
all points in \(B\) are good in the sense of Definition~\ref{def: non-sep int};

\item
\(\S(b) \triangleq \lim_{J^\circ \ni x \to b} \S(x)\in\mathbb R\) and
\begin{equation}\label{eq:130222a1}
\int_B \big|\S (z) - \S (b) \big| \big(\beta (z) \big)^2  \, \S (dz) < \infty,
\end{equation}
where \(\beta\) is as in Definition~\ref{def: non-sep int}, which can be defined unambiguously by the formula~\eqref{eq: beta formula}.
\end{enumerate}
\end{definition}

\begin{intuition}
	Part~(ii) from Definition~\ref{def:130222a1} encodes that, for every \(x \in B\), \(\P_x, \Q_x\)-a.s. 
	\[
	\lim_{t \nearrow T_b} \frac{d \P_{x}}{d \Q_{x}} \Big|_{\mathcal{F}_t} > 0
	\] 
on the trajectories such that $\X_t\to b$, $t\nearrow T_b$, and $\X_\cdot\in B$ on $[0,T_b)$.
	It is known that this property is related to the non-explosion of a non-negative additive functional 
	(cf., e.g., \cite{criens21,criensglau,desmettre,MU12}). By similar ideas as outlined in Intuition~\ref{int: interior}, again considering the case \(\S = \on{Id}\), this question can be reduced to the problem whether 
	\[
	\int_0^{T_b (W)} \beta (W_s)^2 \,ds,\quad W = \text{Brownian motion},  
	\]
	is almost surely finite. Such questions were for example investigated in \cite{MU12c}, leading to the integral condition from (ii) in Definition~\ref{def:130222a1}. Of course, this provides only a rough idea where \eqref{eq:130222a1} comes from, while a rigorous proof needs a careful investigation.
\end{intuition}

\begin{lemma}\label{lem:130222a2}
Assume that \(b \in \partial J (= \partial\tilde{J})\) is half-good in the sense of Definition~\ref{def:130222a1}. Then, one of the following alternatives holds:
\begin{enumerate}[label=\textup{(\alph*)}]
\item
$b\in J$ and $b\in\tilde J$,
i.e., $b$ is an accessible boundary point both for \((x \mapsto \P_x)\) and for \((x \mapsto \Q_x)\);

\item
$b\notin J$ and $b\notin\tilde J$,
i.e., $b$ is an inaccessible boundary point both for \((x \mapsto \P_x)\) and for \((x \mapsto \Q_x)\).
\end{enumerate}
\end{lemma}

This lemma is proved in Section~\ref{sec: pf of lem:130222a2}.
At this point, Lemma~\ref{lem:130222a2} justifies the equivalence
between $b\in\tilde J$ and $b\in J$
in the following definition. For \(b \in \partial J ( = \partial \tilde{J})\) and an arbitrary reference point \(c \in J^\circ (= \tilde{J}^\circ)\), we set 
\begin{equation}\label{eq:271022a1}
\of b, c \gsr \triangleq \begin{cases} [b, c), & b < c,\\ (c, b], & c < b.\end{cases}
\end{equation}

\begin{definition}[\emph{Non-separating} or \emph{good} boundary point]\label{def: non-sep bound}
Take \(b \in \partial J (= \partial\tilde{J})\).
We say that \(b\) is \emph{non-separating} (or \emph{good}) if
\begin{enumerate}[label=\textup{(\roman*)}]
\item
$b$ is half-good in the sense of Definition~\ref{def:130222a1};

\item
if \(b \in \tilde{J}\) (equivalently, $b\in J$), then
either it holds that $\m(\{b\})=\M(\{b\})=\infty$ or it holds that $\m(\{b\}) < \infty$, $\M(\{b\})<\infty$;

\item
if \(b \in \tilde{J}\) (equivalently, $b\in J$) and $\m(\{b\}) < \infty$, $\M(\{b\})<\infty$,
then also the following conditions hold:
\begin{enumerate}
\item[\textup{(a)}] the differential quotient
\begin{align}
\Big(\frac{d^+\s}{d\S}\Big)(b)
&\triangleq
\lim_{J^\circ\ni c\to b}\frac{\s (b) - \s (c)}{\S(b) - \S(c)}, 
\label{eq:101022c1}
\end{align}
exists as a strictly positive and finite number;

\item[\textup{(b)}]
with $B$ as in Definition~\ref{def:130222a1}, there exists a Borel function \(\beta \colon B\to \mathbb{R}\) such that
\begin{equation}\label{eq:271022a2}
\int_{B} \big( \beta (z) \big)^2 \, \S (dz) < \infty,
\end{equation}
and, for all \(y \in B\),
\begin{equation}\label{eq:271022a3}
\Big( \frac{d^+ \s}{d \S} \Big) (b \vee y) - \Big( \frac{d^+ \s}{d \S} \Big) (b \wedge y) = \int_{b\wedge y}^{ b\vee y} \beta (z) \,  \s(dz);
\end{equation}

\item[\textup{(c)}]
the differential quotient \begin{align}
\Big(\frac{d\m}{d\M}\Big)(b)
&\triangleq
\lim_{J^\circ\ni c\to b}\frac{\m(\of b, c \gsr )}{\M(\of b, c \gsr )}
\label{eq:101022c2}
\end{align}
exists as a strictly positive and finite number and 
\begin{align}
\Big(\frac{d\m}{d\M}\Big)(b)
\Big(\frac{d^+\s}{d\S}\Big)(b)
=1.
\label{eq:101022c3}
\end{align}
\end{enumerate}
\end{enumerate}
\end{definition}

We notice that, although the superscript ``$+$'' in the notation $d^+\s/d\S$ does not really make sense when \(b\) is a right boundary point
	in~\eqref{eq:101022c1}, we keep it for conformity with Definition~\ref{def: non-sep int}
(compare \eqref{eq:111022a1} with~\eqref{eq:101022c3}).

\begin{discussion}
What is encoded in Definition~\ref{def: non-sep bound} is worth a discussion.
Let $b\in\partial J$ be a non-separating boundary point.
Then, the following statements hold true,
where the phrase ``for both diffusions'' always means
``both for $(x\mapsto\P_x)$ and for $(x\mapsto\Q_x)$''.

(i) As already mentioned, $b$ is either accessible for both diffusions or inaccessible for both (Lemma~\ref{lem:130222a2}).

(ii) Let $b$ be accessible for both diffusions.
Then, by~\eqref{eq:101022a1}, $|\s(b)|<\infty$ and $|\S(b)|<\infty$.
Furthermore, $b$ is either absorbing for both diffusions ($\m(\{b\})=\M(\{b\})=\infty$) or reflecting for both ones ($\m(\{b\}) < \infty$ and $\M(\{b\})<\infty$).

(iiia) Let $b$ be reflecting for both diffusions.
Then, by \eqref{eq:101022a2} or~\eqref{eq:101022a3},
$\m(\of b, c \gsr )<\infty$ and $\M(\of b, c \gsr )<\infty$ for all $c\in J^\circ$
(recall the convention in the second paragraph before Remark~\ref{rem:170322a1}
that $\m(\{b\})=\infty$ whenever $b$ is an exit boundary,
hence $\m(\{b\})<\infty$ necessarily means that $b$ is a regular boundary).
Together with $|\s(b)|<\infty$ and $|\S(b)|<\infty$ discussed above,
this means that the quotients
$(\s (b) - \s (c))/ (\S (b) - \S(c))$
and
$\m(\of b, c \gsr )/\M(\of b, c \gsr )$
on the right-hand sides
of \eqref{eq:101022c1} and~\eqref{eq:101022c2}
are well-defined.

(iiib) Let $b$ be reflecting for both diffusions.
As, for a finite measure, the $L^2$-space is included in the $L^1$-space, we have $\beta\in L^1(B,\S(dz))$ due to $|\S(b)|<\infty$ together with~\eqref{eq:271022a2}.
Hence, the right-hand side in \eqref{eq:271022a3} is well-defined.
Moreover, due to \eqref{eq:111022a1}, \eqref{eq:271022a3} and~\eqref{eq:101022c3},
$d^+\s/d\S$ and $d\m/d\M$ are continuous on~$B\cup\{b\}$.

(iiic) Let $b$ be reflecting for both diffusions.
Then, it is either instantaneously reflecting for both diffusions ($\m(\{b\})=\M(\{b\})=0$) or slowly reflecting for both ($\m(\{b\}),\M(\{b\})\in(0,\infty)$).
Indeed, if $b$ is instantaneously reflecting for one of the diffusions and slowly reflecting for the other, then the limit in~\eqref{eq:101022c2} is $0$ or $\infty$,
which contradicts to the fact that $b$ is a non-separating boundary point.

(iv) Let $b$ be slowly reflecting for both diffusions.
Then, the limit in~\eqref{eq:101022c2} is nothing else but the quotient
$\m(\{b\})/\M(\{b\})$.
\end{discussion}

\begin{intuition}
	In addition to ``half-goodness'' from Definition~\ref{def:130222a1}, Definition~\ref{def: non-sep bound} imposes additional constraints if $b\in\partial J$ is reflecting at least for one of the diffusions (and, for $b$ to be non-separating, necessarily, for both diffusions).
	The structure of these additional parts appear to be very similar to the definition of an interior non-separating point (see Definition~\ref{def: non-sep int}). This is no coincidence, because reflecting boundary points are related to interior points of a symmetrized diffusion with a larger state space. To see the idea, for a standard Brownian motion \(W\) with \(W_0 = 0\), the formula \(|W|\) defines a Brownian motion with reflection at the origin. So to say, standard Brownian motion is the symmetrization of a Brownian motion with reflection at the origin.
	We refer to Lemmata~\ref{lem: refl} and~\ref{lem:200223a3} for the general picture of the symmetrization procedure. Finally, the additional constraints in Definition~\ref{def: non-sep bound} are nothing else but the requirement that $b$ is a non-separating interior point for the symmetrized diffusion. Let us emphasize that it is by no means obvious that symmetrization is useful for the computation of the separating time (to get an idea of the problem, notice that, in the above example with $|W|$ and $W$, the filtration of $|W|$ is strictly smaller than that of~$W$).
This requires a careful investigation.
\end{intuition}

\begin{definition}[\emph{Separating} or \emph{bad} point]\label{def:101224a1}
We say that a point $x\in\on{cl}(J)(=\on{cl}(\tilde J))$ is \emph{separating} (or \emph{bad})
if it is not non-separating.
We denote the set of all separating points in $\on{cl}(J)$ by $\A$.
Notice that \(\A\) is closed in \(\on{cl}(J)\).
\end{definition}

We set
$$
l \triangleq \inf J
\qquad\text{and}\qquad
r \triangleq \sup J.
$$
Let $\Delta$ be a point outside $[-\infty,\infty]$.
We denote by \(\alpha\) the separating point which is closest to \(x_0\) from the left in the following sense:
\[
\alpha \triangleq \begin{cases} 
\sup \big([l, x_0] \cap \A\big), & [l, x_0] \cap \A \not = \emptyset,\\
\Delta, & [l, x_0] \cap \A = \emptyset.
\end{cases}
\]
Next, we define an extended stopping time $U$,
which is ``a variant of the hitting time'' of the point $\alpha$. 
\begin{enumerate}
\item[-]
If $\alpha=\Delta$, then we set 
\(U\triangleq\delta.\)

\item[-]
If $\alpha=x_0$, then we set \(U\triangleq0.\)

\item[-]
Otherwise (i.e., in case $\on{cl}(J)\ni\alpha<x_0$) we set
\begin{equation}\label{eq:160223a2}
U \triangleq\begin{cases}
T_\alpha,&\liminf_{t\nearrow T_\alpha}\X_t=\alpha,\\
\delta,&\liminf_{t\nearrow T_\alpha}\X_t>\alpha.
\end{cases}
\end{equation}
\end{enumerate}
The structure of $U$ in~\eqref{eq:160223a2} is worth a discussion.
On the event $\{T_\alpha<\infty\}$, the right-hand side of~\eqref{eq:160223a2} equals $T_\alpha$.
However, on $\{T_\alpha=\infty\}$, the right-hand side of~\eqref{eq:160223a2} is either $\infty$ or $\delta$ depending on whether the point $\alpha$ is hit asymptotically or not.

In a symmetric way, we define \(V\) to be ``a variant of the hitting time'' of the separating point which is closest to \(x_0\) from the right. To avoid confusions, we give a precise definition. We set
\[
\gamma \triangleq \begin{cases} 
\inf \big([x_0, r] \cap \A\big), & [x_0, r] \cap \A \not = \emptyset,\\
\Delta, & [x_0, r] \cap \A = \emptyset
\end{cases}
\]
and proceed as above:
\begin{enumerate}
\item[-]
If $\gamma=\Delta$, we set \(V\triangleq\delta.\)

\item[-]
If $\gamma=x_0$, we set \(V\triangleq0.\)

\item[-]
Otherwise (i.e., in case $\on{cl}(J)\ni\gamma>x_0$), we set
\begin{equation}\label{eq:111224a1}
V \triangleq\begin{cases}
T_\gamma,&\limsup_{t\nearrow T_\gamma}\X_t=\gamma,\\
\delta,&\limsup_{t\nearrow T_\gamma}\X_t<\gamma.
\end{cases}
\end{equation}
\end{enumerate}
Of course, the structure of $V$ in~\eqref{eq:111224a1} can be discussed in the same way as the structure of $U$ in~\eqref{eq:160223a2} is discussed above.

We also introduce the deterministic time
\[
R \triangleq \begin{cases}
\infty, & \alpha = \gamma = \Delta,
\text{ \(l\) and \(r\) are reflecting for one (equivalently, for both) of the diffusions},
\\
\delta, & \text{otherwise}.
\end{cases}
\]
Finally, we are in a position to present our main result. Intuitively speaking, it shows that the separating time of two non-identical diffusions is given by ``something like'' the first time the coordinate process hits a separating point (with a careful distinction between $\infty$ and $\delta$ on the trajectories that do not hit separating points in finite time).

\begin{theorem} \label{theo: main1}
	Let \(x_0 \in J \cap \tilde{J}\) and let \(S\) be the separating time for \(\P_{x_0}\) and \(\Q_{x_0}\). 
	\begin{enumerate}[label=\textup{(\roman*)}]
		\item
		If \(\P_{x_0} = \Q_{x_0}\), then \(\P_{x_0}, \Q_{x_0}\)-a.s. \(S = \delta\).
		\item
		If \(\P_{x_0} \not = \Q_{x_0}\), then \(\P_{x_0}, \Q_{x_0}\)-a.s. \(S = U \wedge V \wedge R\).
	\end{enumerate}
\end{theorem}

\begin{remark}\label{rem:191124a0}
(a)
We notice that the (trivial) case $\P_{x_0}=\Q_{x_0}$ needs to be treated separately in Theorem~\ref{theo: main1}.
Technically, the reason is that a boundary point $b$ can be separating also in the case $\P_{x_0}=\Q_{x_0}$.
For example, in case both $\P_{x_0}$ and $\Q_{x_0}$ are the standard Wiener measure, trivially \(\P_{x_0}, \Q_{x_0}\)-a.s. \(S = \delta\), but \(\P_{x_0}, \Q_{x_0}\)-a.s. \(U \wedge V \wedge R = \infty\), because \(\pm\infty\) are separating points and \(\P_{x_0}, \Q_{x_0}\)-a.s. \(\limsup_{t \to \infty} \X_t = - \liminf_{t \to \infty} \X_t = \infty\) by standard properties of Brownian motion.

\smallskip
(b)
It is possible to formulate part~(ii) from Theorem~\ref{theo: main1} without the deterministic time \(R\): 
\begin{enumerate}
	\item[\textup{(ii.1)}] If all points of $\on{cl}(J)$ are non-separating and the boundary points $l$ and $r$ are reflecting for one (equivalently, for both) of the diffusions, then $\P_{x_0},\Q_{x_0}$-a.s. $S=\infty$.
	\item[\textup{(ii.2)}]
	Otherwise, $\P_{x_0},\Q_{x_0}$-a.s. $S=U\wedge V$.
\end{enumerate}
The situation from (ii.1) needs to be treated separately, because diffusions with two reflecting boundary points are necessarily recurrent, which entails that they cannot be equivalent on the infinite time horizon. This is the only case where $\P_{x_0}\ne\Q_{x_0}$ but the separating time still cannot be captured via the set $\A$ of separating points alone. In Example~\ref{ex: BM with reflec bd} and Discussion~\ref{disc:180223a1} we discuss this issue in more detail.
\end{remark}

\begin{discussion}\label{disc:190223a1}
In order to apply Theorem~\ref{theo: main1} in specific situations, we need to compute the set $\A$ of separating points in \(\on{cl}(J) (=\on{cl}(\tilde J))\).
Given the above definitions this might be a very computational task.
Fortunately, there are many interesting interdependencies between the notions, which allow to reduce the computations in many specific situations considerably.
In the following, we collect several helpful observations.

(i) If a boundary point $b \in \partial J (= \partial\tilde{J})$ is accessible for one of the diffusions but inaccessible for the other one, then $b\in \A$ (that is, in such a case we do not need to verify~\eqref{eq:130222a1}).
This follows directly from Lemma~\ref{lem:130222a2}.

(ii) The roles of the diffusions $(x\mapsto\P_x)$ and $(x\mapsto\Q_x)$ in each of the three building blocks in the definition of a non-separating point, namely, Definitions \ref{def: non-sep int}, \ref{def:130222a1} and~\ref{def: non-sep bound}, are symmetric.
That is, in any of these definitions we can interchange the roles between $(\m,\s)$ and $(\M,\S)$.
To avoid confusions, we sketch this in more detail.
\begin{enumerate}
\item[-]
Instead of working with $d^+\s/d\S$ and $d\m/d\M$ one may use $d^+\S/d\s$ and $d\M/d\m$ to define a good (non-separating) interior point.
Instead of $\beta$ we then get another function $\tilde\beta$.
The relation between $\beta$ and $\tilde\beta$ is
\begin{equation}\label{eq:130222a2}
\tilde\beta=-\beta\hspace{0.05cm}\frac{d^+\S}{d\s}.
\end{equation}

\item[-]
To define a half-good and a good (non-separating) boundary point exchange the roles of $(\m,\s)$ and $(\M,\S)$ everywhere and do not forget to replace $\beta$ with $\tilde\beta$ in \eqref{eq:130222a1}, \eqref{eq:271022a2} and~\eqref{eq:271022a3}.

\item[-]
It is possible to exchange the roles in {\em some} but {\em not in all} of the three definitions.
For instance, one might inspect the goodness of the interior points exactly as in Definition~\ref{def: non-sep int} (which also provides the function $\beta$ on the way)
and inspect the goodness of the boundary points using the definitions for the interchanged diffusions (in which case use $\tilde\beta$ from~\eqref{eq:130222a2}).
\end{enumerate}
The possibility to interchange the roles of $(\m,\s)$ and $(\M,\S)$ follows from Theorem~\ref{theo: main1} together with the symmetry in the notion of the separating time: the ($\P_{x_0},\Q_{x_0}$-a.s. unique) separating time for $\P_{x_0}$ and $\Q_{x_0}$ is the same as the separating time for $\Q_{x_0}$ and~$\P_{x_0}$.
The formula~\eqref{eq:130222a2} is a straightforward calculation.

(iii) A boundary point $b$ with either $|\S(b)|=\infty$ or $|\s(b)|=\infty$ is automatically separating, i.e., $b\in \A$.
In the case $|\S(b)|=\infty$ this is seen directly from the definitions.
In the case $|\s(b)|=\infty$ this follows from the previous point in this discussion, i.e., the symmetric roles of $(x\mapsto\P_x)$ and $(x\mapsto\Q_x)$.
\end{discussion}

By virtue of Proposition~\ref{prop: AC Sing}, Theorem~\ref{theo: main1} yields a variety of corollaries concerning absolute continuity and singularity of \(\P_{x_0}\) and \(\Q_{x_0}\).

\begin{corollary} \label{coro: local absolute continuity}
	Let \(\P_{x_0} \not = \Q_{x_0}\) and \(x_0 \in J\cap\tilde J\).
	Then, the following are equivalent:
	\begin{enumerate}
		\item[\textup{(i)}] \(\P_{x_0} \ll \Q_{x_0}\) on \(\mathcal{F}_t\) for some \(t > 0\).
		\item[\textup{(ii)}] \(\P_{x_0} \ll \Q_{x_0}\) on \(\mathcal{F}_t\) for all \(t > 0\), i.e., \(\P_{x_0} \ll_{\textup{loc}} \Q_{x_0}\).
		\item[\textup{(iii)}] All points in \(J\) are non-separating
		(in other words, all points in \(J^\circ  (= \tilde{J}^\circ)\) are non-separating and boundary points that are accessible for \((x \mapsto \P_x)\) are non-separating).
	\end{enumerate}
\end{corollary}

\begin{remark}\label{rem:191124a1}
In case $x_0$ is \emph{not} an absorbing boundary point at least for one of the diffusions, the assumption $\P_{x_0}\ne\Q_{x_0}$ can be removed from Corollary~\ref{coro: local absolute continuity}.
To see this, consider the case where \(x_0\) is not absorbing and $\P_{x_0}=\Q_{x_0}$. Then, all conditions (i), (ii) and~(iii) in Corollary~\ref{coro: local absolute continuity} are satisfied.
While this is clear for (i) and~(ii), for (iii)  this follows from Lemma~\ref{lem:121224a1}.
As a consequence, the assumption $\P_{x_0}\ne\Q_{x_0}$ in Corollary~\ref{coro: local absolute continuity} can be dropped whenever $x_0\in J^\circ$.
\end{remark}

\begin{corollary}\label{cor:110222a1}
	Let \(\P_{x_0} \not = \Q_{x_0}\) and \(x_0 \in J\cap\tilde J\). 
	Then, \(\P_{x_0} \ll \Q_{x_0}\) if and only if
\begin{enumerate}
\item[\textup{(a)}]	
all points in \(J^\circ\) are non-separating,

\item[\textup{(b)}]
each boundary point \(b\) of \(J\) satisfies one of the following:
	\begin{enumerate}
		\item[\textup{(b.i)}]
		\(b\) is non-separating,
		\item[\textup{(b.ii)}]
		\(|\s(b)| = \infty\) and the other boundary point \(b^*\) is non-separating,
	\end{enumerate}

\item[\textup{(c)}]
and, in case both boundary points are non-separating, at least one of the boundary points is not reflecting (for one, equivalently for both, of the diffusions).
\end{enumerate}
\end{corollary}

It is worth noting that, for a boundary point $b$, (b.i) and~(b.ii) above exclude each other because $|\s(b)|=\infty$ implies that $b$ is separating
(recall Discussion~\ref{disc:190223a1}~(iii)).

\begin{remark}\label{rem:191124a2}
Contrary to Corollary~\ref{coro: local absolute continuity} and Remark~\ref{rem:191124a1},
in Corollary~\ref{cor:110222a1} the assumption $\P_{x_0}\ne\Q_{x_0}$ cannot be dropped even in the case $x_0\in J^\circ$.
The reason is discussed in Remark~\ref{rem:191124a0}~(a).
For instance, if $\P_{x_0}=\Q_{x_0}$ is the standard Wiener measure, then both $\infty$ and $-\infty$ are separating boundary points, and hence (b) in Corollary~\ref{cor:110222a1} is not satisfied.
\end{remark}

\begin{proof}[Proof of Corollaries \ref{coro: local absolute continuity} and \ref{cor:110222a1}]
The (trivial) case that $x_0$ is an absorbing boundary for both diffusions is excluded by the assumption \(\P_{x_0} \not = \Q_{x_0}\).
Now, the claims of both corollaries follow from
Theorem~\ref{theo: main1},
Proposition~\ref{prop: AC Sing}
and Lemmata \ref{lem: diff hit points fast}
and~\ref{lem:200223a2}.
\end{proof}

\begin{corollary}
	Let \(x_0 \in J\cap\tilde J\). We have
	\(\P_{x_0} \perp \Q_{x_0}\) on \(\mathcal{F}_0\) if and only if \(x_0 \in \A\).
\end{corollary}

\begin{corollary}
	Take \(x_0 \in J\cap\tilde J\) and suppose that either \(\s = \S\) on \(J^\circ\) or \(\m = \M\) on \(\mathcal{B}(J^\circ)\). The following are equivalent:
	\begin{enumerate}
		\item[\textup{(i)}] \(\P_{x_0} \ll \Q_{x_0}\) on \(\mathcal{F}_t\) for some \(t > 0\).
		\item[\textup{(ii)}] \(\P_{x_0} = \Q_{x_0}\).
		\end{enumerate}
\end{corollary}

\begin{proof}
	Of course, (ii) \(\Rightarrow\) (i) is trivial. Assume that (i) holds and, for contradiction, further assume that \(\P_{x_0} \ne \Q_{x_0}\). By Corollary~\ref{coro: local absolute continuity}, all points in \(J\) are non-separating. In particular, all points in \(J^\circ (= \tilde{J}^\circ\)) are non-separating. Thus, if \(\s = \S\) on \(J^\circ\) then \(d\m/d \M= 1\) on \(J^\circ\), and if \(\m = \M\) on \(\mathcal{B}(J^\circ)\) then \(d^+ \s/d \S= 1\) on \(J^\circ\). Consequently, irrespective of our hypothesis, both \(\s = \S\) on \(J^\circ\) and \(\m = \M\) on \(\mathcal{B}(J^\circ)\) hold.
	As accessibility is characterized by scale and speed on the interior of the state space, we have \(J = \tilde{J}\) and, by continuity, \(\s = \S\) on \(J\). Finally, let \(b \in J \setminus J^\circ\) be accessible for \((x \mapsto \P_x)\). As noted before, \(b\) is non-separating and we get \(\m (\{b\}) = \M(\{b\})\) from Definition~\ref{def: non-sep bound}. Thus, \(\m = \M\) on \(\mathcal{B}(J)\) and we conclude~\(\P_{x_0} = \Q_{x_0}\) from the fact that scale and speed characterize a diffusion uniquely. This is a contradiction and the proof is complete.
\end{proof}

\begin{remark} \label{rem: comment desm}
	The main result from \cite{desmettre} provides necessary and sufficient conditions for the martingale property of certain non-negative local martingales. Equivalently, these are necessary and sufficient conditions for \(\Q_{x_0} \ll_\textup{loc} \P_{x_0}\) in case
	\[
	d\S = d \s /\varphi^2,\qquad d \M= \varphi^2 d \m, 
	\]
	where \(\varphi\) is a positive function in the domain of the extended generator of the diffusion \((x \mapsto \P_x)\). Corollary~\ref{coro: local absolute continuity} gives a complete characterization of local absolute continuity for two general diffusions without predetermined structural assumptions. As a consequence, our result shows that the above structure of \((\S, \M)\) is necessary, where, in general, \(\varphi\) does not need to be in the domain of the extended generator, although $\varphi^2$ always needs to be continuous (recall Remark~\ref{rem:130222a1}). This observation confirms a conjecture of Chris Rogers (personal communication) about the structure of \(\S\) and \(\M\). For diffusions with open state space, the conjecture also follows from the main result in \cite{orey}.
	
	We stress that our approach is quite different from the one in \cite{desmettre}, as we do not work with a candidate density process. 
	In fact, because in our general framework there is no obvious candidate for a density process, a different approach seems to be necessary.
\end{remark}

In the next section we discuss an application of our main result to mathematical finance.

\section{Deterministic Conditions for NFLVR in General Diffusion Models} \label{sec: appl}
In this section we discuss an application of our main result to mathematical finance. Namely, we derive deterministic conditions for the existence and absence of arbitrage in the sense of the notion \emph{no free lunch with vanishing risk (NFLVR)} as introduced in \cite{DS}.
We do this in the most general single asset regular diffusion model whose price process is bounded from below.

Let us start with an introduction to our financial framework. Take \(l \in \mathbb{R}\) and let \(J\) be either \([l, \infty)\) or \((l, \infty)\).
Moreover, let \(\s \colon J \to \mathbb{R}\) be a scale function, let \(\m\) be a speed measure on \((J, \mathcal{B}(J))\) and let \((J \ni x \mapsto \P_x)\) be a regular diffusion with characteristics \((\s, \m)\).

In the following we fix an initial value \(x_0 \in J^\circ\) and a time horizon \(T \in (0, \infty]\). In case \(T = \infty\) we understand the interval \([0, T]\) as \(\mathbb{R}_+\)
and read expressions like ``$s\le T\,$'' as ``$s<\infty$''.
Define  \(\mathcal{G}_t \triangleq \mathcal{F}_t\) for \(t < T\) and \(\mathcal{G}_T \triangleq \sigma (\X_s, s \leq T)\).
In the following, we consider \((\Omega, \mathcal{G}_T, (\mathcal{G}_t)_{t \leq T}, \P_{x_0})\) as our underlying filtered probability space. Moreover, we use the convention that all processes are time indexed over~\([0, T]\). The process \(\X\) will represent an asset price process.
The interest rate is assumed to be zero.

We now recall the notion NFLVR. To introduce necessary terminology, suppose for a moment that \(\X\) is a semimartingale.
A one-dimensional predictable process \(H = (H_t)_{t \leq T}\) is called a \emph{trading strategy} if the stochastic integral \((\int_0^t H_s d \X_s)_{t \leq T}\) is well-defined. Further, a trading strategy \(H\) is called \emph{admissible} if there exists a constant \(c \geq 0\) such that a.s. \(\int_0^t H_s d \X_s \geq - c\) for all \(t \leq T\).
The convex cone of all contingent claims attainable from zero initial capital is given by 
\begin{align*}
K \triangleq \Big\{ \int_0^T H_sd \X_s \colon H &\text{ admissible,} \text{ and if \(T = \infty\), then \(\lim_{t \to \infty} \int_0^t H_s d \X_s\) exists a.s.} \Big\}.
\end{align*}
Let \(C\) be the set of essentially bounded random variables that are dominated by claims in \(K\), i.e.,
\[
C \triangleq \big\{ g \in L^\infty \colon \exists f \in K \text{ such that } g \leq f\ \text{a.s.} \big\}.
\]

\begin{definition}
We say that NFLVR holds in our market if \(\X\) is a semimartingale and
\(
\overline{C} \cap L^\infty_+ = \{0\}, 
\)
where \(\overline{C}\) denotes the closure of \(C\) in \(L^\infty\) w.r.t. the norm topology and \(L^\infty_+\) denotes the cone of non-negative random variables in \(L^\infty\). 
\end{definition}

According to the celebrated fundamental theorem of asset pricing (\cite{DS}), NFLVR is equivalent to the existence of an equivalent local martingale measure (ELMM), i.e., a probability measure \(\mathds{Q}\) on the filtered space \((\Omega, \mathcal{G}_T, (\mathcal{G}_t)_{t \leq T})\) such that \(\P_{x_0} \sim \mathds{Q}\) on \(\mathcal{G}_T\) and \(\X\) is a local \(\mathds{Q}\)-martingale. 
We emphasize that, thanks to Girsanov's theorem (Lemma~\ref{lem: Girs}), the existence of an ELMM immediately implies that \(\X\) is a \(\P_{x_0}\)-semimartingale.
In the following, we describe NFLVR in a deterministic manner for our single asset diffusion market. 

\begin{condition} \label{cond : beta1}
	There exists a function \(\beta \colon J^\circ \to \mathbb{R}\) such that
	\begin{align}
	\qquad\beta^2 \in L^1_\textup{loc}(J^\circ), \label{eq: nflvr1}
	\end{align}
	and, up to increasing affine transformations,
	\begin{align} \label{eq: nflvrS}
	\s (x) = \int^x \exp \Big( \int^y \beta (z) dz \Big) dy, \quad x \in J^\circ.
	\end{align}
	Moreover, if \(J = [l, \infty)\), then \(\m (\{l\}) = \infty\), i.e., the boundary point $l$ is absorbing for the diffusion $(x\mapsto\P_x)$ whenever it is accessible.
\end{condition}

\begin{condition} \label{cond: beta2}
Condition \ref{cond : beta1} holds and, for \(\beta\) as in Condition \ref{cond : beta1},
\begin{equation}\label{eq: nflvr2}
\int_{l +} (x - l) \beta (x)^2 dx < \infty.
\end{equation}
\end{condition}

\begin{condition} \label{cond: beta3}
	Condition \ref{cond : beta1} holds and
\begin{gather}
\text{either}\quad
\Big(\s(l) = - \infty\Big)
\quad\text{or}\quad
\Big(\s(l) > - \infty
\text{ and }
\int_{l+} |\s (x) - \s (l)| \m (dx) = \infty\Big),
\label{eq: nflvr3}\\[2mm]
\int_{l +} (x - l) \s' (x) \m (dx) = \infty,
\label{eq: nflvr4}
\end{gather}
where \(\s'\) denotes the derivative of \(\s\), see \eqref{eq: nflvrS}.
\end{condition}

\begin{theorem} \label{theo: NFLVR finite tiome horizon}
	If \(T < \infty\), then NFLVR holds if and only if at least
	one of Conditions \ref{cond: beta2} and~\ref{cond: beta3} holds.
Moreover, if NFLVR holds, the unique ELMM is given by \(\Q_{x_0}\), where \((x \mapsto \Q_x)\) is the diffusion with
the interior of the state space $J^\circ$,
characteristics $(\on{Id},\s'\,d\m)$ on $J^\circ$
and the boundary point $l$ being absorbing whenever it is accessible.
\end{theorem}

It is worth noting that
Conditions \ref{cond: beta2} and~\ref{cond: beta3}
do not exclude each other.

\begin{remark}\label{rem:070125a1}
Assume that NFLVR holds for some $T<\infty$.
Then either the boundary point $l$ is inaccessible for both diffusions
$(x\mapsto\P_x)$ and $(x\mapsto\Q_x)$
or $l$ is accessible (and absorbing) for both diffusions.
This follows from
Theorem~\ref{theo: NFLVR finite tiome horizon} and
Lemma~\ref{lem:130222a2} together with the facts that
\begin{itemize}
\item
Condition~\ref{cond: beta2} means that $l$ is half-good for $\P_{x_0}$ and~$\Q_{x_0}$
(compare, in particular, \eqref{eq:130222a1} with~\eqref{eq: nflvr2});

\item
\eqref{eq: nflvr3} (resp.,~\eqref{eq: nflvr4}) means that $l$ is inaccessible under~$\P_{x_0}$
(resp., under~$\Q_{x_0}$).
\end{itemize}
Furthermore, $\infty$ is inaccessible both under $\P_{x_0}$ and under~$\Q_{x_0}$.
The former is contained in our setting ($\infty\notin J$),
while the latter follows from $\S(\infty)=\infty$ (with $\S=\on{Id}$).
\end{remark}

\begin{proof}[Proof of Theorem~\ref{theo: NFLVR finite tiome horizon}]
We start with the necessity of the conditions. Suppose that NFLVR holds, which means, by the fundamental theorem of asset pricing (\cite[Corollary 1.2]{DS}), that an ELMM \(\mathds{Q}\) exists. 
As non-negative local martingales (which are supermartingales) cannot resurrect from zero (\cite[Proposition II.3.4]{RY}), in case \(J = [l, \infty)\),
using $\mathds{Q}\sim\P_{x_0}$ we get that
the left boundary \(l\) has to be an absorbing state for \((x \mapsto \P_x)\), i.e., \(\m (\{l\}) = \infty\).

Now we investigate properties of the scale function \(\s\). 
	Take an arbitrary point \(y_0 \in J^\circ\) and define \(\Y_t \triangleq \X_{ (t + T_{y_0}) \wedge T}\) and \(\mathcal{A}_t \triangleq \mathcal{G}_{(t + T_{y_0}) \wedge T}\) for \(t \in [0, T]\). 
	Recall from Lemma \ref{lem: diff hit points fast} that \(\mathds{P}_{x_0} (T_{y_0} < T) > 0\). In particular, as \(\mathds{P}_{x_0} \sim \mathds{Q}\) on \(\mathcal{G}_T\), we have \(\mathds{Q} (T_{y_0} < T) > 0\), which means that we can define a probability measure \(\mathds{K}\) on \((\Omega, \mathcal{G}_T)\) by the formula
	\[
	\mathds{K} ( d \omega ) \triangleq \mathds{Q} (d \omega | T_{y_0} < T) = \frac{ \mathds{Q} (d \omega \cap \{T_{y_0} < T\})}{\mathds{Q} (T_{y_0} < T)}.
	\]
	By the definition of an ELMM and Lemma~\ref{lem: change of time}, \(\Y\) is a local \(\mathds{Q}\)-\((\mathcal{A}_t)_{t \leq T}\)-martingale. 
	Moreover, as \(\{T_{y_0} < T\} \in \mathcal{A}_0\), \(\Y\) is a local \(\mathds{K}\)-\((\mathcal{A}_t)_{t \leq T}\)-martingale. 
	We define 
\[
L (t) \triangleq \inf (s \in [0, T] \colon \langle \Y, \Y \rangle_s > t) \wedge T, \quad t \in \mathbb{R}_+,
\]
where \(\langle \Y, \Y\rangle\) denotes the \(\mathds{K}\)-\((\mathcal{A}_t)_{t \leq T}\)-quadratic variation process of \(\Y\).
As, in case \(J = [l, \infty)\), the left boundary point \(l\) is absorbing for \((x \mapsto \mathds{P}_x)\), we get from Lemmata~\ref{lem: scale fct} and~\ref{lem: change of time} that \(\s(\Y)\) is a local \(\P_{x_0}\)-\((\mathcal{A}_t)_{t \leq T}\)-martingale. 
Hence, by Girsanov's theorem (Lemma~\ref{lem: Girs}), as \(\mathds{K} \ll \mathds{Q} \sim \mathds{P}_{x_0}\) on \(\mathcal{A}_T = \mathcal{G}_T\), the process \(\s (\Y)\) is a \(\mathds{K}\)-\((\mathcal{A}_t)_{t \leq T}\)-semimartingale. 
Using once again Lemma \ref{lem: change of time}, this implies that
\(
\s (\Y_{L})
\)
is a \(\mathds{K}\)-\((\mathcal{A}_{L (t)})_{t \geq 0}\)-semimartingale.
By the Doeblin, Dambis, Dubins--Schwarz theorem (\cite[Theorem V.1.7]{RY}), the time-changed process \(\Y_L\) is a standard \(\mathds{K}\)-\((\mathcal{A}_{L(t)})_{t \geq 0}\)-Brownian motion stopped at \(\langle \Y, \Y\rangle_T\), which is an \((\mathcal{A}_{L (t)})_{t \geq 0}\)-stopping time by \cite[Lemma~10.5]{Jacod}. 
It follows from the fact (\cite[Proposition~IV.1.13]{RY}) that continuous local martingales and their quadratic variation processes have the same intervals of constancy, and Lemma \ref{lem: exit immediately}, that \(\mathds{K}\)-a.s. \(\langle \Y, \Y\rangle_T > 0\) and \(L(0) = 0\).
Consequently, 
\[
\mathds{K} (\Y_{L (0)} = y_0) = \frac{\mathds{Q} (\X_{T_{y_0} \wedge T} = y_0, T_{y_0} < T)}{\mathds{Q} (T_{y_0} < T)} = 1,
\]
and we deduce from Theorem~\ref{lem: diff convex} that, in an open neighborhood of \(y_0\), the scale function \(\s\) is the difference of two convex functions. Recall that being the difference of two convex functions on an open (or closed) convex subset of a finite-dimensional Euclidean space is a local property (see \cite[(I) on p.~707]{hart}).
Hence, as \(y_0\) was arbitrary, we conclude that the scale function \(\s\) is the difference of two convex functions on \(J^\circ\).
In particular, $\s$ has a right-continuous right-hand derivative (and a left-continuous left-hand one) and they are locally bounded on $J^\circ$.
With a little abuse of notation, we denote the right-hand derivative of $\s$ by~\(\s'\). This notation is motivated by the fact that $\s$ is differentiable, as we prove below, cf.~\eqref{eq: nflvrS}.

Next, let us understand the structure of \(\mathds{Q}\) in a more precise manner.
We discuss the case $J=[l,\infty)$ where \(l\) is accessible (and absorbing) under \(\P_{x_0}\).
The inaccessible case \(J = (l, \infty)\) follows the same way.
Notice that $\s'd\m$ is a valid speed measure (recall Remark~\ref{rem:170322a1}).
Indeed, its local finiteness on $J^\circ$ follows from the local boundedness of $\s'$ on $J^\circ$, so we only need to verify~\eqref{eq:090322a2}.
As $\s$ is strictly increasing, $\s'$ cannot have intervals of zeros.
Therefore, for any $a<b$ in $J^\circ$, there exists an $x\in[a,b)$ with $\s'(x)>0$, hence $\s'>0$ on some $[x,x+\varepsilon)\subset[a,b)$ due to the right continuity of $\s'$, and this implies~\eqref{eq:090322a2}.
Now take \(f \in C_b([l, \infty); \mathbb{R})\) such that the following holds true:
the restriction $f|_{(l,\infty)}$ is a difference of two convex functions $(l,\infty)\to\mathbb R$ and,
denoting the right-hand derivative of $f$ on $(l,\infty)$ by $f'_+$ (which necessarily has locally finite variation on $(l,\infty)$),
there exists a function \(g \in C_b([l, \infty); \mathbb{R})\) with $g(l)=0$ such that
$df'_+=2g\s'd\m$ on $(l,\infty)$ in the sense explained in~\eqref{eq:180322a1}.
Let \(\{L^x_t (\X) \colon (t,x) \in \mathbb{R}_+ \times \mathbb{R}\}\) be the
(continuous in $t$ and right-continuous in $x$)
semimartingale local time of the coordinate process \(\X\) under \(\P_{x_0}\).
Under \(\P_{x_0}\), we obtain, for all \(t < T_l\),
\begin{align*}
f (\X_t) &= f(x_0)+\int_0^t \Big(\frac{d^- f}{dx}\Big) (\X_s) d \X_s + \frac{1}{2} \int_{J^\circ} L^x_t (\X) 2g (x) \s' (x) \m (dx)
\\&=  f(x_0)+\int_0^t \Big(\frac{d^- f}{dx}\Big) (\X_s) d \X_s + \int_{J^\circ} L^{\s(x)}_t (\s (\X)) g (x) \m (dx)
\\&= f(x_0)+\int_0^t \Big(\frac{d^- f}{dx}\Big) (\X_s) d \X_s + \int_{\s(J^\circ)} L^{x}_t (\s (\X)) g (\s^{-1} (x)) \m \circ \s^{-1} (dx)
\\&= f(x_0)+\int_0^t \Big(\frac{d^- f}{dx}\Big) (\X_s) d \X_s + \int_0^t g (\s^{-1} (\s(\X_s))) ds, 
\\&= f(x_0)+\int_0^t \Big(\frac{d^- f}{dx}\Big) (\X_s) d \X_s + \int_0^t g (\X_s) ds, 
\end{align*} 
where we use Lemma \ref{lem: occ smg} in the first and second line and Lemmata \ref{lem: diff homo} and \ref{lem: occ formula diff}
(more precisely, formula~\eqref{eq:241124a2})
in the fourth line. 
Now, the coordinate process \(\X\) is a local \(\mathds{Q}\)-martingale and hence, by the above formula, the process
\begin{align}\label{eq: test process ELMM}
f (\X) - f (x_0) - \int_0^{\cdot} g (\X_s) ds
\end{align}
is a local \(\mathds{Q}\)-martingale on \([0, T_l)\). Using \cite[Proposition~5.9]{Jacod} and the fact that the process in \eqref{eq: test process ELMM} is bounded on any finite time interval, we obtain that it is a global \(\mathds{Q}\)-martingale. 
Consequently, we deduce \(\mathds{Q} = \Q_{x_0}\) (on \(\mathcal{G}_T\)) from Lemmata \ref{lem: generator} and~\ref{lem: loc uni}.

In summary, if NFLVR holds, then \(\P_{x_0} \sim \Q_{x_0}\) on \(\mathcal{G}_T\) and the unique ELMM is given by \(\Q_{x_0}\).
By virtue of Corollary \ref{coro: local absolute continuity}
and Remark~\ref{rem:191124a1},
\(\P_{x_0} \sim \Q_{x_0}\) on \(\mathcal{G}_T\) if and only if all interior points are non-separating (which implies Condition \ref{cond : beta1}) and in addition either \(l\) is non-separating (which means that Condition \ref{cond: beta2} holds) or \(\P_{x_0}, \Q_{x_0}\) do not reach $l$ (which means that Condition \ref{cond: beta3} holds). This proves the necessity of the conditions. 

It remains to discuss the sufficiency. As explained above, if either Condition \ref{cond: beta2} or Condition~\ref{cond: beta3} holds, then \(\P_{x_0} \sim \Q_{x_0}\) on \(\mathcal{G}_T\). As \(\X\) is a local \(\Q_{x_0}\)-martingale (Lemma \ref{lem: scale fct}), we can conclude that \(\Q_{x_0}\) is an ELMM and hence, NFLVR holds by the fundamental theorem of asset pricing.
The proof is complete.
\end{proof}

\begin{remark}
	Under the NFLVR condition, in case the asset represented by \(\X\) gets bankrupt in the sense \(\X\) reaches the boundary point \(l\), it remains there. Of course, this observation holds beyond our diffusion framework, as it is a consequence of the fact that non-negative local martingales are non-negative supermartingales, which cannot resurrect from zero. 
\end{remark}

\begin{remark}
In the It\^o diffusion setting (see \cite{MU12b}) conditions \eqref{eq: nflvr1}--\eqref{eq: nflvr4} have the following financial interpretations: condition~\eqref{eq: nflvr1} means that the market price of risk is $\P_{x_0},\Q_{x_0}$-a.s. square integrable on \([0, t]\) for every \(t < T_l\). If, in addition to~\eqref{eq: nflvr1}, either~\eqref{eq: nflvr2} or both \eqref{eq: nflvr3} and~\eqref{eq: nflvr4} hold, then the market price of risk is even $\P_{x_0},\Q_{x_0}$-a.s. locally square integrable on $\mathbb R_+$.
\end{remark}

We also give a result for the infinite time horizon setup.

\begin{theorem} \label{theo: NFLVR infinite tiome horizon}
	If \(T = \infty\), then NFLVR holds if and only if 
		Condition~\ref{cond: beta2} and \(\s (\infty) = \infty\) hold.
	Moreover, in case NFLVR holds, the unique ELMM is given by \(\Q_{x_0}\),  where \((x \mapsto \Q_x)\) is the diffusion with
the interior of the state space $J^\circ$,
characteristics $(\on{Id},\s'\,d\m)$ on $J^\circ$
and the boundary point $l$ being absorbing whenever it is accessible.
\end{theorem}

For completeness, we notice that, as NFLVR with $T=\infty$ implies NFLVR with any $T<\infty$,
the messages of Remark~\ref{rem:070125a1} apply also under NFLVR with $T=\infty$.

\begin{proof}[Proof of Theorem~\ref{theo: NFLVR infinite tiome horizon}]
Using the arguments explained in the proof of Theorem \ref{theo: NFLVR finite tiome horizon}, it suffices to understand when \(\P_{x_0} \sim \Q_{x_0}\) on \(\mathcal{F}\) (\(= \mathcal{G}_T\) as \(T = \infty\)). By Corollary~\ref{cor:110222a1},
in the case when $\P_{x_0}\ne\Q_{x_0}$ (cf. Remark~\ref{rem:191124a2}),
we have \(\P_{x_0} \sim \Q_{x_0}\) on \(\mathcal{F}\)
if and only if all points in \(J^\circ\) are non-separating (which means that Condition~\ref{cond : beta1} holds),
\(l\) is non-separating (which upgrades Condition~\ref{cond : beta1} to Condition~\ref{cond: beta2}) and
$\s(\infty)=\infty$.
In the other case $\P_{x_0}=\Q_{x_0}$,
clearly, \(\P_{x_0} \sim \Q_{x_0}\) on \(\mathcal{F}\),
as well as
both Condition~\ref{cond: beta2} and $\s(\infty)=\infty$ hold.
This concludes the proof.
\end{proof}

\begin{remark}
(a)
Theorem \ref{theo: NFLVR infinite tiome horizon} shows that if NFLVR holds for \(T = \infty\), then bankruptcy is certain on the long run, i.e.,
$\P_{x_0}$-a.s. \(\lim_{t \to \infty} \X_t = l\).
Indeed, Condition~\ref{cond: beta2} means that $l$ is a non-separating boundary point for the diffusions $(x\mapsto\P_x)$ and $(x\mapsto\Q_x)$,
hence, by Discussion~\ref{disc:190223a1}, we have $\s(l)>-\infty$.
Together with \(\s (\infty) = \infty\) this yields the statement via Lemma~\ref{lem:200223a2}.

(b)
As pointed out by the referee, it is interesting to notice that,
under NFLVR in the case \(T = \infty\),
the price process cannot be a uniformly integrable martingale under the ELMM \(\Q_{x_0}\).
Let us mention two ways to see this. First, by~(a) and the equivalence $\P_{x_0}\sim\Q_{x_0}$, we have \(\Q_{x_0}\)-a.s. \(\lim_{t\to\infty}\X_t = l\).
Now, if $\X$ were a uniformly integrable martingale under $\Q_{x_0}$,
then we would obtain the contradiction
\(x_0 = \E^{\Q_{x_0}} [ \X_t ] \to l\) as \(t \to \infty\).
Alternatively, the fact can be deduced from \cite[Theorem~1.1]{HR19} because,
as explained in \cite[Example~3.1]{HR19},
Condition~(A) from \cite{HR19} is violated in diffusion settings with state space \([l, \infty)\) or \((l, \infty)\).
\end{remark}

\begin{remark}
The state spaces \([l,\infty)\) or \((l, \infty)\) seem to be economically interesting choices.
Nevertheless, using identical arguments as in the proofs of Theorems \ref{theo: NFLVR finite tiome horizon} and \ref{theo: NFLVR infinite tiome horizon}, Theorem \ref{theo: main1} also yields deterministic characterizations for NFLVR in case \(J\) is any other interval in \(\mathbb{R}\).
The precise statements are left to the reader.
We emphasize that for financial applications we need to consider $J\subset\mathbb R$
(as opposed to $J\subset[-\infty,\infty]$ in the general setting from Section~\ref{sec: diff basics}), as this is necessary to define NFLVR
(if $\infty\in J$ or $-\infty\in J$, then the coordinate process cannot be a semimartingale).
\end{remark}

\begin{remark}
	Theorems \ref{theo: NFLVR finite tiome horizon} and~\ref{theo: NFLVR infinite tiome horizon} include their counterparts \cite[Theorems 3.1, 3.5]{MU12b} for the It\^o diffusion setting. 
	In the framework from \cite{MU12b} the scale function \(\s\) is assumed to be continuously differentiable with absolutely continuous derivative. Our results show that this assumption is necessary, which seems to be surprising at first glance. Although for NFLVR to hold the scale function has to have the same structure as in \cite{MU12b}, our results apply for arbitrary speed measures, while in \cite{MU12b} these have to be absolutely continuous w.r.t. the Lebesgue measure. In particular, in contrast to the results in \cite{MU12b}, ours cover diffusions with sticky points, which are interesting models in the presence of takeover offers, as explained the Introduction.
\end{remark}

	Let us end this section with two (classical) examples, which show how Theorems \ref{theo: NFLVR finite tiome horizon} and~\ref{theo: NFLVR infinite tiome horizon} can be applied. 
	\begin{example}
		\begin{enumerate}
			\item[(a)]
			The famous Black--Scholes model with drift \(\mu \in \mathbb{R}\) and volatility \(\sigma \not = 0\) can be rephrased in our diffusion language by taking \(J = (0, \infty)\) and 
			\begin{align*}
			\s (x) = \begin{cases} - \frac{x^{- 2 \nu}}{2 \nu}, & \nu \not = 0, \\ \log (x), & \nu = 0,\end{cases}\qquad \m (dx) = \frac{x^{2 \nu - 1}}{\sigma^2} dx, 
		\end{align*}
		where \(\nu = \mu / \sigma^2 -  1 / 2\). In this case Condition~\ref{cond : beta1} holds with 
		\(
		\beta (x) =- 2 \mu /(\sigma^2 x).
		\)
		Moreover, it is straightforward to check that Condition~\ref{cond: beta3} is satisfied, while Condition~\ref{cond: beta2} is violated whenever $\mu\ne0$ (but satisfied in the case $\mu=0$).
As a consequence, the Theorems~\ref{theo: NFLVR finite tiome horizon} and \ref{theo: NFLVR infinite tiome horizon} show that NFLVR holds for \(T < \infty\) but fails for \(T = \infty\)
in the case $\mu\ne0$ (and holds for $T=\infty$ in the trivial case $\mu=0$).
Of course, this recovers the very classical results from the literature.

			\item[(b)]
			A classical counterexample in arbitrage theory is the three-dimensional Bessel process
		(see, e.g., \cite{DS1995a}).
			In our language, this corresponds to the case \(J = (0, \infty)\), \(\s(x) = - 1 / x\) and \(\m (dx) = x^2 dx\). Condition~\ref{cond : beta1} holds with \(\beta (x) = - 2 / x\), but, as the reader easily checks, both Condition~\ref{cond: beta2} and Condition~\ref{cond: beta3} are violated. As a consequence, Theorem~\ref{theo: NFLVR finite tiome horizon} yields that NFLVR fails for any finite time horizon. Again, we recover the known results from the literature.
		\end{enumerate}
	\end{example}

\section{Proof of the Main Theorem} \label{sec: pf}
As homeomorphic space transformations do not affect the question of equivalence and singularity, we can assume that one of the diffusions of interest is on natural scale. More precisely, we assume that \((x \mapsto \Q_x)\) is on natural scale, i.e., \(\S = \on{Id}\). The general result then follows from Lemma \ref{lem: diff homo}.

\subsection{Some Preparations}
In this subsection we collect some auxiliary results which are needed in the proof of Theorem~\ref{theo: main1}. 
One of our main tools is a time-change argument which we learned from \cite{orey}, where it was used to prove local equivalence for diffusions with open state space.
Excluding, for a moment, the possibility of (instantaneously or slowly) reflecting boundaries, the idea is roughly speaking as follows:
via a change of time we reduce certain questions related to equivalence of two diffusions to the same question for Brownian motions with and without (generalized)
drift and possibly absorbing boundaries.
This reduction brings us into the position to apply one of the main results from \cite{cherUru}, i.e., that Theorem~\ref{theo: main1} holds in It\^o diffusion settings with possibly absorbing boundaries.
To treat (instantaneously and slowly) reflecting boundaries, we combine the time change argument with a symmetrization trick, which, roughly speaking, states that reflecting boundaries can be considered as interior points of a symmetrized diffusion. As equivalence of laws of diffusions is typically not preserved by symmetrization, this part of the proof requires some additional technical considerations.

\subsubsection{Proof of Lemma~\ref{lem:130222a2}} \label{sec: pf of lem:130222a2}
	Suppose that \(b\in \partial J\) is half-good.
	As \(\S = \on{Id}\), this implies that \(b\in \mathbb{R}\).
	By definition, there exists a non-empty open interval
	\(B \subsetneq J^\circ\)
	with \(b\) as endpoint such that all points in \(B\)
	and the other endpoint of $B$
	are non-separating (good). 
	By virtue of (i) and (ii) in Definition~\ref{def: non-sep int} and recalling our standing assumption \(\S = \on{Id}\), it follows that the scale function \(\s\) is differentiable on \(B\) (see \cite[p.~204]{saks}) and that its derivative \(\s'\) satisfies the equation
	\[
	d \s' (x) = \beta (x) \s' (x) dx.
	\]
	This means that, up to increasing affine transformations,
	\begin{equation}\label{eq:280922a1}
	\s (x) = \int^x \exp \Big( \int^y \beta (z) dz \Big) dy, \quad x \in B.
	\end{equation}
	Take \(x_0 \in B\), let \(\Qo_{x_0}\)
be the law of a Brownian motion which is absorbed in the boundaries of \(B\) and let \(\Po_{x_0}\)
be a  the law of a diffusion started at \(x_0\) absorbed in accessible boundaries of \(B\) with scale function \(\s\)
	and speed measure 
	\(
	\mo (dx) = dx/\s'(x)\) on \(\mathcal{B} (B). 
	\)
	We deduce from Lemma \ref{lem: sep ito diff} that \(\Po_{x_0}\sim\Qo_{x_0}\).
	In particular, this means that \(\Po_{x_0}\) and \(\Qo_{x_0}\) have the same state space \(\on{cl}(B)\).
	Next, we transfer this equivalence to the (stopped) diffusions \(\P_{x_0} \circ \X_{\cdot \wedge T}^{-1}\) and \(\Q_{x_0} \circ \X_{\cdot \wedge T}^{-1}\), where \(T \triangleq \inf (t \geq 0 \colon \X_t \not \in B).\)
	Define
	\[
	\tilde{\l}^* (t, x) \triangleq 
	\begin{cases}
	\limsup_{h \searrow 0}
	\frac{\int_0^t \1 \{x - h < \X_s < x + h\} ds}{2h}
	&\text{if }x\in B,
	\\[2mm]
	0 & \text{if } x \in \partial B,
	\end{cases}
	\]
	and set 
	\[
	\g (t) \triangleq \begin{cases} \int_B \tilde{\l}^* (t, x) \M (dx), & t < T,\\ \infty, & t \geq T. \end{cases}
	\]
	Further, let \(\g^{-1}\) be the right-inverse of \(\g\), i.e.,
	\(
	\g^{-1} (t) \triangleq \inf (s \geq 0 \colon \g(s+) > t), t \in \mathbb{R}_+.
	\)
	Then, according to the chain rule for diffusions (Theorem~\ref{theo: chain rule}), we get that 
	\begin{align} \label{eq: time change 1st}
	\Qo_{x_0} \circ \X_{\g^{-1}(\cdot)}^{-1} = \Q_{x_0} \circ \X^{-1}_{\cdot \wedge T}.
	\end{align}
	Set
	\[
	\l^* (t, x) \triangleq
	\begin{cases}
	\limsup_{h \searrow 0}
	\frac{\int_0^t \1 \{x - h < \X_s < x + h\} ds}{\m^\circ ((x - h, x + h))}
	&\text{if }x\in B,
	\\[2mm]
	0
	&\text{if }x \in \partial B.
	\end{cases}
	\]
	Now, for \(t < T\), using part~(iii) of Definition~\ref{def: non-sep int},
together with Lemma~\ref{lem:130222a1},
	we obtain
	\begin{equation} \label{eq: comp Lem 2.10 pf}
	\begin{split}
	\g (t) &= \int_B \tilde{\l}^* (t, x) \M(dx) 
	= \int_B \l^* (t, x) \Big(\frac{d \m^\circ}{dx}\Big) (x) \M(dx) 
	= \int_B \frac{ \l^* (t, x) \M(dx)}{\s'(x)}
	\\&= \int_B \l^* (t, x) \Big(\frac{d \m}{d \M}\Big) (x) \M (dx) 
	= \int_B \l^* (t, x) \m(dx).
	\end{split}
	\end{equation}
	Hence, again by the chain rule, we have
	\begin{align} \label{eq: time change 2nd}
	\Po_{x_0} \circ \X_{\g^{-1}(\cdot)}^{-1} = \P_{x_0} \circ \X^{-1}_{\cdot \wedge T}.
	\end{align}
	Take \(G \in \mathcal{F}\) such that \(\P_{x_0} (\X_{\cdot \wedge T} \in G) = 0\). Then, by \eqref{eq: time change 2nd}, also \(\Po_{x_0}(\X_{\g^{-1}(\cdot)} \in G) = 0\), which yields that \(\Qo_{x_0} (\X_{\g^{-1}(\cdot)} \in G) = 0\), by the equivalence of \(\Po_{x_0}\) and \(\Qo_{x_0}\) on \(\mathcal{F}\), and finally, using \eqref{eq: time change 1st}, we get \(\Q_{x_0} (\X_{\cdot \wedge T} \in G) = 0\).
	Conversely, if \(G \in \mathcal{F}\) is such that \(\Q_{x_0} (\X_{\cdot \wedge T} \in G) = 0\), then we get \(\P_{x_0} (\X_{\cdot \wedge T} \in G) = 0\) by a similar reasoning.
	Thus, we conclude that \(\P_{x_0} \circ \X_{\cdot \wedge T}^{-1} \sim \Q_{x_0} \circ \X_{\cdot \wedge T}^{-1}\). As an equivalent change of measure does not change the state space, the claim of Lemma~\ref{lem:130222a2} follows.\qed

\subsubsection{Criteria for Equivalence} \label{sec: cri equivalence}
In the following, we investigate equivalence up to a hitting time.
Recall that \(\A\) denotes the set of separating points.

\begin{lemma} \label{lem: loc equivalence}
	Suppose that \(x_0 \in J^\circ\) and that \(a, c \in J^\circ\) are such that \(a < x_0 < c\) and \([a, c] \subset J^\circ \setminus \A\). Then, \(\P_{x_0} \sim \Q_{x_0}\) on \(\mathcal{F}_{T_a \wedge T_c}\).
\end{lemma}

\begin{proof}
		Take \(a' < a\) and \(c' > c\) such that \([a', c'] \subset J^\circ \setminus \A\). We stress that \(a'\) and \(c'\) exist as \(J^\circ \setminus \A\) is open. To simplify our notation, we set \(T' \triangleq T_{a'} \wedge T_{c'}\).
	Let \(([a', c'] \ni x \mapsto \Qo_x)\) be a Brownian motion which is absorbed in both \(a'\) and~\(c'\).
	Recall that all points in \([a', c']\) are non-separating by hypothesis, which inter alia implies that \(d \m /d \M\) exists as a positive continuous function on \([a', c']\).
	We define a measure \(\m^\circ\) on \(\mathcal{B}([a', c'])\) by
	\[
	\frac{\m^\circ (dx)}{dx} \triangleq \Big( \frac{d \m}{d \M}\Big) (x) \text{ on } (a', c'), \qquad \m^\circ (\{a'\}) \equiv \m^\circ (\{c'\}) \triangleq \infty.
	\]
Notice that for every 
interval \(I \subset (a', c')\) with strictly positive length
we have \(\m^\circ (I) \in (0, \infty)\), because \(d \m /d \M\) is a positive continuous function on \([a', c']\). In other words, \(\m^\circ\) is a valid speed measure (recall Remark~\ref{rem:170322a1}).
Let \((x \mapsto \Po_x)\) be a diffusion with characteristics \((\s, \mo)\). 
Notice that both $a'$ and $c'$ are accessible for \((x \mapsto \Po_x)\) (and then, due to the infinite masses in the speed measure, absorbing), as $|\s(a')|,|\s(c')|<\infty$ and $\mo((a',c'))<\infty$ (recall~\eqref{eq:101022a3}). In particular, the state space of  \((x \mapsto \Po_x)\) is given by $[a',c']$. Further,
as in the proof of Lemma~\ref{lem:130222a2}, $\s$ has the representation~\eqref{eq:280922a1} on $(a',c')$, while, due to~\eqref{eq:111022a1}, $\mo(dx)=dx/\s'(x)$ on $(a',c')$.
Therefore, Lemma~\ref{lem: sep ito diff} applies and yields~\(\Po_{x_0} \sim \Qo_{x_0}\).

We now deduce the claim of the lemma from this equivalence. We use a refined version of the time-change argument from the proof of Lemma \ref{lem:130222a2}.
		Define
		\[
		\tilde{\l}^* (t, x) \triangleq	\begin{cases}
		\limsup_{h \searrow 0}
		\frac{\int_0^t \1 \{x - h < \X_s < x + h\} ds}{2h}
		&\text{if }x\in (a', c'),
		\\[2mm]
		0 & \text{if } x \in \{a', c'\},
		\end{cases}
		\]
		and set 
		\[
		\g (t) \triangleq \begin{cases} \int_{a'}^{c'} \tilde{\l}^* (t, x) \M (dx), & t < T',\\ \infty, & t \geq T'. \end{cases}
		\]
		Further, \(\g^{-1} (t) = \inf (s \geq 0 \colon \g (s+) > t)\) denotes the right-inverse of \(\g\).
		According to the chain rule for diffusions (Theorem~\ref{theo: chain rule}), we get that 
		\(
		\Qo_{x_0} \circ \X_{\g^{-1}(\cdot)}^{-1} = \Q_{x_0} \circ \X^{-1}_{\cdot \wedge T'}.
		\)
		Set
		\[
		\l^* (t, x) \triangleq 	\begin{cases}
		\limsup_{h \searrow 0}
		\frac{\int_0^t \1 \{x - h < \X_s < x + h\} ds}{\m^\circ ( (x - h, x + h))}
		&\text{if }x\in (a', c'),
		\\[2mm]
		0 & \text{if } x \in \{a', c'\}.
		\end{cases}
		\]
		Now, for \(t < T'\), as in~\eqref{eq: comp Lem 2.10 pf}, we obtain
		\begin{align*}
		\g (t) &= \int_{a'}^{c'} \l^* (t, x) \m(dx).
		\end{align*}
		Hence, again by the chain rule, we have
		\(
		\Po_{x_0} \circ \X_{\g^{-1}(\cdot)}^{-1} = \P_{x_0} \circ \X^{-1}_{\cdot \wedge T'}.
		\)		
		Finally, take \(G \in \mathcal{F}_{T_a \wedge T_c}\) such that \(\P_{x_0} (G) = 0\). 
Using Galmarino's test
in the form \cite[10~c), p.~87]{itokean74} (or adjusting \cite[Exercise~I.4.21]{RY} for the right-continuous filtration $(\mathcal F_t)_{t \geq 0}$),
we obtain
		\[
		0 = \P_{x_0} (G) = \P_{x_0} (\X_{\cdot \wedge T'} \in G) = \Po_{x_0} ( \X_{\g^{-1}(\cdot)} \in G).
		\]
		Since \(\Po_{x_0} \sim \Qo_{x_0}\), we obtain
		\[
		0 = \Qo_{x_0}(\X_{\g^{-1}(\cdot)} \in G) = \Q_{x_0}(\X_{\cdot \wedge T'} \in G) = \Q_{x_0} (G),
		\]
		which proves \(\Q_{x_0} \ll \P_{x_0}\) on \(\mathcal{F}_{T_a \wedge T_c}\).
The reverse absolute continuity follows by a similar reasoning.
This completes the proof.
\end{proof}

A minor variation of argument for Lemma \ref{lem: loc equivalence} also shows the next lemma. We leave the details to the reader.
Recall that \(l = \inf J (= \inf \tilde{J})\) and \(r = \sup J (= \sup \tilde{J})\).

\begin{lemma} \label{lem: local equivalence hitting time from the right}
	Suppose that the left boundary \(l\) is non-separating and that it is either inaccessible or absorbing for one (equivalently, for both) of the diffusions.
	Further, assume that \(c \in (x_0, r)\) is such that all points in \([l, c]\) are non-separating. Then, \(\P_{x_0} \sim \Q_{x_0}\) on \(\mathcal{F}_{T_c}\).
\end{lemma}

Clearly, there is also a version of Lemma \ref{lem: local equivalence hitting time from the right} for the right boundary point.

\subsubsection{Singularity under time transformations} \label{sec: singularity time change}
In the following, we give a result, which shows that singularity propagates through certain changes of time.
\begin{lemma} \label{lem: time change singularity}
Let \(\Po_{x_0}\) and \(\Qo_{x_0}\) be as in the first step of the proof of Lemma~\ref{lem: loc equivalence} with
some $a',c'\in J^\circ$, $a'<x_0<c'$.
If \(\Po_{x_0} \perp \Qo_{x_0}\) on \(\mathcal{F}_0\), then \(\P_{x_0} \perp \Q_{x_0}\) on \(\mathcal{F}_0\).
\end{lemma}
\begin{proof}
	For \(t \in \mathbb{R}_+\), define
	\begin{align} \label{eq: diff local time}
		\tilde{\l} (t, x) \triangleq 	\begin{cases}
			\limsup_{h \searrow 0}
			\frac{\int_0^t \1 \{x - h < \X_s < x + h\} ds}{\M ( (x - h, x + h))}
			&\text{if }x\in (a', c'),
			\\[2mm]
			0 & \text{if } x \in \{a', c'\},
		\end{cases}
	\end{align}
	\begin{align*}
		\f (t) & \triangleq \begin{cases} \int_{a'}^{c'} \tilde{\l} (t, x) dx, & t < T' = T_{a'} \wedge T_{c'},\\ \infty, & t \geq T = T_{a'} \wedge T_{c'}, \end{cases}
	\end{align*}
	and 
	\(\f^{-1} (t) \triangleq \inf (s \geq 0 \colon \f(s+) > t).\)
	The chain rule for diffusions as given by Theorem~\ref{theo: chain rule} yields that 
	\begin{align} \label{eq: meas after chain rule singularity}
		\Po_{x_0} = \P_{x_0} \circ \X_{\f^{-1}(\cdot)}^{-1},\qquad \Qo_{x_0} = \Q_{x_0} \circ \X_{\f^{-1}(\cdot)}^{-1}.
	\end{align}
	Take \(G \in \mathcal{F}_0\). By definition, for every \(t > 0\), we have \(G \in \sigma (\X_s, s \leq t)\) and hence,
	\[
	\{\X_{\f^{-1}(\cdot)} \in G\} \in \sigma (\X_{\f^{-1}(s)}, s \leq t) \subset \mathcal{F}_{\f^{-1} (t)}.
	\]
	Consequently, we get
	\[
	\{\X_{\f^{-1}(\cdot)} \in G\} \in \bigcap_{t > 0} \mathcal{F}_{\f^{-1} (t)} = \mathcal{F}_{\inf_{t > 0} \f^{-1} (t)} = \mathcal{F}_{\f^{-1}(0)},
	\]
	see  \cite[Lemma 9.3]{kallenberg} for the first equality.
	Thus, \(\{\X_{\f^{-1}(\cdot)} \in G, \f^{-1}(0) = 0\} \in \mathcal{F}_0\). Now, assume that \(\Po_{x_0} \perp \Qo_{x_0}\) on \(\mathcal{F}_0\). Then, there exists a set \(G \in \mathcal{F}_0\) such that \(\Po_{x_0} (G) = 1 - \Qo_{x_0}(G) = 0\).
	Hence, using~\eqref{eq: meas after chain rule singularity}, we get
	\[
	\P_{x_0} (\X_{\f^{-1}(\cdot)} \in G, \f^{-1}(0) = 0) \leq \P_{x_0} (\X_{\f^{-1}(\cdot)} \in G) = \Po_{x_0}(G) = 0,
	\]
	and 
	\[
	\Q_{x_0} (\X_{\f^{-1}(\cdot)} \not \in G \text{ or } \f^{-1}(0) \not = 0) = \Q_{x_0}(\X_{\f^{-1}(\cdot)} \not \in G) = 1 - \Qo_{x_0}(G) = 0.
	\]
	Here, we also used that \(\Q_{x_0} (\f^{-1} (0) = 0) = 1\) by Lemma \ref{lem: pos LT}. 
	We conclude that \(\P_{x_0} \perp \Q_{x_0}\) and the proof is complete.
\end{proof}

\subsubsection{Criteria for Singularity}
In Section~\ref{sec: cri equivalence} we studied criteria for equivalence up to hitting times and in Section~\ref{sec: singularity time change} we provided some preliminary observations in the direction of singularity. In the following, we take a more complete look at singularity.
The following lemma is a preparatory result.

\begin{lemma} \label{lem: in between stopping time}
	Let \(S\) be the separating time for \(\P_{x_0}\) and \(\Q_{x_0}\) with \(x_0 \in J \cap \tilde{J}\). If \(\P_{x_0}(S > 0) > 0\), then there exists a stopping time \(\xi\) such that \(\P_{x_0}, \Q_{x_0}\)-a.s. \(0 < \xi < S \wedge \infty\).
\end{lemma}

\begin{proof}
	First of all, since \(\{S > 0\} \in \mathcal{F}_0\), Blumenthal's zero-one law (Lemma~\ref{lem: Blumenthal}) yields that the event \(\{S > 0\}\) has
\(\P_{x_0}\)-probability zero or one and \(\Q_{x_0}\)-probability zero or one.
Consequently, by Propositon~\ref{prop: AC Sing}, \(\P_{x_0}(S > 0) = \Q_{x_0}(S > 0) = 1\). Let \(S' \triangleq S \wedge \infty\). Clearly, as \(S\) is an extended stopping time, \(S'\) is a stopping time. By Lemma \ref{lem: Meyer theo}, i.e., Meyer's theorem on predictability, any stopping time coincides \(\P_{x_0}\)-a.s. with a predictable time. 
Let \(S'_1, S'_2, \dots\) be an announcing sequence for \(S'\) under \(\P_{x_0}\), i.e., \(S'_1, S'_2, \dots\) is an increasing sequence of stopping times such that \(\P_{x_0}\)-a.s. \(S'_n \to S'\) and \(S_n' < S'\)
for all \(n \in \mathbb{N}\). There exists an \(N \in \mathbb{N}\) such that \(\P_{x_0}(S'_N > 0) > 0\), because \(\P_{x_0}\)-a.s. \(S'_n \to S' > 0\). Using again Blumenthal's zero-one law, we get that \(\P_{x_0} (S'_N > 0) = 1\). Now, \(\xi_1 \triangleq S'_N\) satisfies \(\P_{x_0}\)-a.s. \(0 < \xi_1 < S'\). In the same manner, we obtain the existence of a stopping time \(\xi_2\) such that \(\Q_{x_0}\)-a.s. \(0 < \xi_2 < S'\). Finally, set \(\xi \triangleq \xi_1 \wedge \xi_2\). We clearly have \(\P_{x_0}, \Q_{x_0}\)-a.s. \(\xi < S'\). It remains to check that also \(\P_{x_0}, \Q_{x_0}\)-a.s. \(\xi > 0\). Since \(\P_{x_0}, \Q_{x_0}\)-a.s. \(S > 0\), we get from Proposition~\ref{prop: AC Sing} that \(\P_{x_0}\sim \Q_{x_0}\) on \(\mathcal{F}_0\). Thus, since \(\{\xi_1 > 0\} \in \mathcal{F}_0\), \(\P_{x_0}(\xi_1 > 0) = 1\) implies \(\Q_{x_0}(\xi_1 > 0) = 1\), and similarly, \(\P_{x_0}(\xi_2 > 0) = 1\) follows from \(\Q_{x_0}(\xi_2 > 0) = 1\). In summary, we have \(\P_{x_0}, \Q_{x_0}\)-a.s. \(\xi> 0\) and the proof is complete.
\end{proof}

\begin{lemma} \label{lem: sing starting in A}
	Suppose that \(x_0 \in J^\circ\) and consider the following three conditions:
	\begin{enumerate}
		\item[\textup{(a)}] for every open neighborhood \(U (x_0) \subset J^\circ\) of \(x_0\) either there exists a point \(z \in \oU(x_0)\) such that
		the differential quotient
		\(d \m (z)/ d\M\) does not exist, or \(d \m / d \M\) exists on \(U (x_0)\) but not as a strictly positive continuous function;
		\item[\textup{(b)}] for every open neighborhood \(U (x_0) \subset J^\circ\) of \(x_0\) either there exists a point \(z \in \oU(x_0)\) such that
		the differential quotient
		\(d^+ \s (z)/ dx\) does not exist, or $d^+ \s/ dx$ ${(\equiv \s')}$ exists on \(U(x_0)\) but not as a strictly positive absolutely continuous function such that \(d \s' (z) = \s' (z) \beta(z) dz\) with \(\beta  \in L^2 (U(x_0))\);
		\item[\textup{(c)}] there exists an
		open neighborhood \(U (x_0) \subset J^\circ\) of \(x_0\) such that the differential quotients \(d \m/ d\M\) and \(d^+ \s/ dx\) exist on \(U(x_0)\) and for all sub-neighborhoods \(V(x_0) \subset U (x_0)\) with \(x_0 \in V (x_0)\) there exists a point \(z \in V(x_0)\) such that \(d \m (z)/ d \M \cdot d^+ \s (z) / d x \not = 1\).
	\end{enumerate}
	Then, the following equivalent statements hold:
	\begin{enumerate}
		\item[\textup{(i)}] 
			If at least one of the conditions (a)--(c) above holds, then \(\P_{x_0} \perp \Q_{x_0}\) on \(\mathcal{F}_0\).
		\item[\textup{(ii)}] 
		If \(\P_{x_0} \sim \Q_{x_0}\) on \(\mathcal{F}_0\), then all three conditions (a)--(c) above are violated.
	\end{enumerate}
\end{lemma}

\begin{proof}
	First, notice that (i) and (ii) are equivalent, because, as a consequence of Blumenthal's zero-one law (Lemma~\ref{lem: Blumenthal}), either \(\P_{x_0} \perp \Q_{x_0}\) on \(\mathcal{F}_0\) or \(\P_{x_0} \sim \Q_{x_0}\) on \(\mathcal{F}_0\). In the following, we will prove (ii),
i.e.,
we assume that \(\P_{x_0} \sim \Q_{x_0}\) on \(\mathcal{F}_0\) and prove that (a)--(c) are violated.
	
	In view of Proposition~\ref{prop: AC Sing}, this means that \(\P_{x_0}, \Q_{x_0}\)-a.s. \(S > 0\), where \(S\) denotes the separating time for \(\P_{x_0}\) and \(\Q_{x_0}\). 
	Thanks to Lemma \ref{lem: in between stopping time}, there exists a stopping time \(\xi\) such that \(\P_{x_0}, \Q_{x_0}\)-a.s. \(0 < \xi < S \wedge \infty\) and, by the definition of separating time, \(\P_{x_0} \sim \Q_{x_0}\) on~\(\mathcal{F}_\xi\).
	
	Let \(\{L_t^x(\X) \colon (t, x) \in \mathbb{R}_+ \times\mathbb R\}\)
	be the on $\mathbb R_+\times J^\circ$ jointly continuous
	\(\Q_{x_0}\)-modification of the
	semimartingale
	local time of \(\X\) and let \(\{L_t^x (\s(\X)) \colon (t, x) \in \mathbb{R}_+ \times\mathbb R\}\) be the on $\mathbb R_+\times\s(J^\circ)$ jointly continuous
	\(\P_{x_0}\)-modification of the
	semimartingale
	local time of \(\s(\X)\), see Lemmata \ref{lem: diff homo} and \ref{lem: occ formula diff}
	and Remark~\ref{rem:241124a1}~(a).
	We define \(G \triangleq G_1 \cap G_2 \cap G_3\), where
	\begin{align*}
	G_1 &\triangleq \Big\{0 < \xi < \infty,\ L_{\xi}^{x_0} (\X), L_{\xi}^{\s(x_0)} (\s(\X)) > 0\Big\},
	\\
	G_2 &\triangleq \Big\{0 < \xi < \infty,\ \int_0^\xi f (\X_{s}) ds = \int f(x) L_\xi^x (\X) \M (dx) 
	\ \forall \, f \in \mathcal{B}^+_0\Big\},\\
	G_3 &\triangleq \Big\{0 < \xi < \infty,\ \int_0^\xi f (\X_{s}) ds = \int f(x) L_\xi^{\s(x)} (\s(\X)) \m (dx)  
	\ \forall \, f \in \mathcal{B}^+_0\Big\}
	\end{align*}
and \(\mathcal{B}^+_0 \triangleq \{ f \colon \on{cl} (J) \to \mathbb{R}_+, \, \text{Borel}, \, f |_{\partial J} = 0 \}\).
By Lemmata \ref{lem: occ formula diff} and~\ref{lem: pos LT}
and Remark~\ref{rem:241124a1}~(b),
the equivalence \(\P_{x_0}\sim\Q_{x_0}\) on \(\mathcal{F}_\xi\) and the fact that \(\P_{x_0}, \Q_{x_0}\)-a.s. \(0<\xi<\infty\), we have \(\P_{x_0} (G) = \Q_{x_0} (G) = 1\), which in particular implies that \(G \not = \emptyset\). We fix some \(\omega \in G\). As the functions
	\[x \mapsto L^x_{\xi (\omega)} (\X) (\omega) \quad \text{and} \quad x \mapsto L^{\s(x)}_{\xi (\omega)} (\s(\X)) (\omega)\] 
	are continuous on $J^\circ$,
	using the definition of \(G_1\), we can find an open neighborhood \(V(x_0)\) of \(x_0\) such that 
	\[
	L^x_{\xi (\omega)} (\X) (\omega), L^{\s(x)}_{\xi (\omega)} (\s(\X)) (\omega) > 0 \ \ \forall \, x \in V(x_0).
	\]
	Using the definition of \(G_2\) and \(G_3\), for every \(z \in V(x_0)\), we get, as \(h \searrow 0\),
	\begin{align*}
	\frac{\int_0^{\xi (\omega)} \1 \{z - h < \X_{s} (\omega) < z + h\} ds}{\M ((z - h, z + h))} &= \frac{\int_{z - h}^{z + h} L_{\xi (\omega)}^x (\X) (\omega) \M (dx)}{\M((z - h, z + h))} \to L_{\xi (\omega)}^z (\X) (\omega) > 0, 
	\end{align*}
	and 
	\begin{align*}
	\frac{\int_0^{\xi (\omega)} \1 \{z - h < \X_{s} (\omega) < z + h\} ds}{\m ((z - h, z + h))} 
	&= \frac{\int_{z - h}^{z + h} L_{\xi (\omega)}^{\s(x)} (\s(\X)) (\omega) \m(dx)}{\m((z - h, z + h))} 
	\to L_{\xi(\omega)}^{\s(z)} (\s(\X)) (\omega) > 0,
	\end{align*}
	which shows that 
	\begin{align} \label{eq: speed measure density local time}
	\Big(\frac{d \m}{d \M}\Big) (z) = \frac{L^z_{\xi (\omega)} (\X) (\omega) }{L^{\s(z)}_{\xi(\omega)} (\s (\X)) (\omega)} > 0.
	\end{align}
	Thanks to the continuity of the function 
	\[
	V(x_0) \ni z \mapsto \frac{L^z_{\xi (\omega)} (\X) (\omega) }{L^{\s(z)}_{\xi(\omega)} (\s (\X)) (\omega)},
	\]
	we conclude that \(d \m /d \M\) exists as a strictly positive continuous function on \(V(x_0)\), i.e., (a)~is violated.

Next, we show that (b) is violated, too.
Let \(\Po_{x_0}\) and \(\Qo_{x_0}\) be as in the first step of the proof of Lemma~\ref{lem: loc equivalence} with \((a', c') = V(x_0)\).
By virtue of Lemma~\ref{lem: time change singularity} and Blumenthal's zero-one law,
\(\P_{x_0} \sim \Q_{x_0}\) on \(\mathcal{F}_0\) implies that \(\Po_{x_0} \sim \Qo_{x_0}\) on \(\mathcal{F}_0\).
Therefore, to simplify our notation, we can and will w.l.o.g.
assume that \(\Q_{x_0}\) is the law of a Brownian motion stopped at the boundaries of an open neighborhood \(V(x_0)\) of \(x_0\) and that \(\P_{x_0}\) is the law of a diffusion with scale function \(\s\) and speed measure
$$
\frac{\mo(dx)}{dx} \triangleq \Big(\frac{d \m}{d \M} \Big)(x)
$$
also stopped at the boundaries of \(V(x_0)\). 
These assumptions are in force for the remainder of this proof. 
By Lemma~\ref{lem: scale fct}, \(Y\ \triangleq \s (\X_{\cdot \wedge \xi})\) is a \(\P_{x_0}\)-semimartingale (in fact, \(Y\) is even a local \(\P_{x_0}\)-martingale). Consequently, as \(\P_{x_0} \sim \Q_{x_0}\) on \(\mathcal{F}_\xi\), Girsanov's theorem (Lemma~\ref{lem: Girs}) yields that \(Y\) is also a \(\Q_{x_0}\)-semimartingale.
	As \(\Q_{x_0}\) is the law of a Brownian motion with absorbing boundaries, thanks to Theorem~\ref{lem: diff convex}, possibly making \(V(x_0)\) a bit smaller,  we get that \(\s\) is the difference of two convex functions on \(V(x_0)\). In particular, \(d^+ \s /dx \equiv \s'\) exists as a right-continuous function of (locally) finite variation (\cite[Proposition~5.1]{CinJPrSha}).
	Let $a<c$ denote the boundary points of $V(x_0)$.
	Take \(a < a' < x_0 < c' < c\) and set \(T \triangleq T_{a'} \wedge T_{c'}\).
	Since \(\P_{x_0} \sim \Q_{x_0}\) on \(\mathcal{F}_\xi\) and because \(\Q_{x_0}\) is the law of a Brownian motion stopped at the boundaries of \(V(x_0)\), Girsanov's theorem (Lemma~\ref{lem: Girs}) shows that under \(\P_{x_0}\) the stopped process \(\X_{\cdot \wedge \xi \wedge T}\) is a Brownian motion stopped at \(\xi \wedge T\) with additional drift \(\int_0^{\cdot \wedge \xi \wedge T} \tilde{\beta}_s ds\) for a predictable process \(\tilde{\beta}\) such that \(\P_{x_0}, \Q_{x_0}\)-a.s.\
	\[\int_0^{t\wedge\xi\wedge T} \tilde{\beta}^2_s ds < \infty\] for all $t\in\mathbb R_+$.
	In other words, under \(\P_{x_0}\), we have, informally,
	\[
	\X_{\cdot \wedge \xi \wedge T} = x_0 + \int_0^{\cdot \wedge \xi \wedge T} \tilde{\beta}_s ds + W_{\cdot \wedge \xi \wedge T},
	\]
	where \(W\) is a standard Brownian motion.
	Let \(\{L^x_t (\X) \colon (t, x) \in \mathbb{R}_+ \times \mathbb{R}\}\) be a jointly continuous modification\footnote{Notice that \(\X\) is a continuous \(\Q_{x_0}\)-martingale, as it is a stopped Brownian motion under \(\Q_{x_0}\), and recall \cite[Theorem~VI.1.7]{RY}.} of the
	semimartingale
	local time of \(\X\) under \(\Q_{x_0}\).
	Then, as \(\s\) is the difference of two convex functions on \([a', c'] \subset (a, c)\), the generalized It\^o formula (Lemma~\ref{lem: occ smg}
		and\footnote{Recall that here $\s'$ denotes the \emph{right-hand} derivative, so we need the left-continuous (in the space variable) local time process in the generalized It\^o formula.
But in the present context the local time process is jointly continuous.}
Remark~\ref{rem:021022a1}) yields that \(\Q_{x_0}\)-a.s. 
	\[
	\s (\X_{\cdot \wedge \xi \wedge T}) = \s(x_0) + \int_0^{\cdot \wedge \xi \wedge T} \s' (\X_s) d \X_s + \frac{1}{2} \int L^x_{\cdot \wedge \xi \wedge T} (\X) \s'' (dx),
	\]
	where \(\s''(dx)\) denotes the
	(signed)
	second derivative measure of \(\s\) defined
	on $(a',c']$
	by \(\s'' ((x, y]) = \s'(y) - \s'(x)\) for all \(a' \leq x \leq y \leq c'\).
	Of course, by the equivalence of \(\P_{x_0}\) and \(\Q_{x_0}\) on \(\mathcal{F}_\xi\), this equality also holds under \(\P_{x_0}\).
	Thus, under \(\P_{x_0}\), we get
	\begin{align*}
	\s (\X_{\cdot \wedge \xi \wedge T}) &- \frac{1}{2} \int L^x_{\cdot \wedge \xi \wedge T} (\X) \s'' (dx) - \int_0^{\cdot \wedge \xi \wedge T} \s' (\X_s) \tilde{\beta}_s ds 
	\\&= \s (x_0) + \int_0^{\cdot \wedge \xi \wedge T} \s' (\X_s) d \Big(\X_s - \int_0^s \tilde{\beta}_u du \Big) 
	= \text{local \(\P_{x_0}\)-martingale}.
	\end{align*}
	By Lemma \ref{lem: scale fct}, the stopped process \(\s (\X_{\cdot \wedge \xi \wedge T})\) is a local \(\P_{x_0}\)-martingale, too. Hence, using the fact that continuous local martingales of (locally) finite variation are constant, we obtain that \(\P_{x_0}\)-a.s., and hence also \(\Q_{x_0}\)-a.s.,
	\begin{align}\label{eq: drift condition}
	\int L^x_{\cdot \wedge \xi \wedge T} (\X) \s'' (dx) = - \int_0^{\cdot \wedge \xi \wedge T} 2 \s'(\X_s) \tilde{\beta}_s ds.
	\end{align}
	Computing the variation of both sides (see \cite[pp. 915--916]{MU15} for a detailed computation of the variation of the first integral) yields that 
	\(\Q_{x_0}\)-a.s.
	\[
	\int L^x_{\cdot \wedge \xi \wedge T} (\X) |\s''| (dx) = \int_0^{\cdot \wedge \xi \wedge T} \big|2 \s' (\X_s) \tilde{\beta}_s\big| ds,
	\]
	where $|\s''|$ denotes the variation measure of the measure $\s''$.
	Let \(\mathcal{N}\) be the collection of all Lebesgue null sets in \(\mathcal{B}((a', c'))\). The occupation time formula (Lemma~\ref{lem: occ smg}) yields that \(\Q_{x_0}\)-a.s., for all \(G \in \mathcal{N}\), \(\int_0^{\cdot \wedge \xi \wedge T} \1_G (\X_s) ds = 0\)
	and hence,
	\[
	\int_0^{\cdot \wedge \xi \wedge T} \1_G (\X_s) d \Big(\int_0^s \big| 2 \s' (\X_r) \tilde{\beta}_r \big| dr \Big) = \int_0^{\cdot \wedge \xi \wedge T} \1_G (\X_s) \big| 2 \s'(\X_s) \tilde{\beta}_s\big| ds = 0.
	\]
	Moreover, we compute that \(\Q_{x_0}\)-a.s. for all \(G \in \mathcal{N}\)
	\begin{align*}
	\int_0^{\cdot \wedge \xi \wedge T} \1_G (\X_s) d \Big( \int L^x_s (\X) |\s''| (dx) \Big) 
	&= 
	\int \Big( \int_0^{\cdot \wedge \xi \wedge T} \1_G (\X_s) d_s L^x_s (\X) \Big) |\s''| (dx) 
	\\&= \int \Big( \int_0^{\cdot \wedge \xi \wedge T} \1_G (\X_s) \1_{\{\X_s = x\}} d_s L^x_s (\X) \Big) |\s''| (dx) 
	\\&= \int_G L_{\cdot \wedge \xi \wedge T}^x (\X) |\s''|(dx),
	\end{align*}
	where we use Lemma \ref{lem: occ smg} for the last two lines.
	Together, we get \(\Q_{x_0}\)-a.s. \(\int_G L^x_{\cdot \wedge \xi \wedge T} (\X) |\s''| (dx) = 0\) for all \(G \in \mathcal{N}.\)
	By the above observations, Lemma~\ref{lem: pos LT} and the fact that \(\Q_{x_0}\)-a.s. \(\xi \in (0, \infty)\), for \(\Q_{x_0}\)-a.a. \(\omega \in \Omega\), we have
	\begin{align} \label{eq: omega cond}
	\begin{cases} &0 < \xi (\omega) < \infty, \\ &L^{x_0}_{\xi (\omega) \wedge T (\omega)} (\X) (\omega) > 0, \\ &\int_G L^{x}_{\xi (\omega) \wedge T (\omega)} (\X) (\omega) |\s''| (dx) = 0\ \ \forall\ G \in \mathcal{N}. \end{cases}
	\end{align}
	Take an \(\omega \in \Omega\) such that \eqref{eq: omega cond} holds.
	As \(x \mapsto L^x_{\xi (\omega) \wedge T (\omega)} (\X) (\omega)\) is continuous, there exists an open neighborhood \(V^\circ(x_0) \subset (a', c')\) of \(x_0\) such that 
	\(L^{x}_{\xi (\omega) \wedge T (\omega)} (\X) (\omega) > 0\) for all \(x \in V^\circ(x_0).\)
	Together with the third part of~\eqref{eq: omega cond}, we conclude that \(\s''(dx) \ll dx\) on \(\mathcal{B}(V^\circ(x_0))\).
	Thus, $\s'$ is an absolutely continuous function on $V^\circ(x_0)$.
		Let \(\{L^x_t (\s(\X)) \colon (t, x) \in \mathbb{R}_+ \times \mathbb{R}\}\) be a jointly continuous modification\footnote{Notice that \(\s(\X)\) is a continuous local
			\(\P_{x_0}\)-martingale, as it is stopped at the boundaries of $V(x_0)$ (Lemma~\ref{lem: scale fct}), and recall \cite[Theorem VI.1.7]{RY}.} of the
		semimartingale
		local time of \(\s(\X)\) under \(\P_{x_0}\).
		Then, $\P_{x_0}$-a.s.\ for all $(t,z)\in\mathbb R_+\times J$ we have
		\begin{equation}\label{eq:071022a1}
		L_{t\wedge\xi\wedge T}^{\s(z)}(\s(\X))
		=
		L_t^{\s(z)}(\s(\X_{\cdot\wedge\xi\wedge T})).
		\end{equation}
		As the local time process is preserved under an absolutely continuous measure change
		(which is a consequence of \cite[Corollary VI.1.9]{RY})
		and $\P_{x_0}\sim\Q_{x_0}$ on $\mathcal F_\xi$,
		it follows that \eqref{eq:071022a1} holds also $\Q_{x_0}$-a.s.\ for all $(t,z)\in\mathbb R_+\times J$.
		Further, $\Q_{x_0}$-a.s.\ for all $(t,z)\in\mathbb R_+\times (a',c')$ we have
		\begin{equation}\label{eq:071022a2}
		L_t^{\s(z)}(\s(\X_{\cdot\wedge\xi\wedge T}))
		=
		\s'(z) L_t^z(\X_{\cdot\wedge\xi\wedge T})
		=
		\s'(z) L_{t\wedge\xi\wedge T}^z(\X),
		\end{equation}
		where in the first equality we use part~(iii) of Lemma~\ref{lem: occ smg}
		(recall that $\s$ is the difference of two convex functions on $V(x_0)=(a,c)$, that $a<a'<x_0<c'<c$ and that $T=T_{a'}\wedge T_{c'}$).
		In particular, \eqref{eq:071022a1} and~\eqref{eq:071022a2} imply that $\Q_{x_0}$-a.s.\ for all $z\in(a',c')$
		\begin{equation}\label{eq:071022a3}
		L_{\xi\wedge T}^{\s(z)}(\s(\X))
		=
		\s'(z) L_{\xi\wedge T}^z(\X).
		\end{equation}
By Lemma~\ref{lem: pos LT}, the equivalence $\P_{x_0}\sim\Q_{x_0}$ on $\mathcal F_\xi$
and the continuity of the local times in the space variable,
		we can take an $\omega\in\Omega$ such that \eqref{eq:071022a3} holds for all $z\in(a',c')$
		and the functions
		$$
		z\mapsto L_{\xi(\omega)\wedge T(\omega)}^z(\X)(\omega)
		\quad\text{and}\quad
		z\mapsto L_{\xi(\omega)\wedge T(\omega)}^{\s(z)}(\s(\X))(\omega)
		$$
		are strictly positive in a neighborhood of $x_0$.
		It follows that $\s'>0$ in a sufficiently small open neighborhood $V^*(x_0)\subset(a',c')$ of $x_0$.
		We now define the open neighborhood $U(x_0)\triangleq V^\circ(x_0)\cap V^*(x_0)$ of $x_0$ and observe that $\s'$ is a strictly positive absolutely continuous function on $U(x_0)$.
		Let $\zeta\colon U(x_0)\to\mathbb R$ be a Borel function such that
		\(\s'' (dx) = \zeta (x) dx\) on \(\mathcal{B}(U(x_0))\)
		(in other words, $\zeta$ equals the second derivative $\s''$ almost everywhere on \(\mathcal{B}(U(x_0))\) w.r.t.\ the Lebesgue measure).
	Set
	\begin{align}\label{eq: def H}
	H \triangleq \inf (t \geq 0 \colon \X_t \not \in U(x_0)).
	\end{align}
	Recalling~\eqref{eq: drift condition} and using the occupation time formula (Lemma~\ref{lem: occ smg}), we get \(\Q_{x_0}\)-a.s.
	\begin{align*}
	\int_0^{\cdot \wedge \xi \wedge H} \zeta (\X_s) ds &= \int L_{\cdot \wedge \xi \wedge H}^x(\X) \zeta (x) dx 
	= \int L_{\cdot \wedge \xi \wedge H}^x(\X) \s'' (dx) 
	= - \int_0^{\cdot \wedge \xi \wedge H} 2\s' (\X_s) \tilde{\beta}_s ds, 
	\end{align*}
	which implies that \(\Q_{x_0}\)-a.s.\ \(\tilde{\beta}_s = -  \zeta (\X_s) / (2 \s' (\X_s)) \equiv - \beta (\X_s)/2\) for \(\mu_L\)-a.a.\ \(s \leq \xi \wedge H\)
	with $\beta=\zeta/\s'$ (cf.\ the formulation of part~(b)).
	Thanks to the
	(mentioned above)
	square integrability of \(\tilde{\beta}\), which follows from Girsanov's theorem, we obtain
	\[
	\Q_{x_0} \Big(\int_0^{t \wedge \xi \wedge H} \big(\beta (\X_s)\big)^2 ds < \infty\Big) = 1,\quad t\in\mathbb R_+.
	\]
	Now, Lemma~\ref{lem:090922a1} implies that 
	the function $\beta$
	is square integrable in a neighborhood of~\(x_0\). All in all, we proved that (b) does not hold.

Finally, for the last part, we notice that formulas \eqref{eq: speed measure density local time} and~\eqref{eq:071022a3} entail that~(c) is violated
(more precisely, use~\eqref{eq: speed measure density local time} with $\xi$ replaced by $\xi\wedge T$).
The proof is complete.
\end{proof}

For the following lemma we use the notation $\of b,c\gs$, which is defined similar to $\ofr b,c\gsr$ and $\of b,c\gsr$ from \eqref{eq:271022a0} and~\eqref{eq:271022a1}.

\begin{lemma} \label{lem: singular boundary separating}
	Suppose that \(b \in \partial J\) is an accessible boundary for both \((x \mapsto \P_{x})\) and \((x \mapsto \Q_{x})\), and that there exists a non-empty open interval
\(B\subsetneq J^\circ\)
with \(b\) as endpoint such that all points in \(B\) are non-separating. Furthermore, assume \(\m (\{b\}), \M(\{b\}) < \infty\). 

Consider the following three conditions:
		\begin{enumerate}
		\item[\textup{(a)}] for every \(c \in B\) the differential quotient
		\(d \m / d\M\) does not exist as a continuous function from \(\of b, c\gs\) into \((0, \infty)\);
		\item[\textup{(b)}] for every \(c \in B\) the differential quotient $d^+ \s/ dx$ ${(\equiv \s')}$ does not exist as an absolutely continuous function from \(\of b, c \gs\) into \((0, \infty)\) such that \(d \s' (z) = \s' (z) \beta(z) dz\) with \(\beta  \in L^2 (\of b, c \gs)\);
		\item[\textup{(c)}] the differential quotients \(d \m (b)/ d\M\) and \(d^+ \s (b)/ dx\) exist but
		\[
				\Big(\frac{d \m}{d \M} \Big) (b) \Big( \frac{d^+ \s}{dx}\Big) (b) \not = 1.
				\]
	\end{enumerate}
		Then, the following equivalent statements hold:
		\begin{enumerate}
			\item[\textup{(i)}] If at least one of the conditions (a)--(c) above holds, then \(\P_b \perp \Q_b\) on \(\mathcal{F}_0\).
			\item[\textup{(ii)}] If \(\P_b \sim \Q_b\) on \(\mathcal{F}_0\), then all three conditions (a)--(c) above are violated.
		\end{enumerate}
\end{lemma}

Related to the assumption of Lemma~\ref{lem: singular boundary separating}, we remark the following:
as $\S=\on{Id}$, the accessibility of $b$ implies that $b\in\mathbb R$
(see~\eqref{eq:101022a1}).

\begin{proof}
As in the proof of Lemma~\ref{lem: sing starting in A}, (i) and~(ii) are equivalent by Blumenthal's zero-one law (Lemma~\ref{lem: Blumenthal}). In the following, we prove~(ii), i.e., we assume that \(\P_b \sim \Q_b\) on \(\mathcal{F}_0\). Equivalently, this means that \(\P_b, \Q_b\)-a.s. \(S > 0\), where \(S\) is the separating time of \(\P_b\) and \(\Q_b\).

Thanks to Lemma \ref{lem: in between stopping time}, there exists a stopping time \(\xi\) such that \(\P_b, \Q_b\)-a.s. \(0 < \xi < S \wedge \infty\).
By the definition of separating time, $\P_b\sim\Q_b$ on $\mathcal F_\xi$.
Take $c\in B$.
Let \(L(\X)=\{L_t^x(\X) \colon (t, x) \in \mathbb{R}_+ \times\mathbb R\}\)
(resp., \(L(\s(\X))=\{L_t^x (\s(\X)) \colon (t, x) \in \mathbb{R}_+ \times \mathbb R\}\))
be the semimartingale local time of $\X$ under $\Q_b$
(resp., of $\s(\X)$ under $\P_b$),
see Lemma~\ref{lem: diff homo} and Lemma~\ref{lem: occ formula diff}~(i).
Notice that, if $b$ is the left (resp., right) boundary point of $B$, then, $\Q_b$-a.s.,
$L(\X)$ (resp., $L^-(\X)\triangleq\{L_t^{x-}(\X):(t,x)\in\mathbb R_+\times\mathbb R\}$)
is jointly continuous for $(t,x)\in\mathbb R_+\times\of b,c\gs$,
see Lemma~\ref{lem: occ formula diff}~(ii) and Remark~\ref{rem:241124a1}~(a).
Similarly, $\P_b$-a.s., $L(\s(\X))$ or
$L^-(\s(\X))\triangleq\{L_t^{x-}(\s(\X)):(t,x)\in\mathbb R_+\times\mathbb R\}$
is jointly continuous on $\mathbb R_+\times\of b,c\gs$ depending on whether $b<c$ or $b>c$.
For notational convenience only, below we assume w.l.o.g.\ that \(b = 0\), \(c \in (0, \infty)\) and \(\s (0) = 0\).\footnote{This is the case $b<c$. In the opposite case $b>c$ we need to work with $L^-(\X)$ in place of $L(\X)$ and $L^-(\s(\X))$ in place of $L(\s(\X))$.}
	We define \(G \triangleq G_1 \cap G_2 \cap G_3\) by
	\begin{align*}
	G_1 &\triangleq \Big\{0 < \xi < \infty,\ L_{\xi}^{0} (\X), L_{\xi}^{0} (\s(\X)) > 0\Big\},
	\\
	G_2 &\triangleq \Big\{0 < \xi < \infty,\ \int_0^\xi f (\X_{s}) ds = \int f(x) L_\xi^x (\X) \M (dx) 
	\ \forall \, f \in \mathcal{B}^+_0
	\Big\},\\
	G_3 &\triangleq \Big\{0 < \xi < \infty,\ \int_0^\xi f (\X_{s}) ds = \int f(x) L_\xi^{\s(x)} (\s(\X)) \m (dx) 
	\ \forall \, f \in \mathcal{B}^+_0
	\Big\},
	\end{align*}	
	with \(\mathcal{B}^+_0 = \{f \colon \mathbb R_+ \to \mathbb{R}_+, \text{Borel}, f |_{[c,\infty)} = 0\}\).
	By Lemmata \ref{lem: occ formula diff} and~\ref{lem: pos LT} and Remark~\ref{rem:241124a1}~(b), the equivalence \(\P_{0}\sim\Q_{0}\) on \(\mathcal{F}_\xi\) and the fact that \(\P_{0}, \Q_{0}\)-a.s. \(\xi \in (0, \infty)\), we have \(\P_{0} (G) = \Q_{0} (G) = 1\). Take \(\omega \in G\). By the space continuity of the local time processes and the definition of the set \(G_1\), there exists a number \(c^\circ = c^\circ (\omega) \in  (0, c)\) such that \(L^z_{\xi (\omega)} (\X) (\omega), L^z_{\xi (\omega)} (\s (\X)) (\omega) > 0\) for all \(z \in [0, c^\circ]\). 
	Using the definition of \(G_2\) and \(G_3\), for every \(z \in [0, c^\circ]\), we get
	\begin{align*}
	\lim_{h \searrow 0} \frac{\int_0^{\xi (\omega)} \1 \{ z \leq \X_{s} (\omega) < z + h\} ds}{\M ([z, z + h))} &= L_{\xi (\omega)}^z (\X) (\omega) > 0, 
	\end{align*}
	and 
	\begin{align*}
	\lim_{h \searrow 0} \frac{\int_0^{\xi (\omega)} \1 \{ z \leq \X_{s} (\omega) < z + h\} ds}{\m ([z, z + h))}  =  L_{\xi(\omega)}^{\s(z)} (\s(\X)) (\omega) > 0,
	\end{align*}
	which shows that 
	\begin{align} \label{eq: speed measure density local time 2}
	\Big(\frac{d \m}{d \M}\Big) (z) = \frac{L^z_{\xi (\omega)} (\X) (\omega) }{L^{\s(z)}_{\xi(\omega)} (\s (\X)) (\omega)} > 0.
	\end{align}
	Thanks to the continuity of the function 
	\[
	[0, c^\circ] \ni z \mapsto \frac{L^z_{\xi (\omega)} (\X) (\omega) }{L^{\s(z)}_{\xi(\omega)} (\s (\X)) (\omega)},
	\]
	we conclude that \(d \m /d \M\) exists as a continuous function from \([0, c^\circ]\) into \((0, \infty)\). This means that (a) has to be violated.
	
Next, we show that (b) is violated, too. 
	Let \(([0, c] \ni x \mapsto \Qo_x)\) be a diffusion on natural scale with speed measure 
\[
\M^\circ (dx) \triangleq dx \text{ on } (0, c), \quad \M^\circ (\{0\}) \triangleq 0, \quad \M^\circ (\{c\}) \triangleq \infty,
\]
i.e.,  \(([0, c] \ni x \mapsto \Qo_x)\) is a Brownian motion which is instantaneously reflected from \(0\) and absorbed in \(c\).
Moreover, let \(([0, c] \ni x \mapsto \Po_x)\) be a diffusion with characteristics \((\s, \mo)\), where 
\[
\frac{\m^\circ (dx)}{dx} \triangleq \Big( \frac{d \m}{d \M}\Big) (x) \text{ on } (0, c), \quad \m^\circ (\{0\}) \triangleq 0, \quad \m^\circ (\{c\}) \triangleq \infty.
\]

\begin{lemma} \label{lem: time change singularity2}
If \(\Po_{0} \perp \Qo_{0}\) on \(\mathcal{F}_0\), then \(\P_{0} \perp \Q_{0}\) on \(\mathcal{F}_0\).
Equivalently, if \(\P_{0} \sim \Q_{0}\) on \(\mathcal{F}_0\), then \(\Po_{0} \sim \Qo_{0}\) on \(\mathcal{F}_0\).
\end{lemma}

\begin{proof}
The equivalence between the two claims follows from Blumenthal's zero-one law (Lemma~\ref{lem: Blumenthal}). We prove the first claim.
	For \(t \in \mathbb{R}_+\), define
	\begin{align*} 
	\tilde{\l} (t, x) \triangleq 
	\limsup_{h \searrow 0}
	\frac{\int_0^t \1 \{x - h < \X_s < x + h\} ds}{\M ( (x - h, x + h))}, \quad x \in (0, c),
	\end{align*}
	\begin{align*}
	\f (t) & \triangleq \begin{cases} \int_{ (0, c) } \tilde{\l} (t, x) dx, & t < T_c,\\ \infty, & t \geq T_c, \end{cases}
	\end{align*}
	and 
\(
	\f^{-1} (t) \triangleq \inf (s \geq 0 \colon \f(s+) > t).
\)
	The chain rule for diffusions as given by Theorem~\ref{theo: chain rule} yields that 
	\begin{align*} 
	\Po_{0} = \P_{0} \circ \X_{\f^{-1}(\cdot)}^{-1},\qquad \Qo_{0} = \Q_{0} \circ \X_{\f^{-1}(\cdot)}^{-1}.
	\end{align*}
From this point on, we can argue verbatim as in the proof of Lemma \ref{lem: time change singularity} to obtain the
first
claim of the lemma. We omit the details. 
\end{proof}
Thanks to Lemma~\ref{lem: time change singularity2}, we can and will assume
(in the part of the proof, where we show that (b) is violated)
that \(\Q_{0} = \Qo_0\) and that \(\P_{0} = \Po_0\).
These assumptions are in force for the remainder of this proof. 
Define 
\[
\s^\leftrightarrow (x) \triangleq \on{sgn} (x) \s (|x|), \quad x \in [- c, c], 
\]
and 
\[
\m^\leftrightarrow (A) \triangleq \begin{cases} \m (A), & A \in \mathcal{B} ( (0, c) ), \\
\m (- A), & A \in \mathcal{B}( (- c, 0)), \\
0, & A = \{0\}, \\
+ \infty, & A \subset \{-c , c\},\;A\ne\emptyset.
\end{cases}
\]
Let \(([- c, c] \ni x \mapsto \P^\leftrightarrow_x)\) be the diffusion with characteristics \((\s^\leftrightarrow, \m^\leftrightarrow)\). Furthermore, let \(([- c, c] \ni x \mapsto \Q^\leftrightarrow_x)\) be a Brownian motion with absorbing boundaries \(-c\) and \(c\). 
By a version of Lemma~\ref{lem: refl} (cf. \cite[Section~6]{ankirchner}), we have
\begin{align*} 
\P_{0} = \P^\leftrightarrow_{0} \circ Y^{-1}, \quad \Q_0 = \Q^\leftrightarrow_0 \circ Y^{-1}, \quad Y \triangleq |\X|.
\end{align*}
Notice that \(\P_0 \sim \Q_0\) on \(\mathcal{F}_\xi\) implies that \(\P^\leftrightarrow_0 \sim \Q^\leftrightarrow_0\) on \(Y^{-1}(\mathcal{F}_\xi) = \mathcal{F}^Y_{\rho}\) with \(\rho = \xi \circ Y\), see Lemma~\ref{lem: sigma field stopped id} for the identity of the \(\sigma\)-fields and the fact that \(\rho\) is an \((\mathcal{F}^Y_t)_{t \geq 0}\)-stopping time.
Let \((\mathbb{R} \ni x \mapsto \W_x)\) be the Wiener measure, take \(0 < c^* < c\) and set 
\[
T \triangleq T_{-c} \wedge T_c = T_c \circ Y, \qquad T^* \triangleq T_{-c^*} \wedge T_{c^*} = T_{c^*} \circ Y.
\]  
Then, \(\mathds{W}_0 \circ \X_{\cdot \wedge T}^{-1} = \Q^\leftrightarrow_0\) and, 
for every \(A \in \mathcal{F}^Y_{\rho \wedge T^*}\), Galmarino's test
in the form \cite[10~c), p.~87]{itokean74} (alternatively, one can adjust \cite[Exercise~I.4.21]{RY} to the right-continuous filtration $(\mathcal F_t)_{t \geq 0}$)
yields that
\[
\W_0 (A) = \W_0 (\X_{\cdot \wedge T} \in A) = \Q^\leftrightarrow_0 (A).
\]
Hence, \(\P_0^\leftrightarrow \sim \W_0\) on \(\mathcal{F}^Y_{\rho \wedge T^*}\). 
We deduce from Lemmata~\ref{lem: scale fct} and \ref{lem: occ smg} that the process \(\s (Y) = |\s^\leftrightarrow (\X)|\) is a \(\P^\leftrightarrow_0\)-\((\mathcal{F}_t)_{t \geq 0}\)-semimartingale (since it is the absolute value of a semimartingale). Thus, by Stricker's theorem (see, e.g., \cite[Theorem~9.19]{Jacod}), the process \(\s (Y)\) is also a \(\P^\leftrightarrow_0\)-\((\mathcal{F}^{Y}_t)_{t \geq 0}\)-semimartingale and, by Girsanov's theorem (Lemma~\ref{lem: Girs}), as  \(\P_0^\leftrightarrow \sim \W_0\) on \(\mathcal{F}^Y_{\rho \wedge T^*}\), the stopped process \(\s (Y_{\cdot \wedge \rho \wedge T^*})\) is a \(\W_{0}\)-\((\mathcal{F}^Y_t)_{t \geq 0}\)-semimartingale.

\begin{lemma} \label{lem: smg bigger filtration}
\(\s (Y_{\cdot \wedge \rho \wedge T^*}) = \s (|\X_{\cdot \wedge \rho \wedge T^*}|)\) is also a \(\W_{0}\)-\((\mathcal{F}_t)_{t \geq 0}\)-semimartingale.
\end{lemma}
\begin{proof}
	Essentially, the claim follows from \cite[(3.24.i)]{CinJPrSha} and \cite[Exericse 6.23]{LeGall}. We provide the details.
	Thanks to \cite[Exericse 6.23]{LeGall}, \(Y\) is an \((\mathcal{F}_t)_{t \geq 0}\)-Markov process under \(\W_0\). Hence, we get for all \(s_1 < s_2 < \dots < s_n \leq t < t_1 < t_2 < \dots < t_m\) and every bounded Borel functions \(f \colon \mathbb{R}^n \to \mathbb{R}\) and \(g \colon \mathbb{R}^m \to \mathbb{R}\), that~\(\W_0\)-a.s.
	\begin{align*}
	\E^{\W_0} \big[ f (Y_{s_1}&, \dots, Y_{s_n}) g (Y_{t_1}, \dots, Y_{t_m}) | \mathcal{F}_t \big] 
	= f (Y_{s_1}, \dots, Y_{s_n}) \E^{\W_0} \big[  g (Y_{t_1}, \dots, Y_{t_m}) | Y_t \big].
	\end{align*}
	By the tower property, this yields that \(\W_0\)-a.s.
	\begin{align*}
	\E^{\W_0} \big[ f (Y_{s_1}, \dots, Y_{s_n}) g (Y_{t_1}, \dots, Y_{t_m}) | \mathcal{F}^Y_t \big]
	&= \E^{\W_0} \big[ \E^{\W_0} \big[ f (Y_{s_1}, \dots, Y_{s_n}) g (Y_{t_1}, \dots, Y_{t_m}) | \mathcal{F}_t \big] | \mathcal{F}^Y_t\big]
	\\&= f (Y_{s_1}, \dots, Y_{s_n}) \E^{\W_0} \big[  g (Y_{t_1}, \dots, Y_{t_m}) | Y_t \big]
	\\&= \E^{\W_0} \big[ f (Y_{s_1}, \dots, Y_{s_n}) g (Y_{t_1}, \dots, Y_{t_m}) | \mathcal{F}_t \big].
	\end{align*}
	Therefore, by a monotone class argument, for every \(\W_0\)-integrable \(\mathcal{F}^Y_\infty\)-measurable random variable \(Z\) and any \(t \in \mathbb{R}_+\), we have \(\W_0\)-a.s.
\[
\E^{\W_0} \big[ Z | \mathcal{F}_t \big] = \E^{\W_0} \big[ Z | \mathcal{F}^Y_t \big].
\]
This implies that any \(\W_0\)-\((\mathcal{F}^Y_t)_{t \geq 0}\)-martingale is also a \(\W_0\)-\((\mathcal{F}_t)_{t \geq 0}\)-martingale.
Finally, by \cite[Proposition 9.28]{Jacod}, this implies that the same implication also holds for semimartingales and hence, the claim follows. 
\end{proof}

Since \(\P_0^\leftrightarrow \sim \W_0\) on \(\mathcal{F}^Y_{\rho \wedge T^*}\) and \(\P_0^\leftrightarrow (\rho \wedge T^* > 0) = \P_0 (\xi \wedge T_{-c^*} \wedge T_{c^*} > 0) = 1\), we have \(\W_0 (\rho \wedge T^* > 0) = 1\).
Now, we deduce from Lemma \ref{lem: smg bigger filtration} and Theorem~\ref{lem: diff convex} that \(\s (| \cdot |)\) is the difference of two convex functions on a closed interval \([-c', c']\) with \(0 < c' < c^* < c\). 
Set 
\(
T' \triangleq T_{-c'} \wedge T_{c'},
\)
and, for \(t \in \mathbb{R}_+\), let \(\mathcal{G}^{\W_0}_t\) be the \(\W_0\)-completion of \(\mathcal{F}^Y_{t \wedge \rho \wedge T'}\) (with subsets of \(\W_0\)-null sets from $\mathcal F$)
and
\(
\mathcal{G}_t \triangleq \mathcal{G}^{\W_0}_t \cap \mathcal{F}^Y_{\rho \wedge T'}.
\)
Since \(\P_0^\leftrightarrow \sim \W_0\) on \(\mathcal{F}^Y_{\rho \wedge T^*}\), we also have
\(\mathcal{G}_t = \mathcal{G}^{\P^\leftrightarrow_0}_t \cap \mathcal{F}^Y_{\rho \wedge T'}\) and \(\P_0^\leftrightarrow \sim \W_0\) on \(\mathcal{G}_{\rho \wedge T'}\). Notice that the integral process
\[
\int_0^{\cdot \wedge \rho \wedge T'} \on{sgn} (\X_s) d \X_s = |\X_{\cdot \wedge \rho \wedge T'}| - L^0_{\cdot \wedge \rho \wedge T'} (\X)
\]
is \((\mathcal{G}_t)_{t \geq 0}\)-adapted by virtue of \cite[Corollary VI.1.9]{RY}. 
Hence, we conclude, by the tower property, that it is a \(\W_0\)-\((\mathcal{G}_t)_{t \geq 0}\)-martingale (notice that \(\on{sgn} (\X)\) is bounded, which yields the true martingale property).
By Girsanov's theorem, there exists a \((\mathcal{G}_t)_{t \geq 0}\)-predictable process \(\tilde{\beta}\) such that \(\P^\leftrightarrow_0\)-a.s.
\begin{align} \label{eq: 15.10 2}
\int_0^{\cdot \wedge \rho \wedge T'} \on{sgn} (\X_s) d \X_s  = \text{local \(\P^\leftrightarrow_0\)-\((\mathcal{G}_t)_{t \geq 0}\)-martingale} + \int_0^{\cdot \wedge \rho \wedge T'} \tilde{\beta}_s ds.
\end{align}
As above, notice that the integral process
\begin{align} \label{eq: 15.10 4}
 \int_0^{\cdot \wedge \rho \wedge T'} \on{sgn} (\s^\leftrightarrow (\X_s)) d \s^\leftrightarrow (\X_s) = \s (|\X_{\cdot \wedge \rho \wedge T'}|) - L^0_{\cdot \wedge \rho \wedge T'} (\s^\leftrightarrow (\X))
\end{align}
 is \((\mathcal{G}_t)_{t \geq 0}\)-adapted by \cite[Corollary VI.1.9]{RY}.
We conclude, by the tower property, that it is a \(\P^\leftrightarrow_0\)-\((\mathcal{G}_t)_{t \geq 0}\)-martingale. 
Next, by the generalized It\^o formula (Lemma~\ref{lem: occ smg}), we get that \(\P^\leftrightarrow_0\)-a.s.
\begin{equation} \label{eq: 15.10-1}
\begin{split}
 \int_0^{\cdot \wedge \rho \wedge T'} \on{sgn} (\s^\leftrightarrow (\X_s)) d \s^\leftrightarrow (\X_s) 
 =  &\int_0^{\cdot \wedge \rho \wedge T'} \on{sgn} (\s^\leftrightarrow (\X_s)) [\s^\leftrightarrow]' (\X_s) d \X_s \\&\quad + \frac{1}{2}  \int_0^{\cdot \wedge \rho \wedge T'} \on{sgn} (\s^\leftrightarrow (\X_s)) d \Big[  \int L^x_{s} (\X) [\s^\leftrightarrow ]'' (dx) \Big].
\end{split}
\end{equation}
Notice that \(\on{sgn}(\s^\leftrightarrow (\X)) = \on{sgn}(\X)\). Hence, by virtue of \eqref{eq: 15.10 2}, we obtain that \(\P^\leftrightarrow_0\)-a.s.
\begin{equation} \label{eq: 15.10 3}
\begin{split}
\int_0^{\cdot \wedge \rho \wedge T'} \on{sgn} (\s^\leftrightarrow (\X_s)) [\s^\leftrightarrow]' (\X_s) d \X_s 
= \text{local \(\P^\leftrightarrow_0\)-\((\mathcal{G}_t)_{t \geq 0}\)-martingale} + \int_0^{\cdot \wedge \rho \wedge T'} [\s^\leftrightarrow]' (\X_s) \tilde{\beta}_s ds.
\end{split}
\end{equation}
Therefore, as the process in \eqref{eq: 15.10 4} is a (local) \(\P^\leftrightarrow_0\)-\((\mathcal{G}_t)_{t \geq 0}\)-martingale, \eqref{eq: 15.10-1} and \eqref{eq: 15.10 3}, together with the fact that continuous local martingales of (locally) finite variation are constant, imply that \(\P^\leftrightarrow_0\)-a.s.
\begin{align*}
\int_0^{\cdot \wedge \rho \wedge T'} \on{sgn} (\s^\leftrightarrow (\X_s)) d \Big[  \int L^x_{s} (\X) [\s^\leftrightarrow ]'' (dx) \Big] = - \int_0^{\cdot \wedge \rho \wedge T'} 2 [\s^\leftrightarrow]' (\X_s) \tilde{\beta}_s ds.
\end{align*}
Integrating \(\on{sgn} (\s^\leftrightarrow (\X)) = \on{sgn} (\X)\) against both sides yields that \(\P^\leftrightarrow_0\)-a.s.
\begin{align*} 
\int L^x_{\cdot \wedge \rho \wedge T'} (\X)  [\s^\leftrightarrow]'' (dx) = - \int_0^{\cdot \wedge \rho \wedge T'} 2 \on{sgn} (\X_s)  [\s^\leftrightarrow]' (\X_s) \tilde{\beta}_s ds.
\end{align*}
From this point on we can argue almost verbatim as in the proof of Lemma \ref{lem: sing starting in A} to obtain the existence of a Borel map \(\bar{\beta} \colon [- c', c'] \to \mathbb{R}\), where \(0 < c'< c\) might have been replaced by a smaller value, such that \([\s^\leftrightarrow]'' (dx) = \bar{\beta} (x) [\s^\leftrightarrow]' (x) dx\) and \(\bar{\beta} \in L^2([- c', c'])\). We omit the details. In summary, condition~(b) is violated.

Finally, we prove that (c) is violated.
We are no longer assuming that \(\Q_{0} = \Qo_0\) and that \(\P_{0} = \Po_0\).
To achieve the aim, it suffices only to observe that
the formula~\eqref{eq: speed measure density local time 2} at the point $z=0$ ($=b$)
and part~(iii) of Lemma~\ref{lem: occ smg} (see also~\eqref{eq:071022a3}) provide a contradiction to part~(c) (more precisely, we need to use~\eqref{eq: speed measure density local time 2} with $\xi$ replaced by $\xi\wedge T'$ because $\s$ is the difference of two convex functions only on $[0,c']$,
whereas such a structure of $\s$ is needed to apply part~(iii) of Lemma~\ref{lem: occ smg}).
The proof is complete.
\end{proof}

\subsection{Proof of Theorem \ref{theo: main1}}
First of all, part~(i) is trivial. 
Let us note that if \(x_0\) is an absorbing boundary point for both diffusions \((x \mapsto \P_x)\) and \((x \mapsto \Q_x)\), then \(\P_{x_0} = \Q_{x_0}\), which is covered by~(i).
We now prove part~(ii).
That is, below we always assume that $\P_{x_0}\ne\Q_{x_0}$.

Next, we discuss the case where \(x_0\) is a boundary point which is absorbing for one of the diffusions, but not for the other.
In this case, \(x_0\) is separating (recall Definition~\ref{def: non-sep bound}).
Consequently,
\(U \wedge V\wedge R= 0\).
We now show that \(\P_{x_0}\perp \Q_{x_0}\) on \(\mathcal{F}_0\), which then yields that \(\P_{x_0}, \Q_{x_0}\)-a.s.
\(S = 0 = U \wedge V\wedge R\)
and hence the claim. Set \(C \triangleq \inf (t > 0 \colon \X_t \not = x_0)\). In case \(x_0\) is an absorbing boundary point for \((x \mapsto \P_{x})\) and a reflecting boundary point for \((x \mapsto \Q_{x})\), we have \(\P_{x_0} (C > 0) = 1\) and \(\Q_{x_0} (C > 0) = 0\). Thus, as \(\{C > 0\} \in \mathcal{F}_0\), it follows that \(\P_{x_0} \perp \Q_{x_0}\) on \(\mathcal{F}_0\). The case where \(x_0\) is absorbing for \((x \mapsto \Q_x)\) and reflecting for \((x \mapsto \P_x)\) follows by symmetry.

We now discuss all remaining cases, where we can assume that, in case \(x_0\) is a boundary point, it is reflecting for both diffusions. 
The proof is split into two parts. First, we show that \(\P_{x_0}, \Q_{x_0}\)-a.s.
\(U\wedge V\wedge R \leq S\)
and, second, we prove that \(\P_{x_0}, \Q_{x_0}\)-a.s.
\(S\leq U \wedge V \wedge R\).
Recall that \(\A\) denotes the set of separating points.
\\

\noindent
\emph{Proof of \(U \wedge V\wedge R \leq S\):}
If \(x_0 \in \A\), then
\(U \wedge V\wedge R = 0\)
and the claim is trivial. Thus, we can and will assume that \(x_0 \not \in \A\). In the following, we will consider the case \(x_0 \in J^\circ\). The case where \(x_0\) is a boundary point (which is reflecting for both diffusions) can be handled the same way. 
Define 
\[
\alpha' \triangleq \begin{cases} l, & \alpha = \Delta,\\ \alpha, &\alpha \not = \Delta,\end{cases} \qquad \gamma' \triangleq \begin{cases} r, & \gamma = \Delta,\\ \gamma, &\gamma \not = \Delta,
\end{cases}
\]
and take two sequences \((a_n)_{n \in \mathbb{N}}\) and \((c_n)_{n \in \mathbb{N}}\) such that \(a_1 < x_0 < c_1,\) \(a_{n + 1} < a_n,\) \(c_{n + 1} > c_n,\) \(\lim_{n \to \infty} a_n = \alpha'\) and \(\lim_{n \to \infty} c_n = \gamma'\).
By Lemma~\ref{lem: loc equivalence} and Proposition~\ref{prop:121224a1},
we have \(\P_{x_0}, \Q_{x_0}\)-a.s.
\(T_{a_n} \wedge T_{c_n} < S\)
for all \(n \in \mathbb{N}\).

\smallskip 
\noindent
\emph{Case 1 (\(\alpha \not = \Delta\) and \(\gamma \not = \Delta\)):}
In this case, we have \(R = \delta\) and, by virtue of Lemma~\ref{lem:200223a2}, \(\P_{x_0}, \Q_{x_0}\)-a.s.
\(T_{a_n} \wedge T_{c_n} \nearrow U \wedge V\).
Hence, \(\P_{x_0}, \Q_{x_0}\)-a.s.
\(U \wedge V\wedge R \leq S\)
follows. 

\smallskip 
\noindent
\emph{Case 2 (\(\alpha = \Delta, \gamma \not = \Delta\) and \(l\) is inaccessible
or absorbing for one, equivalently for both, of the diffusions}):
In this case, $U=R=\delta$.
By Lemma~\ref{lem: local equivalence hitting time from the right} and Proposition~\ref{prop:121224a1}, it holds that \(\P_{x_0} \sim \Q_{x_0}\) on \(\mathcal{F}_{T_{c_n}}\) for all \(n \in \mathbb{N}\). 
Hence, \(\P_{x_0}, \Q_{x_0}\)-a.s.
\(T_{c_n} < S.\)
Notice that
$\{T_{c_n}=\infty\}\nearrow\{V=\delta\}$.
Therefore, \(\P_{x_0}, \Q_{x_0}\)-a.s.\ we have $S=\delta$ on $\{V=\delta\}$.
Together with the fact that \(\P_{x_0}, \Q_{x_0}\)-a.s.\ $T_{c_n}\nearrow V\wedge\infty$ we conclude that \(\P_{x_0}, \Q_{x_0}\)-a.s.\ \(U \wedge V\wedge R =V\leq S\).

\smallskip 
\noindent
\emph{Case 3 (\(\alpha = \Delta, \gamma \not = \Delta\) and \(l\) is reflecting
for one, equivalently for both, of the diffusions):}
Here, we again have $U=R=\delta$.
To simplify our notation, we assume that \(J = \mathbb{R}_+\), in particular, \(l = 0\).
Let  \((\mathbb{R} \ni x \mapsto \P^\leftrightarrow_x)\) and \((\mathbb{R} \ni x \mapsto \Q^\leftrightarrow_x)\)
be diffusions constructed from
$(\mathbb R_+\ni x\mapsto\P_x)$ and $(\mathbb R_+\ni x\mapsto\Q_x)$
as in Lemma~\ref{lem: refl}.
In particular, we have
\[
\P_{x_0} = \P^\leftrightarrow_{x_0} \circ |\X|^{-1}
\qquad\text{and}\qquad
\Q_{x_0} = \Q^\leftrightarrow_{x_0} \circ |\X|^{-1}.
\]
For a moment fix \(n \in \mathbb{N}\).
Since all points in \([- c_n, c_n]\) are non-separating for the symmetrized setting (see Lemma~\ref{lem: refl} for the structure of scale and speed of the symmetrized diffusions), Lemma~\ref{lem: loc equivalence} yields that \(\P^\leftrightarrow_{x_0} \sim \Q^\leftrightarrow_{x_0}\) on \(\mathcal{F}_{T_{-c_n} \wedge T_{c_n}}\) and, consequently, \(\P_{x_0} \sim \Q_{x_0}\) on~\(\mathcal{F}_{T_{c_n}}\).
Therefore, \(\P_{x_0}, \Q_{x_0}\)-a.s.\ $T_{c_n}<S$.
By Lemma~\ref{lem: finiteness hitting times}, \(\P_{x_0}, \Q_{x_0}\)-a.s.\ $T_{c_n}\nearrow V$.
Thus, \(\P_{x_0}, \Q_{x_0}\)-a.s.\ \(U \wedge V\wedge R =V\leq S\).

\smallskip 
\noindent
\emph{Case 4 (\(\alpha = \Delta, \gamma = \Delta\) and both \(l\) and \(r\) are inaccessible or absorbing):}
First of all, notice that in this case we have \(x_0 \in J^\circ\) and that \(l\) and \(r\) are finite (because they are half-good and \(\S= \on{Id}\)). 
Let \(([l, r] \ni x \mapsto \Qo_x)\) be a Brownian motion which is absorbed in \(l\) and \(r\).
Further, let \(([l, r] \ni x \mapsto \Po_x)\) be a diffusion with characteristics \((\s, \mo)\), where the measure \(\m^\circ\) on \(\mathcal{B}([l, r])\) is given by
\[
\frac{\m^\circ (dx)}{dx} \triangleq \Big( \frac{d \m}{d \M}\Big) (x) \text{ on } (l, r), \qquad \m^\circ (\{l\}) \equiv \m^\circ (\{r\}) \triangleq \infty.
\]
That $\m^\circ$ is a valid speed measure and the state space of the latter diffusion is indeed $[l,r]$ (i.e., both endpoints are accessible) follows as in the proof of Lemma~\ref{lem: loc equivalence}.
Moreover, again as in the proof of Lemma~\ref{lem: loc equivalence}, we argue that Lemma~\ref{lem: sep ito diff} applies to the diffusions \(([l, r] \ni x \mapsto \Po_x)\) and \(([l, r] \ni x \mapsto \Qo_x)\) and consequently, \(\Po_{x_0} \sim \Qo_{x_0}\).

We now use the time-change argument from Lemma~\ref{lem: loc equivalence} to conclude that \(\P_{x_0} \sim \Q_{x_0}\).
Define
\[
\tilde{\l}^* (t, x) \triangleq	\begin{cases}
\limsup_{h \searrow 0}
\frac{\int_0^t \1 \{x - h < \X_s < x + h\} ds}{2h}
&\text{if }x\in (l, r),
\\[2mm]
0 & \text{if } x \in \{l, r\},
\end{cases}
\]
and set 
\[
\g (t) \triangleq \begin{cases} \int_{a'}^{c'} \tilde{\l}^* (t, x) \M (dx), & t < T_l \wedge T_r,\\ \infty, & t \geq T_l \wedge T_r. \end{cases}
\]
Let \(\g^{-1}\) be the right-inverse of \(\g\).
According to the chain rule for diffusions (Theorem~\ref{theo: chain rule}), we get that 
\[
\Po_{x_0} \circ \X_{\g^{-1}(\cdot)}^{-1} = \P_{x_0}
\qquad\text{and}\qquad
\Qo_{x_0} \circ \X_{\g^{-1}(\cdot)}^{-1} = \Q_{x_0}
\]
(see the proof of Lemma~\ref{lem: loc equivalence} for more details). In particular, at this stage we used that \(l\) and \(r\) are either inaccessible or absorbing for both diffusions \((x \mapsto \P_x)\) and \((x \mapsto \Q_x)\).
Thanks to these equalities, \(\Po_{x_0} \sim \Qo_{x_0}\) implies \(\P_{x_0} \sim \Q_{x_0}\) and, consequently, \(\P_{x_0}, \Q_{x_0}\)-a.s. \(S = \delta\).

\smallskip 
\noindent
\emph{Case 5 (\(\alpha = \Delta, \gamma = \Delta\), \(l\) is reflecting and \(r\) is inaccessible or absorbing):}
This case is reduced to the previous one via Lemma~\ref{lem: refl} (cf.\ Cases 2 and~3).

\smallskip 
\noindent
\emph{Case 6 (\(\alpha = \Delta, \gamma = \Delta\) and both \(l\) and \(r\) are reflecting):}
Here $U=V=\delta$ and $R=\infty$. We need to prove that $\P_{x_0},\Q_{x_0}$-a.s.\ $S\ge\infty$.
First, we notice that, in this case, $l$ and $r$ are finite (because they are half-good and $\S=\on{Id}$).
To simplify our notation, we assume that $J=[0,1]$, in particular, $l=0$ and $r=1$.
Let $(\mathbb R\ni x\mapsto\QQ_x)$
(resp., $(\mathbb R\ni x\mapsto\tQQ_x)$)
be the diffusion constructed from
$([0,1]\ni x\mapsto\P_x)$
(resp., $([0,1]\ni x\mapsto\Q_x)$)
as in Lemma~\ref{lem:200223a3}.
Let $\overline S$ be the separating time for $\QQ_{x_0}$ and $\tQQ_{x_0}$.
All points in $\mathbb R$ are non-separating for $(x\mapsto\QQ_x)$ and $(x\mapsto\tQQ_x)$.
As the boundaries $\pm\infty$ are inaccessible for both these diffusions,
Case~1 above yields
$\QQ_{x_0},\tQQ_{x_0}$-a.s.\ $\overline S\ge\infty$,
hence, by Proposition~\ref{prop: AC Sing},
$\QQ_{x_0}\sim_\textup{loc}\tQQ_{x_0}$.
As $\P_{x_0}=\QQ_{x_0}\circ\f(\X)^{-1}$ and
$\Q_{x_0}=\tQQ_{x_0}\circ\f(\X)^{-1}$
for a function $\f$ that is described in Lemma~\ref{lem:200223a3}, it follows that
$\P_{x_0}\sim_\textup{loc}\Q_{x_0}$.
Applying Proposition~\ref{prop: AC Sing} once again,
we obtain $\P_{x_0},\Q_{x_0}$-a.s.\ $S\ge\infty$, as required.

\smallskip 
\noindent
Up to symmetry we considered all possible cases.\qed
\\

\noindent
\emph{Proof of \(S\le U \wedge V \wedge R\):} We will distinguish several cases. 

\smallskip 
\noindent
\emph{Case 1 (\(x_0 \in \A \cap J^\circ\)):} Lemma \ref{lem: sing starting in A} yields \(\P_{x_0} \perp \Q_{x_0}\) on \(\mathcal{F}_0\) and hence, \(\P_{x_0}, \Q_{x_0}\)-a.s. \(S = 0 = U \wedge V \wedge R\). 

\smallskip 
\noindent
\emph{Case 2 (\(x_0 \in \A \cap \partial J\)):}
Recall that we suppose that, in case $x_0$ is a boundary point, it is reflecting for both diffusions \((x \mapsto \P_{x})\) and \((x \mapsto\Q_{x})\). 

First, suppose that there is a point \(z \in J^\circ\) such that all points in \(\ofr x_0,z \gsr\)
(recall this notation from~\eqref{eq:271022a0})
are non-separating. Then, Lemma~\ref{lem: singular boundary separating} yields \(\P_{x_0} \perp \Q_{x_0}\) on \(\mathcal{F}_0\) and hence, \(\P_{x_0}, \Q_{x_0}\)-a.s. \(S = 0 = U \wedge V \wedge R\). 

Second, suppose that we can find a monotone sequence \(z_1, z_2, \ldots \in J^\circ\) of separating points such that \(z_n \to x_0\).
As \(x_0\) is a reflecting boundary point, Lemma~\ref{lem: finiteness hitting times} yields that \(\P_{x_0}, \Q_{x_0}\)-a.s.\ \(T_{z_n} < \infty\) for all \(n\in\mathbb N\). Now, the strong Markov property and Lemma~\ref{lem: sing starting in A} yield that \(\P_{x_0} \perp \Q_{x_0}\) on \(\mathcal{F}_{T_{z_n}}\) and, consequently, \(\P_{x_0}, \Q_{x_0}\)-a.s.\ \(S \leq T_{z_n}\), for all \(n\in\mathbb N\).
Define $T_{x_0+}=\lim_{n\to\infty}T_{z_n}$.
By Lemma \ref{lem: exit immediately}, \(\P_{x_0}, \Q_{x_0}\)-a.s.\ $T_{x_0+}=0$.
Thus, \(\P_{x_0}, \Q_{x_0}\)-a.s.\ \(S =0 = U \wedge V \wedge R\).

\smallskip  
\noindent
 \emph{Case 3 (\(l < \alpha < x_0 < \gamma < r\)):} 
In this case, we have \(\P_{x_0}, \Q_{x_0}\)-a.s.\ \(U \wedge V \wedge R = T_\alpha \wedge T_\gamma < \infty\)
by~\eqref{eq:260223a1}.
Using the strong Markov property and the result from Case~1,
we obtain that \(\P_{x_0} \perp \Q_{x_0}\) on \(\mathcal{F}_{T_\alpha} \cap \{T_\alpha < T_\gamma\}\).
Proposition~\ref{prop:121224a1} yields
\(\P_{x_0}, \Q_{x_0}\)-a.s.\ \(S \leq T_\alpha\) on \(\{T_\alpha < T_\gamma\}\). Similarly, we get that \(\P_{x_0}, \Q_{x_0}\)-a.s.\ \(S \leq T_\gamma\) on \(\{T_\alpha > T_\gamma\}\). In summary, \(\P_{x_0}, \Q_{x_0}\)-a.s.\ \(S \leq T_\alpha \wedge T_\gamma = U \wedge V \wedge R\).

\smallskip 
\noindent
{\em Case 4 (\(l < \alpha < x_0, \gamma = \Delta\)):}
It suffices to prove that \(\P_{x_0}, \Q_{x_0}\)-a.s.\ \(S \leq T_\alpha\) on \(\{T_\alpha < \infty\}\). This follows from the strong Markov property,
Lemma~\ref{lem: sing starting in A} and Proposition~\ref{prop:121224a1}.

\smallskip 
\noindent
\emph{Case 5 (\(l < \alpha < x_0 < \gamma = r\)):}
Recall from Lemma~\ref{lem:200223a2} that, restricted to the set \(\{T_\alpha = T_r = \infty\}\), a.a. paths of a diffusion with state space \(J\) travel to \(r\) in the sense that \(\X_t \to r\) as \(t \to \infty\). Consequently, as \(R = \delta\), we have \(\P_{x_0}, \Q_{x_0}\)-a.s. \(U \wedge V \wedge R = T_\alpha \wedge T_r\). It suffices to prove that \(\P_{x_0}, \Q_{x_0}\)-a.s. \(S \leq T_\alpha\) on \(\{T_\alpha < T_r\}\) and \(\P_{x_0}, \Q_{x_0}\)-a.s. \(S \leq T_r\) on \(\{T_\alpha \geq T_r\}\).

The first part, i.e., \(\P_{x_0}, \Q_{x_0}\)-a.s. \(S \leq T_\alpha\) on \(\{T_\alpha < T_r\}\), follows from the strong Markov property,
Lemma~\ref{lem: sing starting in A} and Proposition~\ref{prop:121224a1}.

In the following, we prove that \(\P_{x_0}, \Q_{x_0}\)-a.s. \(S \leq T_r\) on \(\{T_\alpha \geq T_r\}\). 
We distinguish several cases. First, suppose that \(r\) has different boundary classifications for the diffusions \((x \mapsto \P_x)\) and \((x \mapsto \Q_x)\).
In case \(r\) is accessible
for one of the diffusions and inaccessible for the other, the set \(\{T_r < \infty\} \in \mathcal{F}_{T_r}\) shows that \(\P_{x_0} \perp \Q_{x_0}\) on \(\mathcal{F}_{T_r} \cap \{T_\alpha \geq T_r\}\)
(apply Lemma~\ref{lem:020323a2}).
In case \(r\) is reflecting for one of the diffusions and absorbing for the other, the set \(\{\tau = 0, T_r < \infty\} \in \mathcal{F}_{T_r}\), with \(\tau \triangleq \inf (t \geq 0 \colon \X_{t + T_r} \not = r)\), shows that \(\P_{x_0} \perp \Q_{x_0}\) on \(\mathcal{F}_{T_r} \cap \{T_\alpha \geq T_r\}\)
(here we use that, in case of reflection, $\{\tau>0\}$ happens with probability~$0$ by Lemma \ref{lem: exit immediately}).
By Proposition~\ref{prop:121224a1},
in both cases, \(\P_{x_0}, \Q_{x_0}\)-a.s. \(S \leq T_r\) on \(\{T_\alpha \geq T_r\}\).

Next, suppose that \(r\) is reflecting for both diffusions. Then, we can again use the strong Markov property and
Lemma~\ref{lem: singular boundary separating}
to conclude that \(\P_{x_0}, \Q_{x_0}\)-a.s. \(S \leq T_r\) on \(\{T_\alpha \geq T_r\}\). 

Now, assume that \(r\) is either absorbing for both diffusions or inaccessible for both diffusions. In particular, that means that \(r\) cannot be half-good. First, suppose that \(r = \infty\). Then, by Lemma~\ref{lem:200223a2}
and the fact that $\S=\on{Id}$,
\(\Q_{x_0} (T_\alpha \geq T_r) = 0\) and hence, \(\P_{x_0} \perp \Q_{x_0}\) on \(\mathcal{F}_{T_r} \cap \{T_\alpha \geq T_r\}\).
This yields that \(\P_{x_0}, \Q_{x_0}\)-a.s. \(S \leq T_r\) on \(\{T_\alpha \geq T_r\}\).
	From now on, suppose that \(r < \infty\).
	We fix a number \(c \in (\alpha, x_0)\). Recall that all points \([c, r)\) are non-separating and let \(\beta \colon [c, r) \to \mathbb{R}\) be the function as in Definition~\ref{def: non-sep int} (see also \eqref{eq: beta formula}).
	Notice that \(\beta \in L^2_\textup{loc} ([c, r))\), as all points in \([c, r)\) are non-separating
and $\S=\on{Id}$.
Furthermore, again using that $\S=\on{Id}$, \(\X\) is a \(\Q_{x_0}\)-semimartingale (see Lemma~\ref{lem: occ formula diff}) and the stopped process
\(\X_{\cdot \wedge T_c \wedge T_r} = \X_{\cdot \wedge T_c}\) is a continuous local \(\Q_{x_0}\)-martingale.
	The semimartingale occupation time formula (see Lemma \ref{lem: occ smg}) yields that, \(\Q_{x_0}\)-a.s. for all \(t \in \mathbb{R}_+\),
	\[
	\int_0^{t} \, [ \beta (\X_s) ]^2 d \langle \X, \X\rangle_s = \int_{- \infty}^{\infty} [ \beta (x) ]^2 L_{t}^x (\X) dx.
	\]
	As \(\beta \in L^2_\textup{loc} ([c, r))\), this identity shows that, \(\Q_{x_0}\)-a.s. for all \(t < T_r\), 
	\begin{equation}\label{eq:280323a1}
	\int_0^{t \wedge T_c} [ \beta (\X_s) ]^2 d \langle \X, \X\rangle_s < \infty. 
	\end{equation}
	Hence, we can define the process 
	\[
	Z_t \triangleq \begin{cases} \exp \big( - \frac{1}{2} \int_0^{t \wedge T_c} \beta (\X_s) d \X_s - \frac{1}{8} \int_0^{t \wedge T_c} [ \beta (\X_s) ]^2 d \langle \X, \X\rangle_s \big),& t < T_r,
	\\
	\exp \big( - \frac{1}{2} \int_0^{T_c} \beta (\X_s) d \X_s - \frac{1}{8} \int_0^{T_c} [ \beta (\X_s) ]^2 d \langle \X, \X\rangle_s \big), & t \geq T_r, T_c < T_r,\\
	0, & t \geq T_r, T_c \geq T_r. \end{cases}
	\]
	Recall that \(r\) is finite but not half-good. Thus, it holds that
	\(
	\int^r (r - x) [ \beta (x) ]^2 dx = \infty,
	\)
	and Lemma~\ref{lem: pepetual integral diffusion} shows that \(\Q_{x_0}\)-a.s. 
	\[
	\int_0^{T_r} [ \beta( \X_s) ]^2 d \langle \X, \X \rangle_s = \infty \text{ on } \Big\{ \lim_{t \nearrow T_l \wedge T_r} \X_t = r \Big\}.
	\]
	Notice that \(\Q_{x_0}\)-a.s. \(\{T_c \geq T_r\} \subset \{ \lim_{t \nearrow T_l \wedge T_r} \X_t = r\}\)
	by Lemma~\ref{lem:200223a2}, recalling \(\S(r) = r < \infty\).
	Thus, \(\Q_{x_0}\)-a.s. \(Z_{T_r - } = 0\) on \(\{T_c \geq T_r\}\), which proves that \(Z\) is a continuous process. Further, thanks to \cite[Lemma~12.43]{Jacod}, \(Z\) is even a continuous local \(\Q_{x_0}\)-martingale. 
Take a sequence
\(x_0 < r_1 < r_2 < \cdots\)
such that \(r_n \to r\). For \(m, n \in \mathbb{N}\), define
	\[
	\sigma_m^n \triangleq \inf \Big( t \geq 0 \colon \int_0^{t \wedge T_c \wedge T_{r_n}} [ \beta (\X_s) ]^2 d \langle \X, \X \rangle_s \geq m \Big). 
	\]
	The process 
	\[
	t \mapsto \int_0^{t \wedge T_c \wedge T_{r_n}} [ \beta (\X_s) ]^2 d \langle \X, \X \rangle_s \in [0, \infty]
	\]
	is left-continuous (by the monotone convergence theorem) but, in general, it might jump to infinity. Therefore, \(\sigma_m^n\) is an \((\mathcal{F}_t)_{t \geq 0}\)-stopping time but it might fail to be \((\mathcal{F}_t)_{t \geq 0}\)-predictable (cf. \cite[p.~193]{JS} for a more detailed discussion). 
We now pass to a suitable predictable version. 
Notice that \(T_{c} \wedge T_{r_n} = (T_c \wedge T_{r_n}) (\X_{\cdot \wedge T_c \wedge T_{r_n}})\) by Galmarino's test (cf. \cite[Lemma~III.2.43]{JS}). Hence, we obtain 
	\begin{align*}
		\sigma^n_m (\X_{\cdot \wedge T_c \wedge T_{r_n}}) &= \inf \Big( t \geq 0 \colon \int_0^{t \wedge (T_c \wedge T_{r_n}) (\X_{\cdot \wedge T_c \wedge T_{r_n}})} \big[ \beta (\X_{s \wedge T_c \wedge T_{r_n}}) \big]^2 \, d \langle \X_{\cdot \wedge T_c \wedge T_{r_n}}, \X_{\cdot \wedge T_c \wedge T_r} \rangle_s\Big) 
		\\&= \inf \Big( t \geq 0 \colon \int_0^{t \wedge T_c \wedge T_{r_n}} \big[ \beta (\X_{s \wedge T_c \wedge T_{r_n}}) \big]^2 \, d \langle \X_{\cdot \wedge T_c \wedge T_{r_n}}, \X_{\cdot \wedge T_c \wedge T_r} \rangle_s\Big)= \sigma^n_m.
	\end{align*} 
It follows from Lemma~\ref{lem: Meyer theo}
applied to the diffusion \(\Q_{x_0} \circ \X_{\cdot \wedge T_c \wedge T_{r_n}}^{-1}\)
that there exists an \((\mathcal{F}_t)_{t \geq 0}\)-predictable time \(\bar{\sigma}^n_m\) such that \(\Q_{x_0} \circ \X_{\cdot \wedge T_c \wedge T_{r_n}}^{-1}\)-a.s. \(\sigma^n_m = \bar{\sigma}^n_m\). 
		As \(\Q_{x_0}\)-a.s. \(\sigma^n_m > 0\), we may take \(\bar{\sigma}^n_m > 0\) identically.  Hence, by \cite[III.2.36]{JS}, the time \(\bar{\sigma}^n_m\) is an \((\mathcal{F}^o_t)_{t \geq 0}\)-stopping time, where \(\mathcal{F}^o_t \triangleq \sigma (\X_s, s \leq t)\). 
		We also observe that \(\tilde{\sigma}^n_m \triangleq \bar{\sigma}^n_m (\X_{\cdot \wedge T_c \wedge T_{r_n}})\) is an \((\mathcal{F}^o_{t \wedge T_c \wedge T_{r_n}})_{t \geq 0}\)-stopping time (see \cite[Proposition~10.35]{Jacod}). In particular, it is an \((\mathcal{F}_t^o)_{t \geq 0}\)-stopping time.
	Below we use this observation to apply Lemma~\ref{lem: loc uni}.
Since \(\Q_{x_0} \circ \X_{\cdot \wedge T_c \wedge T_{r_n}}^{-1}\)-a.s. \(\sigma^n_m = \bar{\sigma}^n_m\),
we have \(\Q_{x_0}\)-a.s.
\(
\sigma^n_m
= \sigma^n_m (\X_{\cdot \wedge T_c \wedge T_{r_n}})
= \bar\sigma^n_m (\X_{\cdot \wedge T_c \wedge T_{r_n}})
= \tilde{\sigma}^n_m.
\)
By Lemma~\ref{lem: finiteness hitting times}, $\Q_{x_0}$-a.s.\ $T_c\wedge T_{r_n}<\infty$,
hence $T_c\wedge T_{r_n}<T_r$ and therefore, by~\eqref{eq:280323a1},
$\Q_{x_0}$-a.s.
\begin{equation}\label{eq:020423a1}
	\int_0^{T_c\wedge T_{r_n}}[\beta(\X_s)]^2\,d\langle\X,\X\rangle_s<\infty.
\end{equation}
Thus, by Novikov's condition, \(Z_{\cdot \wedge T_{r_n} \wedge \tilde{\sigma}^n_m}\) is a uniformly integrable
\(\Q_{x_0}\)-\((\mathcal{F}^o_t)_{t \geq 0}\)-martingale
that starts in~$1$.
Notice that $\Q_{x_0}$-a.s.\ $Z_{T_{r_n}\wedge\tilde\sigma^n_m}>0$.
We define a probability measure \(\mathds{K}^{n, m}\) by the formula
\begin{align*}
\mathds{K}^{n, m} ( G ) \triangleq \E^{\Q_{x_0}} \big[ Z_{T_{r_n} \wedge \tilde{\sigma}_m^n} \1_{G}\big], \quad G \in \mathcal{F}.
\end{align*}
Below we will prove that 
	\begin{align} \label{eq: K eq}
	\mathds{K}^{n, m} = \P_{x_0} \text{ on } \mathcal{F}^o_{T_{c} \wedge T_{r_n} \wedge \tilde{\sigma}_m^n}, \quad n, m \in \mathbb{N}.
	\end{align}
Suppose for a moment that this identity is established.
Given a $\sigma$-field $\mathcal G$ on $\Omega$ such that $\mathcal G\subset\mathcal F$,
let \((\P_{x_0}|\mathcal G)^{ac}\) denote the absolutely continuous part of the restriction \(\P_{x_0}|\mathcal G\) with respect to the restriction \(\Q_{x_0}|\mathcal G\).
By Jessen's theorem (see Corollary~\ref{cor:030323a1}), $\Q_{x_0}$-a.s.
\begin{equation}\label{eq:020423a2}
Z_{T_{r_n}\wedge\tilde\sigma^n_m}
\equiv
\frac{d(\P_{x_0}|\mathcal F^o_{T_c\wedge T_{r_n}\wedge\tilde\sigma^n_m})}
{d(\Q_{x_0}|\mathcal F^o_{T_c\wedge T_{r_n}\wedge\tilde\sigma^n_m})}
\to
\frac{d(\P_{x_0}|\bigvee_{k \in \mathbb{N}}\mathcal F^o_{T_c\wedge T_{r_n}\wedge\tilde\sigma^n_k})^{ac}}
{d(\Q_{x_0}|\bigvee_{k \in \mathbb{N}}\mathcal F^o_{T_c\wedge T_{r_n}\wedge\tilde\sigma^n_k})^{\phantom{ac}}},
\quad m\to\infty.
\end{equation}
Since $\P_{x_0}\sim\Q_{x_0}$ on $\mathcal F_{T_c\wedge T_{r_n}}$ by Lemma~\ref{lem: loc equivalence},  
\eqref{eq:020423a1} holds $\P_{x_0},\Q_{x_0}$-a.s.
This yields that $\P_{x_0},\Q_{x_0}$-a.s.\ $\lim_{m\to\infty}\sigma^n_m=\infty$.
Moreover, recalling that
$\sigma^n_m = \sigma^n_m (\X_{\cdot \wedge T_c \wedge T_{r_n}})$
and that
\(\tilde{\sigma}^n_m\) is an \((\mathcal{F}^o_{t \wedge T_c \wedge T_{r_n}})_{t \geq 0}\)-stopping time
(hence, both $\sigma^n_m$ and $\tilde\sigma^n_m$ are $\mathcal F^o_{T_c\wedge T_{r_n}}$-measurable),
$\P_{x_0}\sim\Q_{x_0}$ on $\mathcal F_{T_c\wedge T_{r_n}}$ yields
that $\sigma^n_m=\tilde\sigma^n_m$ holds not only $\Q_{x_0}$-a.s. but also $\P_{x_0}$-a.s.
Consequently,
$\P_{x_0},\Q_{x_0}$-a.s.\ $\lim_{m\to\infty}\tilde\sigma^n_m=\infty$,
hence, $\P_{x_0},\Q_{x_0}$-a.s.
\begin{equation}\label{eq:020423a3}
\bigvee_{k \in \mathbb{N}} \mathcal F^o_{T_c\wedge T_{r_n}\wedge\tilde\sigma^n_k}
=
\mathcal F^o_{T_c\wedge T_{r_n}},
\end{equation}
therefore, $\Q_{x_0}$-a.s.
\begin{equation}\label{eq:020423a4}
\frac{d(\P_{x_0}|\bigvee_{k \in \mathbb{N}}\mathcal F^o_{T_c\wedge T_{r_n}\wedge\tilde\sigma^n_k})^{ac}}
{d(\Q_{x_0}|\bigvee_{k \in \mathbb{N}}\mathcal F^o_{T_c\wedge T_{r_n}\wedge\tilde\sigma^n_k})^{\phantom{ac}}}
=
\frac{d(\P_{x_0}|\mathcal F^o_{T_c\wedge T_{r_n}})^{ac}}
{d(\Q_{x_0}|\mathcal F^o_{T_c\wedge T_{r_n}})^{\phantom{ac}}}.
\end{equation}
(Notice that we need \eqref{eq:020423a3} to hold $\Q_{x_0}$-a.s. {\em and} \(\P_{x_0}\)-a.s. to conclude that \eqref{eq:020423a4} holds \(\Q_{x_0}\)-a.s.
Indeed, if the $\sigma$-fields $\bigvee_{k \in \mathbb{N}} \mathcal F^o_{T_c\wedge T_{r_n}\wedge\tilde\sigma^n_k}$ and $\mathcal F^o_{T_c\wedge T_{r_n}}$ were essentially different under $\P_{x_0}$,
then the restrictions of the measures to these $\sigma$-fields could have essentially different absolutely continuous parts.)
Now, \eqref{eq:020423a2} and~\eqref{eq:020423a4} yield that $\Q_{x_0}$-a.s.
\[
Z_{T_{r_n}}
=
\frac{d(\P_{x_0}|\mathcal F^o_{T_c\wedge T_{r_n}})^{ac}}
{d(\Q_{x_0}|\mathcal F^o_{T_c\wedge T_{r_n}})^{\phantom{ac}}}.
\]
Using Jessen's theorem together with the fact that $\bigvee_{n \in \mathbb{N}} \mathcal F^o_{T_c\wedge T_{r_n}}=\mathcal F^o_{T_c\wedge T_r}$, we get that $\Q_{x_0}$-a.s.
$$
Z_{T_{r_n}}
\equiv
\frac{d(\P_{x_0}|\mathcal F^o_{T_c\wedge T_{r_n}})^{ac}}
{d(\Q_{x_0}|\mathcal F^o_{T_c\wedge T_{r_n}})^{\phantom{ac}}}
\to
\frac{d(\P_{x_0}|\mathcal F^o_{T_c\wedge T_r})^{ac}}
{d(\Q_{x_0}|\mathcal F^o_{T_c\wedge T_r})^{\phantom{ac}}},
\quad n\to\infty.
$$
This yields that \(\Q_{x_0}\)-a.s.
\[
Z_{T_{r}}
=
\frac{d(\P_{x_0}|\mathcal F^o_{T_c\wedge T_r})^{ac}}
{d(\Q_{x_0}|\mathcal F^o_{T_c\wedge T_r})^{\phantom{ac}}}.
\]
Using that \(\Q_{x_0}\)-a.s. \(Z_{T_r} = 0\) on \(\{T_c \geq T_r\}\), we conclude that \(\P_{x_0} \perp \Q_{x_0}\) on \(\mathcal{F}^o_{T_r} \cap \{T_c \geq T_r\}\).
Further, as  $\mathcal F^o_{T_r}\subset\mathcal F_{T_r}$, we obtain that
\(\P_{x_0}, \Q_{x_0}\)-a.s. \(S \leq T_r\) on \(\{T_c \geq T_r\}\).
It follows from Lemma~\ref{lem:200223a2} that $\P_{x_0},\Q_{x_0}$-a.s.\ $\{T_c\ge T_r\}\nearrow\{T_\alpha\ge T_r\}$ as $c\searrow\alpha$.
Hence \(\P_{x_0}, \Q_{x_0}\)-a.s. \(S \leq T_r\) on \(\{T_\alpha \geq T_r\}\).

It remains to prove \eqref{eq: K eq}.
Let \(f \in C_b([\s (c), \s( r_n )]; \mathbb{R})\) be such that \(f |_{(\s (c), \s (r_n))}\) is the difference of two convex functions on \((\s (c), \s(r_n))\) and \(d f'_+ = 2 g d \m \circ \s^{-1}\) on \((\s (c),\s( r_n))\) for some \(g \in C_b([\s (c), \s(r_n)]; \mathbb{R})\) with \(g (\s(c)) = g (\s(r_n)) = 0\).
The generalized It\^o formula
(see Lemma~\ref{lem: occ smg}~(ii)) yields that,
\(\Q_{x_0}\)-a.s.\ for all \(t < T_c \wedge T_{r_n}\),
\begin{equation}\label{eq:020423b1}
df(\s(\X_t))=f'_-(\s(\X_t))\,d\s(\X_t)+\tfrac12\, d\int L_t^y(\s(\X))\,df'_+(y).
\end{equation}
As all points in $[c,r_n]$ are non-separating, $\s$ is a $C^1$-function with absolutely continuous derivative on $[c,r_n]$ (recall Definition~\ref{def: non-sep int}).
Hence,
\(\Q_{x_0}\)-a.s.\ for \(t < T_c \wedge T_{r_n}\),
\begin{equation}\label{eq:020423b2}
d\s(\X_t)=\s'(\X_t)\,d\X_t+\tfrac12\,\s''(\X_t)\,d\langle\X,\X\rangle_t.
\end{equation}
Furthermore, applying part~(iii) of Lemma~\ref{lem: occ smg} together with the fact that $d\M=\s'\,d\m$ on $[c,r_n]$ (see Definition~\ref{def: non-sep int}), we get,
\(\Q_{x_0}\)-a.s.\ for \(t < T_c \wedge T_{r_n}\),
\begin{equation} \label{eq:020423b3}
	\begin{split}
\frac12 \int L_t^y(\s(\X))\,df'_+(y)
&=
\int L_t^y(\s(\X))g(y)\,\m\circ\s^{-1}(dy)
=
\int L_t^{\s(x)}(\s(\X))g(\s(x))\,\m(dx)
\\
&=
\int L_t^x(\X)g(\s(x))\,\M(dx)
=
\int_0^t g(\s(\X_s))\,ds,
\end{split}
\end{equation}
where the last equality is the occupation time formula for diffusions
(more precisely, see~\eqref{eq:241124a2}).
Substituting \eqref{eq:020423b2} and \eqref{eq:020423b3} into~\eqref{eq:020423b1} yields that,
\(\Q_{x_0}\)-a.s.\ for \(t < T_c \wedge T_{r_n}\),
$$
df(\s(\X_t))
=
f'_-(\s(\X_t))\s'(\X_t)\,d\X_t
+
\tfrac12\, f'_-(\s(\X_t))\s''(\X_t)\,d\langle\X,\X\rangle_t
+
g(\s(\X_t))\,dt.
$$
Recall that,
\(\Q_{x_0}\)-a.s.\ \(dZ_t=-\frac12 Z_t \beta(\X_t)\,d\X_t\) for \(t < T_c \wedge T_{r_n}\).
Hence,
\(\Q_{x_0}\)-a.s.\ for \(t < T_c \wedge T_{r_n}\),
\begin{align*}
d(f(\s(\X_t))Z_t)
&=
Z_t d f (\s (\X_t)) + f (\s (\X_t)) d Z_t + d \langle f (\s (\X)), Z \rangle_t
\\[1mm]
&=
dM_t
+
Z_t g(\s(\X_t))\,dt +
\tfrac12 Z_t f'_-(\s(\X_t))\s''(\X_t)\,d\langle\X,\X\rangle_t
\\&\hspace{3.5cm}-
\tfrac12 Z_t \beta(\X_t) f'_-(\s(\X_t))\s'(\X_t)\,d\langle\X,\X\rangle_t
\\[0.5mm]
&=
dM_t
+
Z_t g(\s(\X_t))\,dt
\end{align*}
with some local $\Q_{x_0}$-martingale $M$,
where the integrals with respect to $\langle\X,\X\rangle$ are cancelled in the last equality due to $\s''=\beta\s'$ on $[c,r_n]$ (see Definition~\ref{def: non-sep int}).
Hence, $\Q_{x_0}$-a.s.\ for all \(t < T_c \wedge T_{r_n}\),
$$
d\Big[\Big( f (\s (\X_t)) - \int_0^{t} g (\s (\X_s))\,ds \Big) Z_t\Big]
=
dM_t
-
\Big(\int_0^{t} g (\s (\X_s))\,ds\Big)\,dZ_t,
$$
which proves that the process 
\[
\Big( f (\s (\X_{\cdot \wedge T_c \wedge T_{r_n}})) - \int_0^{\cdot \wedge T_c \wedge T_{r_n}} g (\s (\X_s)) ds \Big) Z_{\cdot \wedge T_c \wedge T_{r_n}}
\]
is a local \(\Q_{x_0}\)-martingale. Consequently, by \cite[Proposition III.3.8]{JS}, the stopped process 
\[
f (\s (\X_{\cdot \wedge T_c \wedge T_{r_n} \wedge \tilde{\sigma}^n_m})) - \int_0^{\cdot \wedge T_c \wedge T_{r_n} \wedge \tilde{\sigma}^n_m} g (\s (\X_s)) ds
\]
is a local \(\mathds{K}^{n, m}\)-martingale. By the Lemmata \ref{lem: generator} and \ref{lem: loc uni}, this proves~\eqref{eq: K eq}.

\smallskip
\noindent
\emph{Case 6 (\(\alpha = \gamma = \Delta\)):} In this case we have \(U \wedge V \wedge R = R\). Thus, if \(R = \delta\), i.e., the diffusions have a non-reflecting boundary point, the inequality \(S \leq \delta = U \wedge V \wedge R = \delta\) is trivial. We now discuss the case where \(R = \infty\), i.e., we assume that the boundaries of \((x \mapsto \P_x)\) and \((x \mapsto \Q_x)\) are reflecting. Notice that in this case
\(J=\tilde J\) is compact and, by~\eqref{eq:101022a3}, that
\(\m\) and \(\M\) are finite measures.
Further, thanks to Lemma~\ref{lem:200223a1}, both diffusions are recurrent.
Let \(\mathcal{C}\) be a countable set of bounded continuous functions \(f \colon J \to \mathbb{R}\) that is probability measure determining. 
It follows from the ratio ergodic theorem (see Lemma~\ref{lem: ratio ergodic thm})
that \(\P_{x_0}\)-a.s.
	\[
	\lim_{t \to \infty} \frac{1}{t} \int_0^t f (\X_s) ds = \frac{1}{\m (J)} \int f (x) \m (dx), \ \ \text{for all } f \in \mathcal{C}, 
	\]
	and \(\Q_{x_0}\)-a.s.
	\[
	\lim_{t \to \infty} \frac{1}{t} \int_0^t f (\X_s) ds = \frac{1}{\M (J)} \int f (x) \M (dx), \ \ \text{for all } f \in \mathcal{C}.
	\]
	For contradiction, assume that \(\P_{x_0} (S = \delta) > 0\). Then, by the definition of the separating time,
	\[
	\Q_{x_0} \Big( 
	\lim_{t \to \infty} \frac{1}{t} \int_0^t f (\X_s) ds = \frac{1}{\m (J)} \int f (x) \m (dx), \ \text{for all } f \in \mathcal{C}
	, \ S = \delta \Big) > 0,
	\]
	and consequently, 
	\[
	\Q_{x_0} \Big( 
	\frac{1}{\M (J)} \int f (x) \M(dx) = \lim_{t \to \infty} \frac{1}{t} \int_0^t f (\X_s) ds = \frac{1}{\m (J)} \int f (x) \m (dx), \ \text{for all } f \in \mathcal{C} \Big) > 0.
	\]
	By the uniqueness theorem for probability measures, and the assumption that \(\mathcal{C}\) is probability measure determining, we conclude that
$\m/\m (J) = \M/\M(J)$, or, equivalently,
\(
\m=c\M,
\)
where $c$ is a constant (necessarily, $c=\m(J)/\M(J)$).
As all points in \(J\) are non-separating and $\S=\on{Id}$, we observe that \(\s\) is continuously differentiable and
\(
\s'= 1 / c
\) on \(J\).
Recall
that a speed measure $\m$ is determined uniquely given the scale function $\s$ in the sense that
	if we replace $\s$ with $k\s+l$, $k>0$, $l\in\mathbb R$, then $\m$ is replaced with $\m/k$. Thanks to this observation, and the fact that speed and scale determine a diffusion uniquely, we conclude that \(\P_{x_0} = \Q_{x_0}\). This, however, is a contradiction to our general assumption that \(\P_{x_0} \neq \Q_{x_0}\). Hence, \(\P_{x_0}\)-a.s. \(S \leq \infty\) and, by Proposition~\ref{prop: AC Sing}, \(\P_{x_0}, \Q_{x_0}\)-a.s. \(S \leq \infty = U \wedge V \wedge R\).
\\

\noindent
The cases (\(\alpha = \Delta, x_0 < \gamma = r\)) and (\(l = \alpha < x_0 < \gamma = r\)) can be treated with the techniques developed in Cases 5 and~6. 
The only point where additional arguments are needed is the situation where \(\alpha = \Delta\), \(l\) is a reflecting boundary for (necessarily both) diffusions and \(r\) is either inaccessible for both diffusions or absorbing for both diffusions. 
The reason for this is that \(\X\) is no local $\Q_{x_0}$-martingale anymore.
On the contrary, in the proof of Case~5, the local martingale property of $\X$ was used twice,
namely in the definition of the candidate density \(Z\) and in the proof of~\eqref{eq: K eq}.
We now discuss the necessary changes. 

Suppose that \(\alpha = \Delta\),
\(x_0 < \gamma = r\),
that \(l\) is reflecting for both diffusions and that \(r\) is inaccessible for both diffusions.\footnote{The remaining case where $r$ is absorbing for both diffusions is handled via a sequence $x_0<r_1<r_2<\cdots$ such that $r_n\to r$ in exactly the same way as in Case~5.}
We then infer from Lemma~\ref{lem: finiteness hitting times} that $\P_{x_0},\Q_{x_0}$-a.s.\ $U\wedge V\wedge R=T_r=\infty$, that is, we need only to prove that $\P_{x_0}\perp\Q_{x_0}$.
To simplify our notation, we also assume that \(l = 0\).

Recall from Lemma~\ref{lem: occ formula diff}~(i) that \(\X\) is a \(\Q_{x_0}\)-semimartingale and denote its continuous local \(\Q_{x_0}\)-martingale part by \(\X^c\).
Quite similar as in Case 5 above, we define 
\[
Z \triangleq \exp \Big( - \frac{1}{2} \int_0^{\cdot} \beta (\X_s) d \X^c_s - \frac{1}{8} \int_0^{\cdot} [ \beta ( \X_s) ]^2 d \langle \X, \X \rangle_s \Big).
\]
Using that \(\beta \in L^2_\textup{loc} ([0, r))\), it follows as in Case~5 that \(Z\) is well-defined as a continuous local \(\Q_{x_0}\)-martingale.
We, further, observe that
\begin{equation}\label{eq:040423a1}
\Q_{x_0}\text{-a.s.}\quad Z_{\infty-} (= Z_{T_r-}) = 0.
\end{equation}
Indeed, if $r<\infty$, then \eqref{eq:040423a1} follows as in Case~5 from Lemma~\ref{lem: pepetual integral diffusion} due to the fact that \(r\) necessarily fails to be half-good.
In the opposite case $r=\infty$, the diffusion $(x\mapsto\Q_x)$ is recurrent by Lemma~\ref{lem:200223a1}, hence \eqref{eq:040423a1} follows from Lemma~\ref{lem:040423a1}, once we establish that $\beta$ is non-vanishing.
Now, assuming that $\beta$ vanishes (a.e. with respect to the Lebesgue measure) and noting that all points in $[0,r)\;(=J=\tilde J)$ are non-separating
yields that $\s'$ equals some positive constant $c$ on $[0,r)$ and, hence, $\m=\frac1c\M$
(recall Definitions \ref{def: non-sep int} and~\ref{def: non-sep bound}).
As discussed in Case~6, this would imply $\P_{x_0}=\Q_{x_0}$, which contradicts to our general assumption $\P_{x_0}\ne\Q_{x_0}$ and concludes the proof of~\eqref{eq:040423a1}.
Next, let \(\sigma\) be an \((\mathcal{F}^o_t)_{t \geq 0}\)-stopping time such that \(\E^{\Q_{x_0}}[Z_{\sigma}] = 1\). Then, we can define a probability measure \(\mathds{K}\) by the formula
\[
\mathds{K} (G) \triangleq \E^{\Q_{x_0}} \big[ Z_{\sigma} \1_G \big], \quad G \in \mathcal{F}. 
\]
Below, we prove that \(\mathds{K} = \P_{x_0}\) on \(\mathcal{F}^o_{\sigma}\).
Once this identity is established, we can reuse arguments from Case~5 and deduce \(\Q_{x_0} \perp \P_{x_0}\) from the fact that \(\Q_{x_0}\)-a.s. \(Z_{\infty-} = 0\) and Jessen's theorem. We omit this part of the proof and concentrate on the proof for \(\mathds{K} = \P_{x_0}\) on \(\mathcal{F}^o_{\sigma}\).
Again, as in Case~5, we use a martingale problem argument.
Take \(f \in C_b([\s (0), \s(r)) ; \mathbb{R})\) such that \(f'_+\) exists on \([\s(0), \s (r))\) as a right-continuous function of locally finite variation, \(df'_+ = 2g d \m \circ \s^{-1}\) on \((\s(0), \s(r))\) and \(f'_+(\s(0)) = 2g (\s (0))\m(\{0\})\) for some \(g \in C_b([\s(0), \s(r)); \mathbb{R})\). 
We deduce from Lemma~\ref{lem: compensation reflecting}, that 
\(\X^c = \X - \frac{1}{2} {L}^{0} (\X),\)
where \(L(\X)=\{ {L}^y_t (\X) \colon (t, y) \in \mathbb{R}_+ \times J\}\) denotes the semimartingale local time of \(\X\) under \(\Q_{x_0}\).
For what follows, we notice that $L(\X)$ is jointly continuous on
$\mathbb R_+\times J$ ($= \mathbb R_+\times[0,r)$).
This follows from~\eqref{eq:241124a1} and the joint continuity of the diffusion local time $\widehat L(\X)$ in Lemma~\ref{lem: occ formula diff}.
As all points in $[0, r)$ are non-separating, $\s$ is a $C^1$-function with absolutely continuous derivative on $[0, r)$ (recall the Definitions~\ref{def: non-sep int} and \ref{def: non-sep bound}~(iii)).
Thus, we can apply Lemma~\ref{lem: extended gen Ito formula} with \(\s\) and obtain that \(\Q_{x_0}\)-a.s. 
\begin{equation} \label{eq: s ito ref} \begin{split}
	d \s (\X_t) &= \s' (\X_t) d \X_t + \tfrac{1}{2} \s'' (\X_t) d \langle \X, \X \rangle_t
	\\&= \s' (\X_t) d \X^c_t + \tfrac{1}{2} \s' ( 0 ) d{L}^0_t (\X) + \tfrac{1}{2} \s'' (\X_t) d \langle \X, \X \rangle_t.
\end{split}
\end{equation} 
Using that \(d f'_+ = 2 g d \m \circ \s^{-1}\) on \(\s (J^\circ)\), \(d\M = \s' d\m\) on \(J^\circ\) (see Definition~\ref{def: non-sep int}) and Lemma~\ref{lem: occ smg}~(iii), we obtain \(\Q_{x_0}\)-a.s.
\begin{equation} \label{eq: local inte comp} \begin{split}
	\tfrac{1}{2} \int_{(\s (0), \s(r))} {L}^{x}_t (\s (\X)) f'_+ (dx)  &= \int_{(0, r)} {L}^{\s (x)}_t (\s (\X)) g (\s (x)) \m (dx) = \int_{(0, r)} {L}^{x}_t (\X) g (\s (x)) \M (dx).
\end{split}
\end{equation} 
Using Lemma~\ref{lem: extended gen Ito formula} with \(f\), \eqref{eq: s ito ref}, \eqref{eq: local inte comp}, \(\frac12 f'_+ (\s (0)) \s' (0) = g (\s (0)) \m (\{0\}) \s' (0) = g (\s (0)) \M(\{0\})\), which uses the fact that \(0\) is non-separating (see Definition~\ref{def: non-sep bound}~(iii)), and the occupation time formula for diffusions (see Lemma~\ref{lem: occ formula diff}, in particular, formulas \eqref{eq:241124a1} and~\eqref{eq:020423a5}),
we get \(\Q_{x_0}\)-a.s.
\begin{equation} \label{eq: main comp rel ito}
	\begin{split}
	d  f (\s (\X_t)) &= f'_+ (\s (\X_t)) d \s (\X_t) + \tfrac{1}{2} \, d\int_{(\s(0),\s(r))} {L}^{x}_t (\s (\X)) f'_+ (dx) 
	\\&= d M_t + \tfrac{1}{2} f'_+ (\s (0)) \s' (0) d {L}^0_t (\X) + \tfrac{1}{2} f'_+ (\s (\X_t)) \s'' (\X_t) d \langle \X, \X\rangle_t + d\int_{(0, r)} {L}^{x}_t (\X) g (\s (x)) \M (dx)
	\\&=d M_t  + g (\s (0)) \M (\{0\}) d {L}^{0}_t (\X) + \tfrac{1}{2} f'_+ (\s (\X_t)) \s'' (\X_t) d \langle \X, \X\rangle_t + d\int_{(0, r)} {L}^{x}_t (\X) g (\s (x)) \M (dx)
	\\&=d M_t  + \tfrac{1}{2} f'_+ (\s (\X_t)) \s'' (\X_t) d \langle \X, \X\rangle_t + d\int_{[0, r)} {L}^{x}_t (\X) g (\s (x)) \M (dx)
	\\&= d M_t + \tfrac{1}{2} f'_+ (\s (\X_t)) \s'' (\X_t) d \langle \X, \X\rangle_t + g (\s (\X_t)) dt, \phantom \int
\end{split}
\end{equation} 
where \(d M_t = f'_+ (\s (\X_t)) \s' (\X_t) d \X^c_t\).
Then \(\Q_{x_0}\)-a.s. it holds
\begin{align} \label{eq: QV ref pf}
	d \langle Z, f (\s (\X)) \rangle_t = -\frac12 Z_t \beta (\X_t) f'_+ (\X_t) \s' (\X_t) d \langle \X, \X \rangle_t
	= -\frac12 Z_t f'_+ (\s (\X_t)) \s'' (\X_t) d\langle \X, \X \rangle_t,
\end{align}
where the second equality follows from the fact that \(\mu_L\)-a.e. \(\s'' = \beta \s'\)
(see Definition~\ref{def: non-sep int})
and the semimartingale occupation time formula (Lemma~\ref{lem: occ smg}).
By integration by parts, \eqref{eq: main comp rel ito} and \eqref{eq: QV ref pf},
we obtain that \(\Q_{x_0}\)-a.s.
\begin{align*}
	d f (\s (\X_t)) Z_t &= f (\s (\X_t)) dZ_t + Z_t d f (\X_t) + d \langle Z, f (\X_t) \rangle_t 
	= f (\s (\X_t)) dZ_t + Z_t d M_t + Z_t g (\s (\X_t)) dt.
\end{align*}
Hence, as in Case 5, using \cite[Proposition III.3.8]{JS}, we get that \(f (\s (\X_{\cdot \wedge \sigma})) - \int_0^{\cdot \wedge \sigma} g (\s (\X_s)) ds\) 
is a local \(\mathds{K}\)-martingale and, by Lemmata \ref{lem: generator} and \ref{lem: loc uni}, \(\mathds{K} = \P_{x_0}\) on \(\mathcal{F}^o_{\sigma}\). This finishes our discussion of this case.

\smallskip 
\noindent
Up to symmetry we considered all possible cases.\qed

\appendix

\section{Examples for Separating Times} \label{sec: examples}
In the first example we relate our definition of non-separating (good) points to those from \cite{cherUru,MU12} for the It\^o diffusion setting. Not surprisingly, we will see that the definitions coincide in this case.

\begin{example}[It\^o diffusion setting] \label{ex: ito diffusion}
	Let \(J^\circ = \tilde{J}^\circ = (l, r)\) and take four Borel measurable functions \(b, \tilde{b}, \sigma, \tilde{\sigma} \colon J^\circ \to \mathbb{R}\) such that the Engelbert--Schmidt conditions hold, i.e.,
	\[
	\sigma^2, \tilde{\sigma}^2 > 0\;\;\text{everywhere on }J^\circ, \qquad \frac{1 + |b|}{\sigma^2}, \frac{1 + |\tilde{b}|}{\tilde{\sigma}^2} \in L^1_\textup{loc}(J^\circ).
	\]
	Suppose that \(x_0 \in J^\circ\) and define
	\[
	\s (x) \triangleq \int^x \exp \Big( - 2 \int^z \frac{b (y) dy}{\sigma^2(y)}\Big) dz, \qquad \S (x) \triangleq \int^x \exp \Big( - 2 \int^z \frac{\tilde{b} (y) dy}{\tilde{\sigma}^2(y)}\Big) dz,
	\]
	and 
	\begin{equation}\label{eq:130222b3}
		\m (dx) \triangleq \frac{dx}{\s' (x) \sigma^2(x)}, \qquad \M (dx) \triangleq \frac{dx}{\S'(x) \tilde{\sigma}^2(x)}.
	\end{equation}
	Moreover, we suppose that \(\m (\{l\}) \equiv \m (\{r\}) \equiv \M(\{l\}) \equiv \M(\{r\}) \equiv \infty\) in case the points are accessible. In other words, we stipulate that the diffusions are absorbed in the boundaries of their state spaces in case they can reach them in finite time. 
	Providing some intuition, in this case \(\P_{x_0}\) is the law of an It\^o diffusion $X$ with dynamics 
	\[
	d X_t = b(X_t) dt + \sigma (X_t) dW_t, \quad X_0 = x_0, 
	\]
	up to the hitting time of the boundaries,
	and \(\Q_{x_0}\) is the law of an It\^o diffusion $Y$ with dynamics 
	\[
	d Y_t = \tilde{b}(Y_t) dt + \tilde{\sigma} (Y_t) dB_t, \quad Y_0 = x_0, 
	\]
	up to the hitting time of the boundaries, where $W$ and $B$ are standard Brownian motions.
	We refer to \cite[Section~5.5]{KaraShre} for precise definitions.
	
	Let us now understand Definitions \ref{def: non-sep int}, \ref{def:130222a1} and~\ref{def: non-sep bound} in this specific setting.
	First, the differential quotient $d^+\s/d\S$ clearly exists everywhere on $J^\circ$ and it equals
	\[
	\frac{d^+ \s}{d \S} = \frac{\s'}{\S'} = \exp \Big(2 \int^\cdot \frac{\tilde{b}(y) dy}{\tilde{\sigma}^2(y)} - 2 \int^\cdot \frac{b(y)dy}{\sigma^2(y)}\Big).
	\]
	In particular, $d^+\s/d\S$ is an absolutely continuous function.
	This yields that \eqref{eq:050122a1} is satisfied on $J^\circ$ with
	\[
	\beta = \frac{d^2 \s}{d \S d \s} = \Big( \frac{\s'}{\S'}\Big)' \frac{1}{\s'} = \frac{2 }{\S'} \Big(\frac{\tilde{b}}{\tilde{\sigma}^2} - \frac{b}{\sigma^2}\Big).
	\]
	Let $\mu_L$ denote the Lebesgue measure.
	We deduce from~\eqref{eq:130222b3} that
	\[
	\frac{d \m}{d \M} = \frac{\S' \tilde{\sigma}^2}{\s' \sigma^2}\;\;\mu_L\text{-a.e. on }J^\circ
	\]
	(more precisely, this holds for any Lebesgue point, see \cite[Chapter~7]{rudin}).
	As in Definition~\ref{def: non-sep int}, we now consider a point $x\in J^\circ$ and an open neighborhood $U(x)\subset J^\circ$ of $x$. We see that
	\[
	\frac{d \m}{d \M}\frac{d^+ \s}{d \S} =1\;\;\text{on }U(x)\quad \Longrightarrow \quad\tilde{\sigma}^2 = \sigma^2\;\;\mu_L\text{-a.e. on }U(x).
	\]
	Conversely, if \(\tilde{\sigma}^2 = \sigma^2\) $\mu_L$-a.e. on $U(x)$, then we get everywhere on $U(x)$
	\[
	\frac{d \m}{d \M} = \frac{\S'}{\s'}
	\]
	by the continuity of \(\S'\) and \(\s'\) and the mean value theorem for Riemann--Stieltjes integrals (\cite[p. 197]{walter}). In summary, we have 
	\[
	\frac{d \m}{d \M}\frac{d^+ \s}{d \S}=1\;\;\text{on }U(x) \quad \Longleftrightarrow \quad \tilde{\sigma}^2 = \sigma^2\;\;\mu_L\text{-a.e. on }U(x).
	\]
	Now, suppose that $\mu_L$-a.e. on $U(x)$ we have \(\tilde{\sigma}^2 = \sigma^2\).
	Hence, $\mu_L$-a.e. on $U(x)$
	\[
	\beta = \frac{2}{\S'} \Big(\frac{\tilde{b} - b}{\sigma^2}\Big)
	\]
	and 
	\begin{align}\label{eq: cond beta comp Ito setting}
		\big( \beta (z)\big)^2 \, \S (dz) =  \frac{4}{\S' (z)} \frac{(\tilde{b}(z) - b(z))^2}{\sigma^4(z)} \,dz.
	\end{align}
	As \(\S'\) is positive and continuous on \(J^\circ\), provided that \(\on{cl}(U(x)) \subset J^\circ\), we have
	\[
	\int_{U(x)} \big( \beta (z) \big)^2 \, \S(dz) < \infty \quad \Longleftrightarrow \quad \int_{U(x)} \frac{(\tilde{b}(z) - b(z))^2}{\sigma^4(z)} \,dz< \infty. 
	\]
	We conclude that a point \(x \in J^\circ\) is non-separating in the sense of Definition \ref{def: non-sep int} if and only if there exists an open neighborhood \(U(x) \subset J^\circ\) of \(x\) such that \(\tilde{\sigma}^2 = \sigma^2\) $\mu_L$-a.e. on \(U(x)\) and \((\tilde{b} - b)^2 / \sigma^4 \in L^1 (U(x))\).
	This is precisely the definition of a non-separating (good) interior point from \cite{cherUru, MU12}.
	Recalling \eqref{eq:130222a1} and~\eqref{eq: cond beta comp Ito setting}, we also see that, for the boundary points \(l\) and \(r\), Definition~\ref{def:130222a1} of half-goodness coincides with the definition of a non-separating (good) boundary point from \cite{cherUru, MU12}.
	It remains to notice that in this case, where accessible boundaries are forced to be absorbing, a boundary point is non-separating (good) in the sense of Definition~\ref{def: non-sep bound} if and only if it is half-good in the sense of Definition~\ref{def:130222a1}.\footnote{Definition~\ref{def: non-sep bound} starts to be essential as long as we include instantaneously or slowly reflecting boundaries into consideration.}
	In summary, Theorem \ref{theo: main1} includes \cite[Theorem~5.1]{cherUru} and \cite[Theorem~5.5]{MU12}.
\end{example}

Example \ref{ex: ito diffusion} gives the impression that \(d \m / d \M \cdot d^+ \s / d \S  = 1\) means that the diffusion coefficients coincide, which is well-known to be necessary for (local) absolute continuity, recall Girsanov's theorem (Lemma~\ref{lem: Girs}). However, the condition encodes much more as the following example illustrates.

\begin{example}[Sticky Brownian motions] \label{ex: sticky}
	Suppose that \(J = \tilde{J} = \mathbb{R}\) and that \((x \mapsto \P_x)\) and \((x \mapsto \Q_x)\) are Brownian motions sticky at the origin. More precisely, we assume that both \((x \mapsto \P_x)\) and \((x \mapsto \Q_x)\) are on natural scale and that 
	\[
	\m (dx) = dx + \gamma \delta_0 (dx), \qquad \M(dx) = dx + \tilde{\gamma} \delta_0 (dx), 
	\]
	where \(\gamma, \tilde{\gamma} \in (0, \infty)\). It is well-known (\cite{bass,EngPes}) that \(\P_{x_0}\) is the law of a solution process to
	\begin{equation}\label{eq:130222b1}
		d X_t = \1_{\{X_t \not = 0\}} dW_t, \quad \1_{\{X_t = 0\}} dt = \gamma d L^0_t (X),\quad X_0 = x_0, 
	\end{equation}
	and \(\Q_{x_0}\) is the law of a solution process to
	\begin{equation}\label{eq:130222b2}
		d Y_t = \1_{\{Y_t \not = 0\}} dB_t, \quad \1_{\{Y_t = 0\}} dt = \tilde{\gamma} d L^0_t (Y), \quad Y_0 = x_0, 
	\end{equation}
	where \(L^0 (X)\) and \(L^0 (Y)\) denote the semimartingale local times of \(X\) and \(Y\) in zero,
	and $W$ and $B$ are standard Brownian motions.
	In this setting we clearly have 
	\(
	d^+ \s/d \S = 1
	\)
	and 
	\[
	\Big(\frac{d \m}{d \M} \Big) (z) = \begin{cases} 1,& z \not = 0,\\
		\gamma/\tilde{\gamma},& z = 0.\end{cases}
	\]
	Consequently, the origin is separating if and only if \(\gamma \not = \tilde{\gamma}\). All other points in \(\mathbb{R}\) are non-separating.
	Finally, since \(\s(\pm \infty) = \S(\pm \infty) = \pm \infty\), the boundary points \(\pm \infty\) are separating.
	In summary, 
	\[
	\A = \begin{cases} \{- \infty, \infty\},&\gamma = \tilde{\gamma},\\
		\{- \infty, 0, \infty\},&\gamma \not = \tilde{\gamma}. \end{cases}
	\]
	With this observation at hand, we can deduce the following result from our main theorem.
	\begin{corollary} \label{coro: sep sticky}
		Let \(S\) be the separating time for \(\P_{x_0}\) and \(\Q_{x_0}\).
		\begin{enumerate}
			\item[\textup{(i)}]
			If \(\gamma = \tilde{\gamma}\), then \(\P_{x_0}, \Q_{x_0}\)-a.s. \(S = \delta\).
			\item[\textup{(ii)}]
			If \(\gamma \not = \tilde{\gamma}\), then \(\P_{x_0}, \Q_{x_0}\)-a.s. \(S = T_0\) and, in particular, \(\P_0 \perp \Q_0\) on \(\mathcal{F}_0\).
		\end{enumerate}
	\end{corollary}
	\begin{proof}
		Of course, in case \(\gamma = \tilde{\gamma}\) we have \(\P_{x_0} = \Q_{x_0}\) and hence, (i) is trivial. We now discuss part (ii) and therefore assume that \(\gamma \not = \tilde{\gamma}\).
		Clearly, we have $R=\delta$.
		If \(x_0 = 0\), then \(U = V = T_0 = 0\) and the claim follows.
		Suppose now that \(x_0 > 0\) (resp. \(x_0 < 0\)). By virtue of Lemma~\ref{lem:200223a2}, we get \(\P_{x_0}, \Q_{x_0}\)-a.s. \(U = T_0\) (resp. \(V = T_0\)) and \(\P_{x_0}, \Q_{x_0}\)-a.s. \(V = T_\infty = \infty\) (resp. \(U = T_{- \infty} = \infty\)).
		We conclude from Theorem~\ref{theo: main1} that \(\P_{x_0}, \Q_{x_0}\)-a.s. \(S = T_0\).
	\end{proof}
	
	We emphasize that in the second case the probabilities \(\P_{x_0}\) and \(\Q_{x_0}\) are never locally absolutely continuous.
	This observation is not surprising
	due to \eqref{eq:130222b1} and~\eqref{eq:130222b2},
	as the local time is preserved by a locally absolutely continuous change of measure (\cite[Exercise~29.17]{kallenberg}).
	On the other hand, (ii) tells us for instance that \(\P_1 \sim \Q_1\) on \(\mathcal{F}_{T_a}\) for all \(a \in (0, 1)\).
\end{example}

The existence of the differential \(d^+ \s/ d \S\) as a continuous function excludes certain skewness properties of the scale function. 

\begin{example}[Skew Brownian motions] \label{ex: skew}
	Suppose that \(J = \tilde{J} = \mathbb{R}\) and that \((x \mapsto \P_x)\) and \((x \mapsto \Q_x)\) are skew Brownian motions. More precisely, we assume that 
	\[
	\s (x) = \begin{cases}(1 - \alpha)x,& x \geq 0,\\ \alpha x,& x \leq 0,\end{cases}\qquad \S(x) = \begin{cases}(1 - \tilde{\alpha})x,& x \geq 0,\\ \tilde{\alpha}x,& x \leq 0,\end{cases}
	\]
	and
	\begin{align*}
		\m (dx) &= \big( (1 - \alpha)^{-1} \1_{\{x \geq 0\}} + \alpha^{-1} \1_{\{x < 0\}}\big) dx, \\ 
		\M(dx) &= \big( (1 - \tilde{\alpha})^{-1} \1_{\{x \geq 0\}} + \tilde{\alpha}^{-1} \1_{\{x < 0\}}\big) dx, 
	\end{align*}
	where \(\alpha, \tilde{\alpha} \in (0, 1)\). 
	We easily see that
	\[
	\Big(\frac{d^+ \s}{d \S}\Big) (x) = \begin{cases} (1 - \alpha) / (1 - \tilde{\alpha}),& x \geq 0,\\
		\alpha /\tilde{\alpha},& x < 0.
	\end{cases}
	\]
	Hence, in case \(\alpha \not = \tilde{\alpha}\) the origin is separating. Moreover, as
	\[
	\Big(\frac{d \m}{d \M}\Big) (x) = \begin{cases} (1 - \tilde{\alpha}) /(1 - \alpha),& x > 0,\\
		[(1 - \tilde{\alpha}) \tilde{\alpha}] / [(1 - \alpha) \alpha],& x = 0,\\
		\tilde{\alpha}/\alpha,& x < 0, \end{cases}
	\]
	we note that all interior points except the origin are non-separating regardless of \(\alpha\) and \(\tilde{\alpha}\). 
	Let us stress that (ii) as well as (iii) from Definition \ref{def: non-sep int} fail for the origin. Finally, as \(\s(\pm \infty) = \S(\pm \infty) = \pm \infty\), the boundary points \(\pm \infty\) are separating.
	Now, as in the proof of Corollary \ref{coro: sep sticky}, we get the following:
	\begin{corollary}
		Let \(S\) be the separating time for \(\P_{x_0}\) and \(\Q_{x_0}\).
		\begin{enumerate}
			\item[\textup{(i)}]
			If \(\alpha = \tilde{\alpha}\), then \(\P_{x_0}, \Q_{x_0}\)-a.s. \(S = \delta\).
			\item[\textup{(ii)}]
			If \(\alpha \not = \tilde{\alpha}\), then \(\P_{x_0}, \Q_{x_0}\)-a.s. \(S = T_0\) and, in particular, \(\P_0 \perp \Q_0\) on \(\mathcal{F}_0\).
		\end{enumerate}
	\end{corollary}
\end{example}

In \cite[Theorem 4.1]{cherUru} it was shown that the first hitting time of the origin is the separating time for two Bessel processes with different dimensions. In the next example we deduce this result from Theorem~\ref{theo: main1}.

\begin{example}[Generalized Bessel processes] \label{ex: bessel}
	For \(\gamma > 0\) and \(x_0 \in \mathbb{R}_+\), a solution \(Y\) to
	\begin{align}\label{eq: SDE Bessel}
		d Y_t = \gamma dt + 2 \sqrt{Y_t} d W_t, \quad Y_0 = x_0, \quad W = \text{Brownian motion},
	\end{align}
	is called a \emph{square Bessel process of dimension \(\gamma\)} started at \(x_0\). The number \(\nu \triangleq \gamma /2  - 1\) is called its \emph{index}.
	The square root $Z\triangleq\sqrt Y$
	is called a \emph{Bessel process of dimension \(\gamma\)}. A detailed discussion of (square) Bessel processes can be found in \cite[Chapter~XI]{RY}. 
	
	In the following we discuss the separating time for
	\emph{generalized} Bessel processes
	(in the sense that we allow for arbitrary boundary behavior in the origin whenever the origin is accessible).
	More precisely, we take \(J^\circ = (0, \infty)\) and \((x \mapsto \P^\gamma_x)\) to be a regular diffusion with scale function, for \(x > 0\),
	\[
	\s^\gamma (x) \triangleq \begin{cases} - \on{sgn}(\nu) x^{- 2 \nu}, & \nu \not = 0,\\ 2 \log (x),& \nu = 0,\end{cases}
	\]
	and speed measure, on \(\mathcal{B}((0, \infty))\),
	\[
	\m^\gamma (dx) \triangleq \begin{cases} x^{2 \nu + 1} dx/ 2|\nu|,& \nu \not = 0,\\ \frac{1}{2} x\, dx,& \nu = 0.\end{cases}
	\]
	The above scale and speed coincide with those of the Bessel process $Z$.
	For all $\gamma>0$ we have that $\infty$ is inaccessible.
	In case \(0<\gamma < 2\) the origin is regular
	(in particular, $J=[0,\infty)$)
	and in case \(\gamma \geq 2\) the origin is inaccessible
	(so, $J=(0,\infty)$).
	For the square Bessel process as defined in \eqref{eq: SDE Bessel} with \(0<\gamma < 2\) the origin is instantaneously reflecting, which corresponds to \(\m^\gamma (\{0\}) = 0\) for the Bessel process.
	For our \emph{generalized} Bessel process,
	we allow for all values \(\m^\gamma (\{0\}) \in [0, \infty]\) in case~\(0<\gamma < 2\).
	Notice that stopping at the origin is included as particular case $\m^\gamma (\{0\})=\infty$.
	
	From now on, fix \(x_0,\gamma, \tilde{\gamma} > 0\) such that \(\gamma \not = \tilde{\gamma}\).
	\begin{lemma} \label{lem: sep}
		Recalling that \(\A \subset [0, \infty]\) is the set of separating points for \(\P_{x_0}^\gamma\) and \(\P_{x_0}^{\tilde{\gamma}}\), 
		\(\A = \{0, \infty\}\).
	\end{lemma}
\begin{proof}
	We write \(\tilde{\nu} \triangleq \tilde{\gamma}/2 - 1\). 
	In the following we investigate which points in \([0, \infty]\) are separating.
	If \(\gamma, \tilde{\gamma} \not = 2\), we get, for \(x > 0\),
	\begin{align*}
		\Big(\frac{d^+ \s^{\gamma}}{d \s^{\tilde{\gamma}}} \Big) (x) = \frac{(\s^{\gamma})' (x)}{(\s^{\tilde{\gamma}})' (x)} = \Big| \frac{\nu}{\tilde{\nu}}\Big| x^{2 (\tilde{\nu}-\nu)}, \qquad
		\Big(\frac{d \m^{\gamma}}{d \m^{\tilde{\gamma}}} \Big) (x) = \Big| \frac{\tilde{\nu}}{\nu}\Big| x^{2 (\nu - \tilde{\nu})}, 
	\end{align*}
	and
	\begin{align*}
		\beta (x) \equiv \frac{(d^+ \s^{\gamma}/d \s^{\tilde{\gamma}})' (x)}{(\s^{\gamma})' (x)} = \frac{(\tilde{\nu} - \nu)}{|\tilde{\nu}|} x^{2 \tilde{\nu}}.	
	\end{align*}
	Further, if \(\gamma = 2\) and \(\tilde{\gamma} \not = 2\), we obtain, for \(x > 0\),
	\begin{align*}
		\Big(\frac{d^+ \s^{\gamma}}{d \s^{\tilde{\gamma}}} \Big) (x) = \frac{x^{2 \tilde{\nu}}}{|\tilde{\nu}|},
		\qquad
		\Big(\frac{d \m^{\gamma}}{d \m^{\tilde{\gamma}}} \Big) (x) = \frac{|\tilde{\nu}|}{x^{2 \tilde{\nu}}},
		\qquad
		\beta (x) = 2 \on{sgn}(\tilde{\nu}) x^{2 \tilde{\nu}}.
	\end{align*}
	The case where \(\gamma \not = 2\) and \(\tilde{\gamma} = 2\) looks similar.
	Thus, all points in \((0, \infty)\) are non-separating. 
	
	Next, we discuss the boundary points.
	Notice the following:
	\begin{center}
		\begin{tabular}{ c | c c c} 
			\(\tilde{\gamma}\) & \(< 2\) & \(= 2\) & \(> 2\) \\
			\(\s^{\tilde{\gamma}}(\infty)\) & \(\infty\) & \(\infty\) & 0
			\\
			\(\s^{\tilde{\gamma}}(0)\) & 0 &\(-\infty\) & \(-\infty\)	
		\end{tabular}		
	\end{center}
	Thus, \(\infty\) is separating if \(\tilde{\gamma} \leq 2\) and \(0\) is separating if \(\tilde{\gamma} \geq 2\).
	For \(\tilde{\gamma} < 2\) the origin is separating, because
	\begin{align*}
		\int_{0 + } | \s^{\tilde{\gamma}} (x) - \s^{\tilde{\gamma}}(0)| \big( \beta (x) \big)^2 \, \s^{\tilde{\gamma}} (dx) &\overset{c}= \int_{0 +}  x^{- 2 \tilde{\nu}}x^{4\tilde{\nu}} x^{- 2 \tilde{\nu} - 1} dx = \int_{0 +} \frac{dx}{x} = \infty, 
	\end{align*}
	where \(\overset{c}=\) denotes an equality up to a positive multiplicative constant, and for \(\tilde{\gamma} > 2\) the boundary \(\infty\) is separating, as 
	\begin{align*}
		\int^{\infty-} | \s^{\tilde{\gamma}} (x) - \s^{\tilde{\gamma}}(\infty)| \big( \beta (x) \big)^2 \, \s^{\tilde{\gamma}} (dx) &\overset{c}= \int^{\infty-} \frac{dx}{x} = \infty.
	\end{align*}
	We conclude that \(\A = \{0, \infty\}\).
\end{proof}
	
	Thanks to this lemma, we obtain the following representation of the separating time for  \(\P^\gamma_{x_0}\) and \(\P^{\tilde{\gamma}}_{x_0}\). 
	\begin{corollary} \label{coro: bessel}
		Let \(S\) be the separating time for \(\P_{x_0}^{\gamma}\) and \(\P_{x_0}^{\tilde{\gamma}}\). Then, \(\P_{x_0}^\gamma, \P_{x_0}^{\tilde{\gamma}}\)-a.s. \(S =T_0\).
	\end{corollary}
\begin{proof}
	Clearly, $R=\delta$.
	Notice the following:
	\begin{center}
		\begin{tabular}{ c | c c c} 
			\(\gamma\) & \(< 2\) & \(= 2\) & \(> 2\) \\
			\(\s^{\gamma}(\infty)\) & \(\infty\) & \(\infty\) & 0
			\\
			\(\s^{\gamma}(0)\) & 0 &\(-\infty\) & \(- \infty\)	
		\end{tabular}		
	\end{center}
	Thus, by virtue of Lemmata \ref{lem: sep} and~\ref{lem:200223a2}, we obtain \(\P^\gamma_{x_0}\)-a.s.
	\begin{align*}
		U &= \begin{cases}
			T_0, & \gamma \le 2,\\ \delta,& \gamma > 2,
		\end{cases}
		\qquad
		V =  \begin{cases}
			T_\infty (= \infty), & \gamma < 2,\;\m^\gamma (\{0\}) < \infty,\\
			\delta,& \gamma < 2,\;\m^\gamma (\{0\}) = \infty,\\
			T_\infty (= \infty = T_0),& \gamma \geq 2
		\end{cases}
	\end{align*}
	(notice that, in the case $\gamma < 2$ and $\m^\gamma (\{0\}) < \infty$,
	Lemma~\ref{lem: finiteness hitting times} yields $\P^\gamma_{x_0}$-a.s.\ $\limsup_{t\nearrow T_\infty}X_t=\infty$),
	therefore, \(\P^\gamma_{x_0}\)-a.s. \(U \wedge V\wedge R = T_0\). As this computation is independent of \(\gamma\), the corollary follows from Theorem~\ref{theo: main1}.
\end{proof}
	\begin{remark}
		It is interesting to observe that separability of the origin is independent of the boundary behavior, i.e., in particular of its attainability or the values \(\m^{\gamma}(\{0\})\) and \(\m^{\tilde{\gamma}}(\{0\})\) in the attainable case. This fact shows that equivalence of \(\P^{\gamma}_{x_0}\) and \(\P_{x_0}^{\tilde{\gamma}}\) is already lost \emph{at the time the origin is hit} and that the separating time is not affected by stopping at the origin.
		This is different for sticky and skew Brownian motions (see Examples \ref{ex: sticky} and~\ref{ex: skew}). Indeed, if these processes are stopped in the origin, they coincide, which means they are trivially equivalent. In the following, we present a non-trivial example of two processes whose equivalence is lost \emph{right after} the time a boundary point is hit (but not \emph{at} this time).
	\end{remark}
\end{example}

\begin{example} \label{ex: loss after hit}
	We take \(J = \tilde{J} = \mathbb{R}_+,\) and, for \(x \in \mathbb{R}_+\),
	\begin{align*}
		\s (x) &= \int^x \exp \Big( - \int^y \frac{dz}{\sqrt{z}}\Big) dy
		= \text{const } \int^x e^{-2 \sqrt{y}} dy \quad
		\Big( = \text{const } - \text{ const } e^{-2 \sqrt{x}} \Big( \sqrt{x} + \frac12 \Big) \Big),
		\\
		\S (x) &= x,
	\end{align*}
	and, on \(\mathcal{B}((0, \infty))\),
	\[
	\m (dx) = \frac{dx}{\s' (x)},\qquad \M(dx) = dx.
	\]
	Notice that 
	\[
	\int^\infty \big( \s (\infty) - \s (x) \big) \m (dx)
	=
	\int^\infty \left( \sqrt{x} + \frac12 \right)\, dx = \infty.
	\]
	It follows from~\eqref{eq:160223a1} that \(\infty\) is not accessible for the diffusion with characteristics \((\s, \m)\). Clearly, the same is true for the diffusion with characteristics \((\S, \M)\).
	Furthermore, \eqref{eq:101022a3} yields that the origin is regular for both diffusions.
	Therefore, the values \(\m(\{0\}), \M(\{0\})\) can be arbitrarily chosen from $[0,\infty]$.
	We thus take an arbitrary $\m(\{0\})\in[0,\infty]$ and an arbitrary $\M(\{0\})\in[0,\infty]$.
	
	Fix \(x_0 > 0\). Providing an intuition, under \(\P_{x_0}\), we have
	\[
	d \X_t = \frac{dt}{2 \sqrt{\X_t}} + d W_t, \quad t < T_0, 
	\]
	and, under \(\Q_{x_0}\), we have
	\(
	d \X_t = d B_t\) for \(t < T_0,\) where \(W\) and \(B\) are standard Brownian motions.
	The behavior after $T_0$ under $\P_{x_0}$ (resp., $\Q_{x_0}$) depends on the choice of $\m(\{0\})$ (resp., $\M(\{0\})$).
	
	\begin{lemma} \label{lem: sep modified Bessel}
		Recalling that \(\A\subset [0, \infty]\) is the set of separating points for \(\P_{x_0}\) and \(\Q_{x_0}\),
		\[
		\A = \begin{cases}
			\{\infty\},& \m (\{0\}) = \M(\{0\}) = \infty,\\
			\{0, \infty\},& \m (\{0\}) \wedge \M(\{0\}) < \infty. 
		\end{cases}
		\]
	\end{lemma}
\begin{proof}
	We compute
	$$
	\Big(\frac{d^+ \s}{d \S}\Big) (x)
	=
	\s'(x)
	=
	e^{-2\sqrt x},
	\quad
	\beta(x)
	\equiv
	\frac{d^+}{d\s}  \Big(\frac{d^+ \s}{d \S}\Big) (x)
	=
	\frac{(d^+\s/d\S)'(x)}{\s'(x)}
	=
	-\frac{1}{\sqrt{x}}.
	$$
	It is easy to see that all points in \((0, \infty)\) are non-separating. As \(\S(\infty) = \infty\), \(\infty\) is separating.
	Notice that 
	\[
	\int_{0 +} \big| \S(x) - \S(0)\big| \big(\beta(x)\big)^2 \, \S(dx) = \int_{0+} dx < \infty,
	\]
	that is, the origin is in any case half-good, but
	\[
	\int_{0 +} \big(\beta(x)\big)^2 \, \S(dx) = \int_{0 +} \frac{dx}{x} = \infty.
	\]
	Hence, the origin is non-separating if and only if \(\m (\{0\}) = \M(\{0\}) = \infty\). 
\end{proof}
	
	With this representation of \(\A\) at hand, one can deduce from Theorem~\ref{theo: main1} the following result on the separating time. 
	
	\begin{corollary}\label{coro: example equ lost after hit}
		Let \(S\) be the separating time for \(\P_{x_0}\) and \(\Q_{x_0}\).
		\begin{enumerate}
			\item[\textup{(i)}]
			Let \(\m(\{0\}) = \M(\{0\}) = \infty\). Then, the following hold:
			\begin{itemize}
				\item
				\(\Q_{x_0}\)-a.s. \(S = \delta\), while \(\P_{x_0}\)-a.s. \(S \geq \infty\).
				
				\item
				$\P_{x_0}(S=\infty)>0$ and $\P_{x_0}(S=\delta)>0$.
				
				\item
				We have the following mutual arrangement between $\P_{x_0}$ and $\Q_{x_0}$ from the viewpoint of their (local) absolute continuity and singularity:
				$$
				\P _{x_0}\sim_\textup{loc}\Q_{x_0},\quad
				\Q_{x_0}\ll\P_{x_0},\quad
				\P_{x_0}\not\ll\Q_{x_0},\quad
				\P_{x_0}\not\perp\Q_{x_0}.
				$$
			\end{itemize}
			
			\item[\textup{(ii)}]
			Let \(\m (\{0\}) \wedge \M(\{0\}) < \infty\). Then, the following hold:
			\begin{itemize}
				\item
				\(\P_{x_0}, \Q_{x_0}\)-a.s. \(S = T_0\).
				
				\item
				$\Q_{x_0}(T_0<\infty)=1$,
				$\P_{x_0}(T_0<\infty)>0$,
				$\P_{x_0}(T_0=\infty)>0$.
				
				\item
				We have the following mutual arrangement between $\P_{x_0}$ and $\Q_{x_0}$ from the viewpoint of their (local) absolute continuity and singularity:
				$$
				\P _{x_0}\not\ll_\textup{loc}\Q_{x_0},\quad
				\Q _{x_0}\not\ll_\textup{loc}\P_{x_0},\quad
				\P_{x_0}\perp\Q_{x_0}.
				$$
			\end{itemize}
		\end{enumerate}
	\end{corollary}
\begin{proof}
	Let us first notice the following:
	\begin{center}
		\begin{tabular}{ c c c c } 
			\(\s(0)\)&\(\s(\infty)\)&\(\S(0)\)&\(\S(\infty)\)\\
			\(> - \infty\)&\(< \infty\)&\(= 0\)&\(= \infty\)
		\end{tabular}		
	\end{center}
	Moreover, as discussed above, $\infty$ is inaccessible and $0$ is regular for both diffusions.
	The claims, therefore, follow from Theorem~\ref{theo: main1}, Proposition~\ref{prop: AC Sing}
	and Lemmata \ref{lem: sep modified Bessel} and~\ref{lem:200223a2}.
\end{proof}
	
	For \(\m (\{0\}) \not = \M (\{0\})\) it is intuitive that equivalence might get lost right after the origin is hit, because the diffusions have different boundary behavior. However, Corollary~\ref{coro: example equ lost after hit} shows that even in the case \(\m (\{0\}) = \M(\{0\}) < \infty\) the equivalence of \(\P_{x_0}\) and \(\Q_{x_0}\) is lost right after the origin is hit.
	On the other hand, stopping the processes in the origin transfers us to the case \(\m (\{0\}) = \M(\{0\}) = \infty\), where hitting the origin does not break the equivalence any longer.
\end{example}

In our final example we discuss the importance of taking the (deterministic) time \(R\) into consideration.
We thank Paul Jenkins for bringing the issue and the example to our attention.

\begin{example} \label{ex: BM with reflec bd}
	Suppose that \(([0, 1] \ni x \mapsto \Q_x)\) is a Brownian motion
	and that \(([0, 1] \ni x \mapsto \P_x)\) is a Brownian motion with a constant drift \(\mu \ne 0\), both with instantaneous reflection in their boundaries \(0\) and~\(1\).
	To be more precise, the corresponding scale functions on $[0,1]$ are given by
	\[
	\S (x) = x, \qquad \s(x) = -\frac{e^{-2\mu x}}{2\mu}
	\]
	and the speed measures on $\mathcal B([0,1])$ are given by
	\[
	\M (dx) = dx,
	\qquad
	\m (dx) = e^{2 \mu x} dx.
	\]
	We take an arbitrary $x_0\in[0,1]$ and discuss the separating time $S$ for $\P_{x_0}$ and $\Q_{x_0}$.
	It is not hard to see that all points in \([0, 1]\) are non-separating, which means that $U=V=\delta$.
	However, the separating time for \(\P_{x_0}\) and \(\Q_{x_0}\) cannot be~\(\delta\). The reason for this stems from the fact that a diffusion whose boundaries are both accessible and reflecting is necessarily recurrent (see Lemma~\ref{lem:200223a1}).
	More precisely, using this fact and the ratio ergodic theorem restated as Lemma~\ref{lem: ratio ergodic thm}, we obtain that, for all bounded Borel functions \(f \colon [0, 1] \to \mathbb{R}_+\), it holds \(\P_{x_0}\)-a.s. 
	\[
	\frac{1}{t} \int_0^t f (\X_s) ds \to \frac{1}{\m ([0, 1])} \int_0^1 f (x) \m (dx), \quad t \to \infty,
	\]
	and \(\Q_{x_0}\)-a.s.
	\[
	\frac{1}{t} \int_0^t f (\X_s) ds \to  \int_0^1 f (x) \M (dx), \quad t \to \infty.
	\]
	Hence, in the case \(\P_{x_0} (S = \delta) \vee \Q_{x_0} (S = \delta) > 0\), we get
	\[
	\frac{1}{\m ([0, 1])} \int_0^1 f (x) \m (dx) =  \int_0^1 f(x) \M(dx)
	\]	
	for any countable collection of test functions \(f\). By the uniqueness theorem for probability measures, this implies that \(\m/ \m ([0, 1]) = \M \).
	The latter is, however, false, as \(\mu \ne 0\). Therefore, we must have \(\P_{x_0}, \Q_{x_0}\)-a.s. \(S \leq \infty\).
	And, indeed, as $R=\infty$, Theorem~\ref{theo: main1} yields \(\P_{x_0}, \Q_{x_0}\)-a.s.
	\(S = U\wedge V\wedge R = \infty\).
\end{example}

\begin{discussion}\label{disc:180223a1}
	More generally,
	in the case where both diffusions $(x\mapsto\P_x)$ and $(x\mapsto\Q_x)$ are recurrent,
	an argument based on the ratio ergodic theorem similar to those from Example \ref{ex: BM with reflec bd} shows that
	$\P_{x_0},\Q_{x_0}$-a.s.\ $S\le\infty$.
	The question arises how this fact is encoded in Theorem~\ref{theo: main1} for the cases when $R=\delta$.
	If there is a separating point in $J^\circ$, then, by recurrence, $\P_{x_0},\Q_{x_0}$-a.s.\ $U\wedge V\le\infty$.
	Now, assume that all points in $J^\circ$ are non-separating, that both diffusions are recurrent,
	and that at least one boundary point (say,~$l$) is not reflecting for at least one of the diffusions (say, for $(x\mapsto\P_x)$).
	By Lemma~\ref{lem:200223a1}, in this case we have $\s(l)=-\infty$.
	As for the other boundary point, using that  $(x\mapsto\P_x)$ and $(x\mapsto\Q_x)$ are recurrent, we deduce from Lemma~\ref{lem:200223a1} that
	\begin{enumerate}
		\item[-]
		either $\s(r)=\infty$ or $r$ is reflecting for $(x\mapsto\P_x)$ and
		
		\item[-]
		either $\S(r)=\infty$ or $r$ is reflecting for $(x\mapsto\Q_x)$.
	\end{enumerate}
	Then, it follows that
	\begin{enumerate}
		\item[-]
		$l$ is a separating boundary  point (as $\s(l)=-\infty$).
		
		\item[-]
		$\P_{x_0}$-a.s.\ $U\le\infty$ (by Lemma~\ref{lem:200223a2} in the case $\s(r)=\infty$; by Lemma~\ref{lem: finiteness hitting times}~(iii) in the case $r$ is reflecting for $(x\mapsto\P_x)$).
		
		\item[-]
		$\Q_{x_0}$-a.s.\ $U\le\infty$ (by the same argument as for the other diffusion).
	\end{enumerate}
	In summary, we see that $\P_{x_0},\Q_{x_0}$-a.s.\ $U\wedge V\le\infty$
	in all cases when both $(x\mapsto\P_x)$ and $(x\mapsto\Q_x)$ are recurrent
	except the only case, when all points in $\on{cl}(J)$ are non-separating and both boundaries are reflecting for one, equivalently for both, of the diffusions.
	To account for this case, we need the deterministic time $R$.
\end{discussion}

\section{Martingale Problem for Diffusions}\label{app: A}
The martingale problem method is one of the key techniques to analyze Markov processes. The first martingale problem was introduced by Stroock and Varadhan for multidimensional diffusions which can be described via SDEs (see \cite[Section~5.4]{KaraShre} for an overview). It seems to us that the literature contains no martingale problem for general one-dimensional regular diffusions, although all required tools can be found in the monographs \cite{freedman,liggett}.
In this paper, we need such a martingale problem for the proof of our main Theorem \ref{theo: main1} as well as Theorem~\ref{theo: NFLVR finite tiome horizon}. In general, we also think that such a martingale problem is of independent interest. 
\\

Let \((J \ni x \mapsto \P_x)\) be a regular diffusion and set
\[
(G_\alpha f ) (x) \triangleq \int_0^\infty e^{- \alpha t} \E_x \big[ f (\X_t) \big] dt, \qquad f \in C_b (J; \mathbb{R}), \alpha > 0,
\]
where $C_b(J;\mathbb R)$ denotes the space of bounded continuous functions $J\to\mathbb R$.
Notice that $G_\alpha$ is a bounded linear operator from $C_b(J; \mathbb{R})$ into $C_b(J; \mathbb{R})$ (see \cite[Lemma~30, p.~116]{freedman}).
Let us collect some useful facts (see \cite[Corollaries 43,~44, p.~119]{freedman}):

\begin{lemma} \quad
	\begin{enumerate}
		\item[\textup{(i)}]
		\(G_\alpha f \equiv 0\) if and only if \(f \equiv 0.\)
		\item[\textup{(ii)}]
		The range of \(G_\alpha\) does not depend on \(\alpha\).
	\end{enumerate}
\end{lemma}

We define 
\[
\Delta \triangleq \on{rng} (G_1) \big( = \on{rng} (G_\alpha) \big), \qquad \Gamma \triangleq \on{Id}\ -\ G_1^{-1} \text{ on } \Delta.
\]
Notice that $\Delta\subset C_b(J; \mathbb{R})$.
The operator \((\Gamma, \Delta)\) is called the \emph{(extended) generator of the diffusion \((x \mapsto \P_x)\)}.
Here, we use the term ``extended'' to emphasize that we work on the space $C_b(J; \mathbb{R})$, which is less conventional than the space $C_0(J; \mathbb{R})$ of continuous functions vanishing at infinity.\footnote{The conventional approach is to work with \emph{Feller processes}, which correspond to \emph{Feller semigroups}, i.e., strongly continuous contraction semigroups on $C_0(J; \mathbb{R})$.
This, however, excludes diffusions with entrance boundaries, as $C_0(J; \mathbb{R})$ is not invariant under the respective semigroup.
To include all diffusions, in particular, those with entrance boundaries, we work with $C_b(J; \mathbb{R})$.
In this way, we gain generality but face the problem that, on the space $C_b(J; \mathbb{R})$, many diffusion semigroups (e.g., the Brownian semigroup on $C_b(\mathbb{R}; \mathbb{R})$) are \emph{not} strongly continuous.
Therefore, standard results for Feller processes should be applied with care.
In particular, this is the reason for presenting a full proof of Theorem~\ref{theo: martingale problem}.}
\\

The following result is given in \cite[Lemma~46, p.~119]{freedman}.

\begin{lemma} \label{lem: app2}
	\(G^{-1}_\alpha = \alpha \on{Id}\ -\ \Gamma\) on \(\Delta\).
\end{lemma}

The next result can be proved (with minor changes) as \cite[Theorem~3.33]{liggett}. For reader's convenience we present a proof, which is a slight modification of the one from \cite{liggett}. The theorem provides the martingale problem for \((x \mapsto \P_x)\).

\begin{theorem} \label{theo: martingale problem}
	If \(\P\) is a probability measure on \((\Omega, \mathcal{F})\) such that
	$\P(C(\mathbb R_+;J))=1$,
	\(\P (\X_0 = x) = 1\) and 
	\[
	M^f \triangleq f (\X) - f(x) - \int_0^\cdot (\Gamma f) (\X_s) ds
	\]
	is a local\,\footnote{Hence, necessarily, a \(\P\)-martingale, as the process $M^f$ is bounded on finite time intervals due to $\Delta\subset C_b(J; \mathbb{R})$.} \(\P\)-martingale for all \(f \in \Delta\), then \(\P = \P_x\).
\end{theorem}

To understand the condition $\P(C(\mathbb R_+;J))=1$, recall that $\Omega=C(\mathbb R_+;[-\infty,\infty])$ (see Section~\ref{sec: diff basics}), while $f\in\Delta\subset C_b(J;\mathbb R)$ is defined only on $J$.
That is, the assumption $\P(C(\mathbb R_+;J))=1$ guarantees that $M^f$ is well-defined under $\P$. 

\begin{proof}
	Take \(g \in C_b(J; \mathbb{R})\) and \(\alpha > 0\). Set \(f \triangleq G_\alpha g\) and note that \(\alpha f - \Gamma f = G^{-1}_\alpha f = g\) by Lemma~\ref{lem: app2}.
	Further, we emphasize that
	 \(f (\X)\) is a \(\P\)-semimartingale by hypothesis (because \(f \in \Delta\)).
	Using integration by parts yields that (under \(\P\))
\begin{align*}
d e^{- \alpha t} f (\X_t) &= f (\X_t) d e^{- \alpha t} + e^{- \alpha t} d f (\X_t)
\\&= - \alpha e^{- \alpha t} f (\X_t) dt + e^{- \alpha t} d M^f_t + e^{- \alpha t} (\Gamma f)(\X_t) dt
\\&= -e^{- \alpha t} ( \alpha f (\X_t) - (\Gamma f)(\X_t) ) dt + e^{- \alpha t} d M^f_t
\\&= - e^{- \alpha t} g(\X_t) dt + e^{- \alpha t} d M^f_t.
\end{align*}
Consequently, the process 
\[
e^{- \alpha t} f(\X_t) + \int_0^t e^{- \alpha s} g(\X_s) ds, \quad t \in \mathbb{R}_+, 
\]
is a \(\P\)-martingale. Fix \(s \in \mathbb{R}_+\). Then, for every \(t > s\), we have \(\P\)-a.s.
\[
\E^\P \Big[ e^{- \alpha t} f(\X_t) + \int_s^t e^{- \alpha r} g(\X_r) dr \big| \mathcal{F}_s \Big] = e^{- \alpha s} f (\X_s).
\]
Letting \(t \to \infty\) and using the dominated convergence theorem yields \(\P\)-a.s.
\[
\E^\P \Big[ \int_s^\infty e^{- \alpha r} g (\X_r) dr \big| \mathcal{F}_s \Big] = e^{- \alpha s} f (\X_s).
\]
Thus, by the definition of the conditional expectation, we have for all \(A \in \mathcal{F}_s\)
\[
\E^\P \Big[ \int_0^\infty e^{- \alpha (r + s)} g (\X_{r + s}) dr \ \1_A \Big] = \E^\P\big[e^{- \alpha s} f (\X_s) \1_A\big], 
\]
which means 
\begin{align} \label{eq: used by liggett}
\E^\P \Big[ \int_0^\infty e^{- \alpha r} g (\X_{r + s}) dr \ \1_A \Big] = \E^\P\big[f (\X_s) \1_A\big].
\end{align}
By virtue of Dynkin's formula (\cite[Lemma~48, p.~119]{freedman}), also the measure \(\P_x\) satisfies the hypothesis of the theorem. Thus, the above equation also holds for \(\P\) replaced by \(\P_x\), i.e.,
\begin{align} \label{eq: used by liggett2}
\E^{\P_x} \Big[ \int_0^\infty e^{- \alpha r} g (\X_{r + s}) dr \ \1_A \Big] = \E^{\P_x}\big[f (\X_s) \1_A\big].
\end{align}
 Taking \(s = 0\) and \(A = \Omega\) in \eqref{eq: used by liggett} and \eqref{eq: used by liggett2} yields that 
\[
\int_0^\infty e^{- \alpha t} \E^\P \big[ g (\X_t) \big] dt = f(x) =  \int_0^\infty e^{- \alpha t} \E^{\P_x} \big[ g (\X_t) \big] dt.
\]
Hence, the uniqueness theorem for Laplace transforms (the maps \[t \mapsto \E^\P\big[ g(\X_t) \big], \E^{\P_x} \big[g (\X_t)\big]\] are continuous by the continuity of \(t \mapsto \X_t\) and the dominated convergence theorem) yields that \(\E^\P [g (\X_t)] = \E^{\P_x}[ g (\X_t)]\) for all \(t \in \mathbb{R}_+\). As \(g \in C_b(J; \mathbb{R})\) was arbitrary, this implies that \(\P\) and \(\P_x\) have the same one-dimensional distributions. This result can be extended to the finite-dimensional distributions by induction. Assume that the \(n\)-dimensional distributions coincide and take \(0 \leq t_1 \leq \dots \leq t_n < \infty\). Further, let \(A = \{\X_{t_1} \in F_1, \dots, \X_{t_n} \in F_n\}\) for \(F_1, \dots, F_n \in \mathcal{B}(J)\). Applying \eqref{eq: used by liggett} and \eqref{eq: used by liggett2} with \(s = t_n\) and using the induction hypothesis (i.e., that \(\E^\P[f(\X_{t_n}) \1_A] = \E^{\P_x}[f(\X_{t_n})\1_A]\)) yields 
\[
\E^\P \Big[ \int_0^\infty e^{- \alpha r} g (\X_{r + t_n}) dr \ \1_A \Big] = \E^{\P_x} \Big[ \int_0^\infty e^{- \alpha r} g (\X_{r + t_n}) dr \ \1_A \Big].
\]
Using again the uniqueness theorem for Laplace transforms, we get that the distribution of \((\X_{t_1}, \dots,\) \(\X_{t_n}, \X_{t_n + t})\) coincide under \(\P\) and \(\P_x\) for all \(t \in \mathbb{R}_+\).
In summary, \(\P\) and \(\P_x\) have the same finite-dimensional distributions and the usual monotone class argument yields that \(\P = \P_x\). 
The proof is complete.
\end{proof}

The generator \((\Gamma, \Delta)\) is known explicitly (see \cite{freedman} or \cite{itokean74}). For reader's convenience and for later reference we collect the most important cases for diffusions on natural scale.
To this end, we first prepare a couple of notations.
Consider an open interval $I\subset\mathbb R$ and recall (\cite[Proposition 5.1]{CinJPrSha}) that the following are equivalent:
\begin{enumerate}[label=\textup{(\alph*)}]
\item[\textup{(a)}]
$f\colon I\to\mathbb R$ is a difference of two convex functions;

\item[\textup{(b)}]
$f\colon I\to\mathbb R$ is a continuous function such that its right-hand derivative $f'_+$ ($=d^+f/dx$) exists everywhere on $I$ and $f'_+$ is a right-continuous function with locally finite variation on $I$.
\end{enumerate}
(There is, of course, a symmetric equivalent condition that involves the left-hand derivative.)
We stress that, if $I$ is \emph{not} open, then the equivalence between (a) and~(b) breaks: for instance, on $[0,\infty)$ the root $\sqrt{\cdot}$ satisfies~(a) but not~(b).
In fact, for the interval $[0,\infty)$, we easily conclude from above that the following are equivalent:
\begin{enumerate}
\item[\textup{(a$'$)}]
$f\colon[0,\infty)\to\mathbb R$ has a representation $f=h_1-h_2$ such that $h_i\colon[0,\infty)\to\mathbb R$ are convex functions with $|(h_i)'_+(0)|<\infty$, $i=1,2$;

\item[\textup{(a$''$)}]
$f\colon[0,\infty)\to\mathbb R$ has a representation $f=h_1-h_2$ such that $h_i\colon[0,\infty)\to\mathbb R$ are convex functions that can be extended beyond zero to convex functions $\mathbb{R} \to\mathbb R$;

\item[\textup{(b$'$)}]
$f\colon[0,\infty)\to\mathbb R$ is a continuous function such that $f'_+$ exists everywhere on $[0,\infty)$ and is a right-continuous function with locally finite variation on $[0,\infty)$.
\end{enumerate}
Further, given an open interval $I\subset\mathbb R$,
a function $f\colon I\to\mathbb R$ satisfying~(a) (equivalently,~(b)) and a locally finite measure $\mu$ on $(I,\mathcal B(I))$ the notation
``$df'_+=d\mu$
on $I$''
is a shorthand for
\begin{equation}\label{eq:180322a1}
f'_+(b)-f'_+(a)=\mu((a,b])\quad\text{for all }a<b\text{ in }I.
\end{equation}

\begin{lemma} \label{lem: generator}
Assume that \((J \ni x \mapsto \P_x)\) has characteristics \((\on{Id}, \m)\).
\begin{enumerate}[label=\textup{(\roman*)}]
\item[\textup{(i)}]
Suppose that \(J = \mathbb{R}\). Let \(\Delta^1\) be the set of all \(f \in C_b(\mathbb{R}; \mathbb{R})\) such that $f$ is a difference of two convex functions and \(df'_+ = 2g d \m\) on $\mathbb R$
for some \(g \in C_b(\mathbb{R}; \mathbb{R})\). Further, set \(\Gamma^1 f = g\). Then, \((\Gamma^1, \Delta^1) = (\Gamma, \Delta)\).

\item[\textup{(ii)}]
Suppose that \(J = (0, \infty)\). Let \(\Delta^2\) be the set of all \(f \in C_b((0, \infty); \mathbb{R})\) such that $f$ is a difference of two convex functions $(0,\infty)\to\mathbb R$ and \(df'_+ = 2g d \m\)
on $(0,\infty)$
for some \(g \in C_b((0, \infty); \mathbb{R})\). Further, set \(\Gamma^2 f = g\). Then, \((\Gamma^2, \Delta^2) = (\Gamma, \Delta)\).

\item[\textup{(iii)}]
Suppose that \(J = [0, \infty)\) and that \(\m(\{0\}) = \infty\). Let \(\Delta^3\) be the set of all \(f \in C_b([0, \infty); \mathbb{R})\) such that, restricted to \((0, \infty)\), $f$ is a difference of two convex functions $(0,\infty)\to\mathbb R$ and \(df'_+ = 2g d \m\)
on $(0,\infty)$
for some \(g \in C_b([0, \infty); \mathbb{R})\) with \(g(0) = 0\). Further, set \(\Gamma^3 f = g\). Then, \((\Gamma^3, \Delta^3) = (\Gamma, \Delta)\).

\item[\textup{(iv)}]
Suppose that \(J = [0, \infty)\) and that \(\m(\{0\}) < \infty\). Let \(\Delta^4\) be the set of all \(f \in C_b([0, \infty); \mathbb{R})\) such that
$f$ satisfies the equivalent conditions \textup{(a$'$), (a$''$), (b$'$)} above,
\(df'_+ = 2g d \m\) on \((0, \infty)\) and \(f'_+(0) = 2g (0)\m(\{0\})\) for some \(g \in C_b([0, \infty); \mathbb{R})\). Further, set \(\Gamma^4 f = g\). Then, \((\Gamma^4, \Delta^4) = (\Gamma, \Delta)\).
\end{enumerate}
\end{lemma}

Other possible cases (e.g., $J=[0,1]$, $\m(\{0\})<\infty$, $\m(\{1\})<\infty$) are similar to the ones above.
Alternatively, with somewhat more effort, descriptions of such type for general diffusion generators can be read off from \cite[Section~4.7, p.~135]{itokean74}.

\begin{proof}
For~(i) see \cite[Theorem~75, p.~131]{freedman}, for~(ii) see \cite[Theorem~79, p.~133]{freedman}, for~(iii) see \cite[Theorem~81, p.~135]{freedman}.
We now prove~(iv).
\cite[Theorem~89, p.~137]{freedman} yields a description similar to that in~(iv)
with the only difference that the requirement that $f$ satisfies (b$'$) is replaced by a seemingly weaker requirement that
\begin{enumerate}
\item[\textup{(b$''$)}]
$f\colon[0,\infty)\to\mathbb R$ is a continuous function such that
$f'_+$ exists, is finite and right-continuous on $[0,\infty)$,
and $f'_+$ has locally finite variation on the \emph{open} interval $(0,\infty)$.
\end{enumerate}
It remains to establish that, if $f$ satisfies (b$''$) together with other things listed in~(iv), then $f$ satisfies~(b$'$). For $a\in(0,1)$ we have
$$
\int_{(a,1]} |df'_+|
=
\int_{(a,1]} 2|g|\,d\m.
$$
As the boundary $0$ is regular in~(iv), by~\eqref{eq:101022a3}, $\m((0,1])<\infty$.
Letting $a\searrow0$, we obtain
$$
\big(\,\text{variation of }f'_+\text{ on }[0,1]\,\big)
=
\int_{(0,1]} |df'_+|
=
\int_{(0,1]} 2|g|\,d\m
<\infty.
$$
Together with~(b$''$), this yields~(b$'$) and concludes the proof.
\end{proof}

Finally, we also need a local uniqueness property of martingale problems. We define \(\mathcal{F}^o_t \triangleq \sigma (\X_s, s \leq t)\) for \(t \in \mathbb{R}_+\).
The following result is, in fact, a generalization of Theorem~\ref{theo: martingale problem}, and it can be proved similarly to \cite[Lemma~9.1]{criens20}.
We emphasize that Theorem~\ref{theo: martingale problem} is needed for its proof.

\begin{lemma}\label{lem: loc uni}
	Let \(\tau\) be a stopping time for the filtration \((\mathcal{F}^o_t)_{t \geq 0}\) and let \(\P\) be a probability measure on \((\Omega, \mathcal{F})\) such that
	$\P(C(\mathbb R_+;J))=1$,
	\(\P (\X_0 = x) = 1\) and 
	\[
	M^f \triangleq f (\X_{\cdot \wedge \tau}) - f(x) - \int_0^{\cdot \wedge \tau} (\Gamma f) (\X_s) ds
	\]
	is a local\,\footnote{Hence, necessarily, a \(\P\)-martingale, as $M^f$ is bounded on finite time intervals due to $\Delta\subset C_b(J; \mathbb{R})$.} \(\P\)-martingale for all \(f \in \Delta\). Then, \(\P = \P_x\) on \(\mathcal{F}^o_\tau\).
\end{lemma}

\section{Technical Facts on Diffusions and Semimartingales}\label{app: B}
To make our paper as self-contained as possible we collect technical facts about diffusions and semimartingales which are used in our proofs.

\subsection{Facts on Diffusions}
In this section, 
\((J \ni x \mapsto \P_x)\) is a regular diffusion with scale function \(\s\) and speed measure \(\m\).
Recall from~\eqref{eq:020323a1} that \(T_y =  \inf (s \geq 0 \colon \X_s = y)\).
We denote \(l \triangleq \inf J\) and \(r \triangleq \sup J\).

\begin{lemma}\label{lem:121224a1}
In addition to the regular diffusion
$(J \ni x \mapsto \P_x)$
consider a state space $\tilde J$ with $J^\circ=\tilde J^\circ$
and a regular diffusion
$(\tilde J \ni x \mapsto \tilde\P_x)$.
Let $x_0\in J\cap\tilde J$ be such that
either $x_0\in J^\circ$
or $x_0\in\partial J$ is not an absorbing boundary for at least one of the diffusions.
Assume that $\P_{x_0}=\tilde\P_{x_0}$.
Then $J=\tilde J$ and $\P_y=\tilde\P_y$ for all $y\in J$
(that is, the diffusions coincide).
\end{lemma}

It is worth noting that the assumption on $x_0$ cannot be dropped:
if $x_0$ is an absorbing boundary point for both diffusions, then $\P_{x_0}=\tilde\P_{x_0}$, but, clearly, the diffusions can be different.

\begin{proof}
Notice first that if $x_0$ is a boundary point, then it is clearly reflecting for both diffusions.
Now, fix some $y\in J$ and observe that
\begin{equation}\label{eq:121224a1}
\P_{x_0}(T_y<\infty)>0.
\end{equation}
Indeed, if $x_0\in J^\circ$, then \eqref{eq:121224a1} follows from the regularity of the diffusion $(x\mapsto\P_x)$
(recall~\eqref{eq:121224a0});
if $x_0\in\partial J$, then \eqref{eq:121224a1} follows from Lemma~\ref{lem: finiteness hitting times} below.
Take any $A\in\mathcal F$.
By the strong Markov property,
$$
\P_{x_0}(T_y<\infty,\X_{\cdot+T_y}\in A)
=
\E_{x_0}\big[\1_{\{T_y<\infty\}}\P_{x_0}(\X_{\cdot+T_y}\in A|\mathcal F_{T_y})\big]
=
\P_y(A)\,\P_{x_0}(T_y<\infty).
$$
Thus,
\begin{equation}\label{eq:121224a2}
\P_y(A)
=
\frac{\P_{x_0}(T_y<\infty,\X_{\cdot+T_y}\in A)}{\P_{x_0}(T_y<\infty)}.
\end{equation}
As $y\in J$ and $A\in\mathcal F$ are arbitrary,
it follows that the measure $\P_{x_0}$ uniquely determines the whole diffusion $(J\ni x\mapsto\P_x)$.
Similarly, the measure $\tilde\P_{x_0}$ uniquely determines the whole diffusion $(\tilde J\ni x\mapsto\tilde\P_x)$.
Moreover, the determining formula \eqref{eq:121224a2}, its analogue for the other diffusion and the fact that $\P_{x_0}=\tilde\P_{x_0}$ yield that
$\P_y=\tilde\P_y$ for all $y\in J\cap\tilde J$
(in particular, for all $y\in J^\circ$).

It remains only to prove that $J=\tilde J$.
To this end, take an interior point $c\in J^\circ(=\tilde J^\circ)$ and a boundary point $b\in\partial J(=\partial\tilde J)$.
Recall that $b\in J$ (resp., $b\in\tilde J$)
if and only if
$\P_c(T_b<\infty)>0$ (resp., $\tilde\P_c(T_b<\infty)>0$).
As we already established that $\P_c=\tilde\P_c$,
we obtain that $b\in J$ is equivalent to $b\in\tilde J$.
It follows that $J=\tilde J$.
This concludes the proof.
\end{proof}

The following lemma is Blumenthal's zero-one law (\cite[Lemma~4, p.~106]{freedman}).

\begin{lemma} \label{lem: Blumenthal}
For any \(x \in J\) the \(\sigma\)-field \(\mathcal{F}_0\) is \(\P_x\)-trivial, i.e., \(\P_x(A) \in \{0, 1\}\) for all \(A\in \mathcal{F}_0\).
\end{lemma}

The following lemma is a restatement of \cite[Corollary V.46.15]{RW2}. 

\begin{lemma} \label{lem: scale fct}
	The process \(\s (\X_{\cdot \wedge T_l \wedge T_r})\) is a continuous local \(\P_{x}\)-martingale for all \(x \in J\).
\end{lemma}

The next lemma explains how the characteristics are changed via an homeomorphic change of space. It is a restatement of \cite[Exercise VII.3.18]{RY}.

\begin{lemma} \label{lem: diff homo} 
Let \(\phi\) be a	homeomorphism from \(J\) onto an interval \(I\). Then, \((I \ni x \mapsto \P_{\phi^{-1} (x)} \circ \phi (\X)^{-1})\) is also a regular diffusion
with speed measure \(\m \circ \phi^{-1}\)
and scale function \(\pm\,\s \circ \phi^{-1}\) (the sign is ``$+$'' if $\phi$ is increasing and ``$-$'' of \(\phi\) is decreasing).
\end{lemma}

The next lemma is a restatement of \cite[Theorem 1.1]{bruggeman}. It says that diffusions hit points arbitrarily fast with positive probability, which can be viewed as an irreducibility property.

\begin{lemma} \label{lem: diff hit points fast}
	Take \(x, z \in J\) such that \(x \not = z\) and \(\P_x(T_z < \infty) > 0\). Then, \[\P_x (T_z < \varepsilon) > 0\] for all \(\varepsilon > 0\).
\end{lemma}

The following lemma explains that diffusions exit non-absorbing points immediately. It follows, for instance, from \cite[Corollary 5, Fact 6 on p. 107, Lemma 12 on p. 109]{freedman}.

	\begin{lemma} \label{lem: exit immediately}
		For every \(x \in J\) with \(\m (\{x\}) < \infty\)\footnote{In other words, we exclude only the case where $x$ is an absorbing boundary point.},
		\(
		\P_x ( S_x = 0 ) = 1, 
		\)
		where \(S_x \triangleq \inf (t \geq 0 \colon \X_t \not = x)\).
	\end{lemma}

Part~(i) of the following lemma is a restatement of \cite[Corollary~19, p.~112]{freedman}, parts (ii) and~(iii) are restatements of \cite[Lemmata 20,~21, p.~112]{freedman}.

\begin{lemma} \label{lem: finiteness hitting times}
\begin{enumerate}
	\item[\textup{(i)}]
	If $a,x,c\in J$ are such that $a\le x\le c$, then \(\E_{x}[T_a \wedge T_c] < \infty\).
	
	\item[\textup{(ii)}]
	If \(l\) is reflecting and $x,c\in J$ are such that \(x \leq c \), then \(\E_{x}[T_c] < \infty\).
	
	\item[\textup{(iii)}]
	If \(r\) is reflecting and $a,x\in J$ are such that \(a \leq x\), then \(\E_{x}[T_a] < \infty\).
\end{enumerate}
\end{lemma}

Let us also recall some  general  path properties of diffusions.
The following lemma is a consequence of~\eqref{eq:260223a1}
(see, e.g., \cite[Proposition 5.5.22]{KaraShre} for a detailed proof).

\begin{lemma}\label{lem:200223a2}
	Take \(x \in J^\circ\).
	\begin{enumerate}
		\item[\textup{(i)}]
		If \(\s (l) = - \infty, \s (r) = \infty\), then 
		\[
		\P_{x} \big(T_l \wedge T_r = \infty\big) = \P_{x} \Big( \limsup_{t\to\infty} \X_t = r \Big) = \P_{x} \Big(\liminf_{t\to\infty} \X_t = l\Big) = 1.
		\]
		
		\item[\textup{(ii)}]
		If \(\s (l) > - \infty, \s (r) = \infty\), then 
		\[
		\P_{x} \Big( \sup_{t < T_l \wedge T_r} \X_t < r \Big) = \P_{x} \Big(\lim_{t \nearrow T_l \wedge T_r} \X_t = l\Big) = 1.
		\]
		
		\item[\textup{(iii)}]
				If \(\s (l) = - \infty, \s (r) < \infty\), then 
		\[
		\P_{x} \Big( \lim_{t \nearrow T_l \wedge T_r} \X_t = r \Big) = \P_{x} \Big(\inf_{t < T_l \wedge T_r} \X_t > l\Big) = 1.
		\]
		
		\item[\textup{(iv)}]
		If \(\s (l) > - \infty, \s (r) < \infty\), then 
		\[
		\P_{x} \Big(\lim_{t \nearrow T_l \wedge T_r} \X_t = l\Big) = 1 - \P_{x} \Big(\lim_{t \nearrow T_l \wedge T_r} \X_t= r\Big) = \frac{\s(r) - \s(x)}{\s(r) - \s(l)}.
		\]
	\end{enumerate}
\end{lemma}

The next result implies that there is a dichotomy:
in the case when with positive probability the trajectories tend to a boundary point $b\in\{l,r\}$, as $t\nearrow T_b$,\footnote{From Lemma~\ref{lem:200223a2} we know that this is exactly the case $|\s(b)|<\infty$.}
either they do this in infinite time only
(when $r$ is inaccessible)
or they do this in finite time \emph{only}
(when $r$ is accessible).
It is the latter ``only'' that needs to be formally stated.
We do this in the next lemma, which follows from \cite[Theorem 33.15]{kallenberg}
(in fact, \cite[Theorem 33.15]{kallenberg} contains much more information).

\begin{lemma}\label{lem:020323a2}
Let $b\in\{l,r\}$ be an accessible boundary and $x\in J^\circ$. Then
$$
\P_x\Big(T_b=\infty,\lim_{t\nearrow T_b}\X_t=b\Big)=0.
$$
\end{lemma}

We recall that a diffusion $(J\ni x\mapsto\P_x)$ is called \emph{recurrent} if $\P_x(T_y<\infty)=1$ for all $x,y\in J$.
The following characterization of recurrence is a consequence of
Lemmata \ref{lem: finiteness hitting times} and~\ref{lem:200223a2}.

\begin{lemma}\label{lem:200223a1}
Diffusion $(J\ni x\mapsto\P_x)$ is recurrent if and only if each boundary point $b\in\{l,r\}$ satisfies one of the following:
\begin{enumerate}
\item[\textup{(i)}]
$|\s(b)|=\infty$,

\item[\textup{(ii)}]
$b$ is reflecting.
\end{enumerate}
\end{lemma}

The following version of the {\em ratio ergodic theorem} follows from \cite[Theorem 33.14]{kallenberg} and Lemma~\ref{lem: diff homo}.

\begin{lemma} \label{lem: ratio ergodic thm}
Let \((J\ni x \mapsto \P_x)\) be recurrent.
Then, for any Borel functions \(f, g \colon J \to \mathbb{R}_+\) with \(\int f\, d \m < \infty\) and \(\int g \,d \m > 0\), we have 
	\[
	\lim_{t \to \infty} \frac{\int_0^t f (\X_s) \,ds}{\int_0^t g (\X_s) \,ds} = \frac{\int f\, d \m}{\int g \,d \m} \ \ \P_x\text{-a.s. for all } x \in J.
	\]
\end{lemma}

The law of a reflecting Brownian motion coincides with the law of the absolute value of a standard Brownian motion. More generally, for any diffusion with
an (instantaneously or slowly) reflecting boundary point
one can find a diffusion on an extended state space with inaccessible or absorbing boundaries and a Lipschitz function \(\f\) such that in law the original one is obtained via space transformation by \(\f\) from the extended one.
In the following we explain this fact for the cases \(J = [0, \infty)\) (Lemma~\ref{lem: refl})
and $J=[0,1]$ where both $0$ and $1$ are reflecting boundaries (Lemma~\ref{lem:200223a3}).
Lemmata \ref{lem: refl} and~\ref{lem:200223a3} follow from the proof of
\cite[Proposition VII.3.10]{RY}.
We also refer to \cite[Section~6]{ankirchner} for a similar discussion for diffusions on natural scale.

\begin{lemma} \label{lem: refl}
	Suppose that \(J = \mathbb{R}_+\), \(\s (0) = 0\) and \(\m (\{0\}) < \infty\). Define scale function \(\s^\leftrightarrow\colon \mathbb{R} \to \mathbb{R}\) by 
	\[
	\s^\leftrightarrow (x) \triangleq \begin{cases}\s(x),&x \geq 0,\\- \s (- x),& x < 0,\end{cases}
	\]
	and speed measure \(\m^\leftrightarrow\) on \((\mathbb{R}, \mathcal{B}(\mathbb{R}))\) by
\begin{equation}\label{eq:101022b1}
	\m^\leftrightarrow(A) = \begin{cases} \m (A),& A \in \mathcal{B}((0,\infty)),\\
	\m(- A), & A \in \mathcal{B}((- \infty, 0)) \quad (- A \triangleq \{x \in (0, \infty) \colon - x \in A\}),\\
	2 \m (\{0\}),& A = \{0\}.
	\end{cases}
\end{equation}
	Let \((\mathbb{R} \ni x \mapsto \P^\leftrightarrow_x)\) be regular diffusion with characteristics \((\s^\leftrightarrow, \m^\leftrightarrow)\). Then, 
	\[
	\P_x = \P^\leftrightarrow_x \circ | \X |^{-1}, \quad x \in \mathbb{R}_+ = J.
	\]
\end{lemma}

We emphasize that in the situation of Lemma~\ref{lem: refl}
formula~\eqref{eq:101022b1} always defines a valid speed measure on $(\mathbb R,\mathcal B(\mathbb R))$ (recall Remark~\ref{rem:170322a1}).
Indeed, only the fact that $\m^\leftrightarrow$ is locally finite on $\mathbb R$ requires an explanation,
and this is a direct consequence
of \eqref{eq:101022a2} or~\eqref{eq:101022a3}
(notice that boundary point $0$ is regular due to the assumption $\m(\{0\})<\infty$;
to this end, recall the convention in the second paragraph before Remark~\ref{rem:170322a1}
that $\m(\{b\})=\infty$ whenever $b$ is an exit boundary).

\begin{lemma}\label{lem:200223a3}
Suppose that \(J = [0,1]\), \(\s (0) = 0\), \(\m (\{0\}) < \infty\) and $\m(\{1\})<\infty$.
Define scale function $\overline\s\colon\mathbb R\to\mathbb R$ by
\begin{itemize}
\item
$\overline\s(x)\triangleq\s(x)$, $x\in[0,1]$,

\item
$\overline\s(x)\triangleq-\s(-x)$, $x\in[-1,0]$,

\item
$\overline\s(x+2k)\triangleq\overline\s(x)+2k\overline\s(1)$, $k\in\mathbb Z\setminus\{0\}$, $x\in[-1,1]$,
\end{itemize}
and speed measure $\overline\m$ on $(\mathbb R,\mathcal B(\mathbb R))$ by
\begin{itemize}
\item
$\overline\m(A)=\overline\m(-A)\triangleq\m(A)$, $A\in\mathcal B((0,1))$,

\item
$\overline\m(A+2k)\triangleq\overline\m(A)$, $k\in\mathbb Z\setminus\{0\}$, $A\in\mathcal B((-1,0)\cup(0,1))$,

\item
$\overline\m(\{2k\})=2\m(\{0\})$, $k\in\mathbb Z$,

\item
$\overline\m(\{2k+1\})=2\m(\{1\})$, $k\in\mathbb Z$.
\end{itemize}
Let \((\mathbb{R} \ni x \mapsto \QQ_x)\) be regular diffusion with characteristics \((\overline\s,\overline\m)\). Then,
\[
\P_x = \QQ_x \circ\f(\X)^{-1}, \quad x \in [0,1] = J,
\]
where $\f\colon\mathbb R\to[0,1]$ is the periodic function with period $2$ satisfying $\f(x)=|x|$, $x\in[-1,1]$.
\end{lemma}

It is necessary to remark that in the situation of Lemma~\ref{lem:200223a3} both $\infty$ and $-\infty$ are inaccessible boundaries for $(x\mapsto\QQ_x)$
(that is, the state space of that diffusion is indeed $\mathbb R$), as $|\overline\s(\pm\infty)|=\infty$.
Furthermore, as both $0$ and $1$ are regular boundaries for $(x\mapsto\P_x)$,
the measure $\overline\m$ is locally finite on $\mathbb R$ due to \eqref{eq:101022a2} or~\eqref{eq:101022a3}, which is necessary for $\overline\m$ to be a valid speed measure
(cf.\ with the discussion after Lemma~\ref{lem: refl}).

\smallskip
The next lemma explains the general structure of diffusions on natural scale.
It is a restatement of \cite[Theorem 33.9]{kallenberg}.

\begin{lemma} \label{lem: diff time changed BM}
	Suppose that \(\s = \on{Id}\) and take \(x \in J\).
	\begin{enumerate}
		\item[\textup{(i)}]
		Possibly on an extension of the underlying filtered probability space, there exists a Brownian motion \(W\) starting in \(x\) such that \(\P_{x}\)-a.s.
		\(
		\X = W_{\gamma}, 
		\)
		where
		\[
		\gamma_t \triangleq \inf (s \geq 0 \colon A_s > t), \quad A_t \triangleq \int_J L^y_t (W) \m (dy), \quad t \in \mathbb{R}_+, 
		\]
		and \(\{L^y_t (W) \colon (t,y) \in \mathbb{R}_+ \times \mathbb{R}\}\) is the local time process of \(W\).
		\item[\textup{(ii)}]
		The law of any time-changed Brownian motion \(W_{\gamma}\) as in \textup{(i)} above coincides with \(\P_{x}\).
	\end{enumerate}
\end{lemma}

Notice that,
as a.s.\ $L_\infty^y(W)=\infty$ for all $y\in\mathbb R$ and, hence, $A_\infty=\infty$,
it follows that the time-change $\gamma$ in Lemma~\ref{lem: diff time changed BM} is always finite.
Further notice that $A$ can have intervals of constancy, i.e., the time-change $\gamma$ might have jumps.
For instance, think about $J=[0,\infty)$ with $\m(\{0\})<\infty$ (in particular, $0$ is a regular boundary, i.e., also $\m((0,1))<\infty$, see~\eqref{eq:101022a3}).
Then, $A$ is constant (and finite) during the negative excursions of $W$, i.e., $\gamma$ overjumps the negative excursions of~$W$.

\smallskip
By virtue of Lemma~\ref{lem: diff homo} any diffusion can be brought to natural scale (i.e., \(\s = \on{Id})\) by a homeomorphic space transformation via its scale function. Thus,
Lemma~\ref{lem: diff time changed BM} shows that any regular diffusion is a space and time-changed Brownian motion. In fact, any diffusion on natural scale is a semimartingale.\footnote{We stress that the semimartingale property might be lost via the space transformation by the scale function. For instance, the square root of reflecting Brownian motion (which is a semimartingale and a diffusion on natural scale) is not a semimartingale. This example is known under the name \emph{Yor's example}.}
The next lemma formalizes this fact, providing the \emph{occupation time formula for diffusions}. In particular, it emphasizes the difference between the \emph{semimartingale local time} and the \emph{diffusion local time}.\footnote{In this connection, we thank Zhesheng Liu and Mihail Zervos for bringing formula~\eqref{eq:241124a1} to our attention.}

\begin{lemma}\label{lem: occ formula diff}
Suppose that \(\s = \on{Id}\) and take \(x \in J\).

\textup{(i)}
The process \(\X\) is a continuous \(\P_x\)-semimartingale.
In particular, it has continuous in time and c\`adl\`ag in space
\emph{semimartingale local time} process
$L(\X)=\{L_t^y(\X):(t,y)\in\mathbb R_+\times\mathbb R\}$
(see Lemma~\ref{lem: occ smg} below for a recap).

\textup{(ii)}
For the \emph{diffusion local time process}
$\widehat L(\X)=\{\widehat L_t^y(\X):(t,y)\in\mathbb R_+\times J\}$
defined as
\begin{equation}\label{eq:241124a1}
\widehat L_t^y(\X)\triangleq
\begin{cases}
L_t^y(\X),&(t,y)\in\mathbb R_+\times(J\setminus\{r\}),
\\
L_t^{y-}(\X),&(t,y)\in\mathbb R_+\times\{r\}\text{ (in case }r\in J)
\end{cases}
\end{equation}
(recall that $r=\sup J$),
$\P_x$-a.s. for all $(t,y)\in\mathbb R_+\times J$,
\begin{equation}\label{eq:260223a4}
\widehat L^y_t(\X)=L^y_{\gamma_t} (W),
\end{equation}
where we use the notation from Lemma~\ref{lem: diff time changed BM}.
Moreover, \(\P_x\)-a.s.\
the diffusion local time $\widehat L(\X)$ is
jointly continuous in time and space
on \(\mathbb{R}_+ \times J\),
and \(\P_x\)-a.s.\ we have
	\begin{equation}\label{eq:020423a5}
	\int_0^t f (\X_s) ds = \int_J f (y) \widehat L^y_t (\X) \m (dy)
	\end{equation}
simultaneously for all \(t \in \mathbb{R}_+\) and all Borel functions
\(f \colon J \to [0,\infty]\)
with \(f (b) = 0\) for all boundary points \(b \in J \setminus J^\circ\) with \(\m (\{b\}) = \infty\).
\end{lemma}

\begin{remark}\label{rem:241124a1}
(a)
We emphasize that the semimartingale local time is defined for semimartingales and for $(t,y)\in\mathbb R_+\times\mathbb R$.
On the other hand, the diffusion local time is defined for diffusions on natural scale and for $(t,y)\in\mathbb R_+\times J$.
Comparing them for levels $y\in J$, we see that, by~\eqref{eq:241124a1}, $L(\X)$ and $\widehat L(\X)$ can only differ at the level $y=r$.
To see the difference, the right boundary point \(r\) needs to be in \(J\) and therefore, as $\X$ is on natural scale, \(r\) is necessarily finite.
In that case, as $\X$ is $J$-valued and c\`adl\`ag in the space variable, we always have
$L_\cdot^r(\X)=0$, but $L_\cdot^{r-}(\X) > 0$ is also possible.
Thus, the only correction at the level $y=r$ performed in~\eqref{eq:241124a1} is needed to obtain the joint continuity of $\widehat L(\X)$ in time and space
(but of course on $\mathbb R_+\times J$ only; not on $\mathbb R_+\times\mathbb R$!).
To ease several references, we also note the following consequence of the joint continuity of $\widehat L(\X)$ on $\mathbb R_+\times J$:
$\P_x$-a.s. the semimartingale local time $L(\X)$ is jointly continuous in time and space on $\mathbb R_+\times J^\circ$.

(b)
Finally, we stress that, in the above generality, it is essential to have the diffusion local time $\widehat L(\X)$ in~\eqref{eq:020423a5} (with $L(\X)$ in place of $\widehat L(\X)$ the claim would be incorrect).
On the other hand,
\eqref{eq:020423a5} obviously implies that \(\P_x\)-a.s.\ we have
	\begin{equation}\label{eq:241124a2}
	\int_0^t f (\X_s) ds = \int_J f (y) L^y_t (\X) \m (dy)
	\end{equation}
simultaneously for all \(t \in \mathbb{R}_+\) and all Borel functions
\(f \colon J \to [0,\infty]\)
with \(f (b) = 0\) for all boundary points \(b \in J \setminus J^\circ\)
(cf. \cite[Theorem V.49.1]{RW2}).
\end{remark}

\begin{proof}[Proof of Lemma~\ref{lem: occ formula diff}]
	The semimartingale property follows from Lemma \ref{lem: diff time changed BM} and Lemma~\ref{lem: change of time} below.\footnote{We stress that,
	contrary to the semimartingale property,
	the (local) martingale property is in general not preserved by the change of time.
	Indeed, a reflected Brownian motion is a time-changed Brownian motion by Lemma \ref{lem: diff time changed BM} but it is not a (local) martingale.}
The representation~\eqref{eq:260223a4} of the diffusion local time and the occupation time formula~\eqref{eq:020423a5}
are restatements of \cite[Theorem~V.49.1]{RW2} modulo an obvious monotone class argument to extend the formula to non-negative unbounded functions and the extension to functions which vanish only at absorbing boundary points. The latter generalization follows from an inspection of the proof in \cite[Theorem V.49.1]{RW2}.
It remains to prove the joint continuity of $\widehat L(\X)$ in time and space on $\mathbb R_+\times J$.
To this end, we use~\eqref{eq:260223a4}.
First, we notice that jumps of the time-change $\gamma$ correspond to excursions of $W$ outside $J$, while the local time $L^z(W)$ at all levels $z\in J$ stays constant in time during these excursions.
Therefore, \(\widehat L(\X)\equiv L_{\gamma}(W)\) is \(\P_x\)-a.s. continuous in time. 
Furthermore, by virtue of \cite[Corollary VI.1.8]{RY}, \(L (W)\) is \(\P_x\)-a.s.\ H\"older continuous of any order \(\alpha \in (0, 1/2)\) in space uniformly on compact time intervals.
Hence, the same holds for $\widehat L(\X)$.
Together with the continuity in time, this implies that \(\widehat L(\X)\) is \(\P_x\)-a.s.\ jointly continuous in time and space on $\mathbb R_+\times J$.
	To be more precise, let \((t_n, x_n)_{n = 0}^\infty \subset \mathbb{R}_+ \times J\) be a sequence such that \((t_n, x_n) \to (t_0, x_0)\in \mathbb{R}_+ \times J\), $n\to\infty$, take some \(\alpha \in (0, 1/2)\) and let \(C > 0\) be such that 
	\[
	|\widehat L^{x}_{t_n} (\X) - \widehat L^y_{t_n} (\X)| \leq C |x - y|^\alpha
	\]
for all \(n \in \mathbb{Z}_+\) and all \(x,y \in J\)
(here the constant $C$ depends on $\omega$).
Then,
	\begin{align*}
	|\widehat L^{x_n}_{t_n} (\X) - \widehat L^{x_0}_{t_0} (\X)| &\leq |\widehat L^{x_n}_{t_n} (\X) - \widehat L^{x_0}_{t_n} (\X)| + |\widehat L^{x_0}_{t_n} (\X) - \widehat L^{x_0}_{t_0} (\X)| 
	\\&\leq C |x_n - x_0|^\alpha + |\widehat L^{x_0}_{t_n} (\X) - \widehat L^{x_0}_{t_0} (\X)|  \to 0
	\end{align*}
as \(n \to \infty\).
This completes the proof.
\end{proof}

	The next lemma
	is a restatement of \cite[V.47.23(ii)]{RW2}.
	\begin{lemma} \label{lem: compensation reflecting}
		Take \(J = [l, r)\), \(\s = \on{Id}\) and assume that \(l\) is a reflecting boundary point. Fix \(x \in J\) and let \(\{ \widehat{L}^y_t (\X) \colon (t, y) \in \mathbb{R}_+ \times J\}\) be the diffusion local time of \(\X\) under \(\P_{x}\) (see Lemma~\ref{lem: occ formula diff}).
		Then, \(\X - \frac{1}{2} \, \widehat{L}^{l} (\X)\) is a continuous local \(\P_{x}\)-martingale.
	\end{lemma}

\begin{lemma} \label{lem: pos LT}
Suppose that \(\s = \on{Id}\) and take \(x \in J\) with \(\m(\{x\}) < \infty\).
Let \(\{\widehat L^y_t(\X) \colon (t, y) \in \mathbb{R}_+ \times J\}\) be the diffusion local time of \(\X\) under \(\P_x\)
(see Lemma~\ref{lem: occ formula diff}).
Then, \(\P_x\)-a.s. \(\widehat L^x_t (\X) > 0\)
for all \(t > 0\).
\end{lemma}

\begin{proof}
	By Lemma~\ref{lem: occ formula diff}, using also its notation, \(\P_x\)-a.s. \(\widehat L(\X) = L_\gamma (W)\), where we recall that \(W\) is a Brownian motion starting in \(x\).
	Furthermore, as \(x\) is non-absorbing (which is the meaning of the assumption \(\m (\{x\}) < \infty\)), \(\P_x\)-a.s. \(\gamma_t > 0\) for all \(t > 0\), see Lemma~\ref{lem: diff time changed BM}~(i) or \cite[Section~2.9]{freedman}.
Finally, since a.s. \(L^x_t (W) > 0\) for all \(t > 0\) by \cite[Lemma~106, p.~146]{freedman}, the claim follows.
\end{proof}

\begin{lemma} \label{lem:090922a1}
Take \(x \in J \) with \(\m (\{x\}) < \infty\).
	Furthermore, take a Borel function \(f \colon J \to [0,\infty]\) and set
	\[
	A \triangleq \int_0^\cdot f (\X_s) ds.
	\]
	The following are equivalent:
	\begin{enumerate}
		\item[\textup{(i)}]
		There exists an open in $J$ neighborhood $I$ of $x$ such that
		\(\int_{I} f (y) \m (dy) < \infty\).
		
		\item[\textup{(ii)}]
		There exists a random time $\rho$ such that $\P_x$-a.s. $\rho>0$ and \(\P_x(A_\rho < \infty) > 0\).
		
		\item[\textup{(iii)}]
		There exists a stopping time $\tau$ such that $\P_x$-a.s. $\tau>0$ and~\(A_\tau < \infty\).
	\end{enumerate}
\end{lemma}

In relation with Lemma~\ref{lem:090922a1} we emphasize that, in~(i),
the neighborhood $I$ of $x$ should be \emph{open in $J$}.
In particular, if $x$ is left (resp., right) boundary point of $J$, then we search for $I$ of the form $[x,c)$ (resp., $(c,x]$) for some $c\in J^\circ$.

\begin{remark}
In an equivalent and condensed form, the claim in Lemma~\ref{lem:090922a1} can be restated as follows:
\begin{enumerate}
\item[-]
If (i) holds, then (iii) holds.

\item[-]
If (i) does not hold, then $\P_x$-a.s. we have $A_t=\infty$ for all $t>0$.
\end{enumerate}
\end{remark}

\begin{proof}[Proof of Lemma~\ref{lem:090922a1}]
By Lemmata \ref{lem: diff homo} and~\ref{lem: occ formula diff}, $\P_x$-a.s.\ it holds for all $t\in\mathbb R_+$
\begin{equation} \label{eq:090922a1}
\begin{split}
A_t=\int_0^t f(\X_u)\,du
&=
\int_0^t f\circ\s^{-1}(\s(\X_u))\,du
\\[1mm]
&=
\int_{\s(J)} f\circ\s^{-1}(z)
\widehat L_t^z(\s(\X))
\,\m\circ\s^{-1}(dz)
\\[1mm]
&=
\int_{J} f(y)
\widehat L_t^{\s(y)}(\s(\X))
\,\m(dy).
\end{split}
\end{equation}
Assume that (i) holds and let $I$ be an open in $J$ neighborhood of $x$ such that
\(\int_{I} f (y) \m (dy) < \infty\). Via shrinking, we can w.l.o.g.\ assume that \(I\) is relatively compact in \(J\).
Define the stopping time
$$
\tau \triangleq \inf( t\ge0 \colon \X_t\notin I )
$$
and notice that for \(\P_x\)-a.a. \(\omega \in \Omega\),
\(J\ni y\mapsto\widehat L_\tau^{\s(y)}(\s(\X)) (\omega)\)
is continuous with compact support, because it vanishes outside the relatively compact set \(I\).
Now, \eqref{eq:090922a1} implies that (iii) holds with the stopping time~$\tau$.
Statement~(iii), in turn, implies~(ii).

We finally show that (ii) implies~(i) or, equivalently, that the negation of~(i) implies the negation of~(ii).
So we assume that for any open in $J$ neighborhood $I$ of $x$ we have
\(\int_{I} f (y) \m (dy) = \infty\) and take any random time $\rho$ such that $\P_x$-a.s. $\rho>0$.
Thanks to Lemma~\ref{lem: pos LT}, we have \(\P_x\)-a.s.
$\widehat L_\rho^{\s(x)}(\s(\X))>0$.
Further, by continuity (Lemma~\ref{lem: occ formula diff}), for
$\P_x$-a.a. \(\omega \in \Omega\), the function
$J\ni y\mapsto\widehat L_\rho^{\s(y)}(\s(\X)) (\omega)$
is bounded away from zero in a small neighborhood of $x$ in $J$.
Then, \eqref{eq:090922a1} implies that $\P_x$-a.s.\ $A_\rho=\infty$.
This concludes the proof.
\end{proof}

The next lemma can be viewed as an extension of \cite[Theorem 2.11]{MU12c} beyond the class of It\^o diffusions. For its statement, recall that diffusions on natural scale are semimartingales
(Lemma~\ref{lem: occ formula diff}~(i)).
Of course, Lemma~\ref{lem: pepetual integral diffusion} also has an analogue for the left boundary point.

\begin{lemma} \label{lem: pepetual integral diffusion}
Assume that \(\s = \on{Id}\) and $r<\infty$ (recall the notations $r=\sup J$ and $l=\inf J$).
Let \(f \colon J \to \mathbb{R}_+\) be a Borel function such that \(f \in L^1_\textup{loc} (J^\circ)\) and set
\[
\zeta \triangleq \inf (t \geq 0 \colon \X_t \not \in J^\circ)
\quad\text{and}\quad
D \triangleq \Big\{ \lim_{t \nearrow \zeta} \X_t = r\Big\}.
\]
With $\langle\X,\X\rangle$ denoting the quadratic variation process of $\X$, we have:
\begin{align}
\int^{r-} (r - y) f (y) dy < \infty
\;\;&\Longrightarrow\;\;
\forall x \in J^\circ:\;\;
\int_0^{T_r} f (\X_s) d \langle \X, \X\rangle_s < \infty\;\;\P_{x}\text{-a.s.\ on } D,
\label{eq:260223a2}\\
\int^{r-} (r - y) f (y) dy = \infty
\;\;&\Longrightarrow\;\;
\forall x \in J^\circ:\;\;
\int_0^{T_r} f (\X_s) d \langle \X, \X\rangle_s = \infty\;\;\P_{x}\text{-a.s.\ on } D.
\label{eq:260223a3}
\end{align}
Furthermore, if $l$ is reflecting (necessarily, $l>-\infty$) and \(f \in L^1_\textup{loc} ([l,r))\), then \eqref{eq:260223a2} and~\eqref{eq:260223a3} also hold for all $x\in[l,r)$ and with \(D\) replaced by $\Omega$.
\end{lemma}

\begin{proof}
In case $l$ is reflecting and \(f \in L^1_\textup{loc} ([l,r))\) take some $x\in[l,r)$.
Otherwise take some \(x \in J^\circ\).
By Lemma~\ref{lem: diff time changed BM}, we have \(\P_{x}\)-a.s.\ \(\X = W_\gamma\), where, on an extension of the underlying space, \(W\) is a Brownian motion starting in \(x\) and the time-change \(\gamma\) is defined as in Lemma~\ref{lem: diff time changed BM}.
We extend the function $f\colon J\to\mathbb R_+$ to a function $\mathbb R\to\mathbb R_+$ by setting $f(x)\triangleq 0$ for $x\in\mathbb R\setminus J$.
By the semimartingale occupation time formula (see Lemma~\ref{lem: occ smg} below)
together with \eqref{eq:241124a1} and~\eqref{eq:260223a4}
in Lemma~\ref{lem: occ formula diff}, we have a.s.\ for all \(t \in \mathbb{R}_+\)
\begin{equation}\label{eq:260223a5}
\int_0^{t} f (\X_s) d \langle \X, \X\rangle_s = \int_J f(y) L^y_t (\X) dy = \int_{- \infty}^\infty f (y) L_{\gamma_t}^y (W) dy = \int_0^{\gamma_t} f (W_s) ds.
\end{equation}
Here the following comments are in order:
\begin{enumerate}
\item[-]
The third and the fourth expressions in~\eqref{eq:260223a5} require the extended function~$f$.

\item[-]
The seemingly indirect way of proving~\eqref{eq:260223a5} via the occupation time formula is due to the fact that the time-change $\gamma$ can have jumps.
\end{enumerate}
Notice that a.s.\ \(\gamma_{T_r} = T_r (W)\triangleq \inf (t \geq 0 \colon W_t = r)\) on \(\{ \lim_{t \nearrow T_r} \X_t = r \}\),
while this event equals~a.s.
\begin{enumerate}
\item[-]
$\Omega$ if $l$ is reflecting,

\item[-]
$D$ otherwise.
\end{enumerate}
By \cite[Lemma 4.1]{MU12c}, 
for any nonnegative Borel function \(g \in L^1_\textup{loc} ((- \infty, r))\), 
\begin{align*}
\int^{r-} (r - y) g(y) dy < \infty \ \ &\Longrightarrow \ \ \text{ a.s. } \int_0^{T_r (W)} g (W_s) ds < \infty,
\\
\int^{r-} (r - y) g(y) dy = \infty \ \ &\Longrightarrow \ \ \text{ a.s. } \int_0^{T_r (W)} g (W_s) ds = \infty.
\end{align*}
Therefore, the claims follow.
\end{proof}

We also need an analogue of the previous lemma for the recurrent case.

\begin{lemma}\label{lem:040423a1}
Assume that \(\s = \on{Id}\) and that $(x \mapsto \P_x)$ is recurrent.
Let \(f \colon J \to \mathbb{R}_+\) be a non-vanishing Borel function in the sense that
$\int_J f(y)\,dy>0$.
Then
$$
\forall x \in J:\;\;
\int_0^\infty f(\X_s)\,d\langle\X,\X\rangle_s=\infty
\;\;\P_x\text{-a.s.}
$$
\end{lemma}

\begin{proof}
Take some $x\in J$.
Using the notation from the beginning of the proof of Lemma~\ref{lem: pepetual integral diffusion}
and extending the function $f$ to the whole $\mathbb R$ by zero,
we obtain a.s. for $t\in\mathbb R_+$
$$
\int_0^{t} f (\X_s) d \langle \X, \X\rangle_s = \int_{- \infty}^\infty f (y) L_{\gamma_t}^y (W) dy.
$$
Letting $t\to\infty$ and noting that
\begin{enumerate}
\item[-]
a.s. $\gamma_\infty=\infty$ (due to the recurrence of $\X$) and

\item[-]
a.s., for all $y\in\mathbb R$, it holds $L_\infty^y(W)=\infty$,
\end{enumerate}
we deduce the claim.
\end{proof}

We end this subsection with a version of Meyer's theorem on predictability, see \cite[Proposition~4]{chungwalsh} and \cite[Lemma~I.2.17]{JS}.

\begin{lemma} \label{lem: Meyer theo}
For every \(x \in J\), any stopping time coincides \(\P_x\)-a.s.\ with a predictable time.
\end{lemma}

\subsection{Chain Rule for Diffusions} \label{sec: chain rule}
In this subsection we put the chain rule for diffusions from \cite[Section 5.5]{itokean74} to a general scale.

Let \((J \ni x \mapsto \P_x)\) and \((I \ni x \mapsto \Q_x)\) be regular diffusions with characteristics \((\s, \m)\) and \((\S, \M)\), respectively. Suppose the following:
\begin{enumerate}
	\item[(a)] \(I \subset J\).
	\item[(b)] \(\s = \S\) on \(I\).
	\item[(c)] If \(l^* \triangleq \inf I\) is reflecting for \((x \mapsto \Q_x)\), then either \(l \triangleq \inf J\) is reflecting for \((x \mapsto \P_x)\) or [\(\s(l) = - \infty\) and \(l < l^*\)].
	\item[(d)] If \(r^* \triangleq \sup I\) is reflecting for \((x \mapsto \Q_x)\), then either \(r \triangleq \sup J\) is reflecting for \((x \mapsto \P_x)\) or [\(\s(r) = \infty\) and \(r^* < r\)].
\end{enumerate}
We emphasize that the terminology ``reflecting'' in all instances in (c) and~(d) above is understood as ``instantaneously or slowly reflecting''.
Let \(G \subset \{l^*, r^*\}\) be the set of absorbing boundaries of \((x \mapsto \Q_x)\) and set \[\e \triangleq \inf (t \geq 0 \colon \X_t \in G)\] and
\[
\g (t) \triangleq \begin{cases} 
\int_{I} \l (t, x)\,\M(dx),&t < \e,\\ \infty,& t \geq \e,
\end{cases}
\]
where
\[
\l (t, x) \triangleq
\begin{cases}
\limsup_{h \searrow 0}
\frac{\int_0^t \1 \{x - h < \X_s < x + h\} ds}{\m ((x - h, x + h))}
&\text{if }x\in I^\circ,
\\[2mm]
\limsup_{h \searrow 0}
\frac{\int_0^t \1 \{x \le \X_s < x + h\} ds}{\m ([x, x + h))}
&\text{if }x=l^*\in I,
\\[2mm]
\limsup_{h \searrow 0}
\frac{\int_0^t \1 \{x-h < \X_s \le x\} ds}{\m ((x-h, x])}
&\text{if }x=r^*\in I,
\end{cases}
\]
which is a measurable function with values in \([0, \infty]\).
Let \(\g^{-1}\) be the right-inverse of \(\g\), i.e.,
\[
\g^{-1} (t) \triangleq \inf (s \geq 0 \colon \g(s+) > t), \quad t \in \mathbb{R}_+.
\]
The following theorem provides the chain rule for diffusions:

\begin{theorem} \label{theo: chain rule}
	\(\P _x\circ \X_{\g^{-1}(\cdot)}^{-1} = \Q_x\) for all \(x \in I\).
\end{theorem}

\begin{proof}
Under $(J\ni x\mapsto\P_x)$, the process $\s(\X)$ is a diffusion on natural scale (Lemma~\ref{lem: diff homo}).
Lemma~\ref{lem: occ formula diff} yields for the
diffusion
local time process
$\widehat L(\s(\X))$
that, for all $x\in I$, $\P_x$-a.s.\ we have, for all $t\in \mathbb{R}_+$ and 
all $K \in \mathcal{B}(I)$,
\begin{align*}
\int_0^t\1\{\X_s\in K\}\,ds
&=
\int_0^t\1\{\s(\X_s)\in\s(K)\}\,ds
=
\int_{\s(K)}\widehat L_t^{y}(\s(\X))\,\m\circ\s^{-1}(dy)
\\
&=
\int_K \widehat L_t^{\s(y)}(\s(\X))\,\m(dy)
=
\int_K \widehat L_t^{\S(y)}(\s(\X))\,\m(dy),
\end{align*}
where in the last equality we use that \(\s = \S\) on \(I\).
The continuity of $\widehat L(\s(\X))$ in the space variable implies that the above defined $\{\l(t,x)\colon (t, x) \in \mathbb{R}_+ \times I\}$ provides a version of the processes $\{\widehat L_t^{\S(x)}(\s(\X))\colon t \in \mathbb{R}_+\}$ simultaneously for all $x\in I$.
This means that, for all $x\in I$, $\P_x$-a.s.\ we have, for all $t\in \mathbb{R}_+$,
\begin{equation}
\g (t) = \begin{cases}\label{eq:200922a1}
\int_{\S(I)} \widehat L_t^{x}(\s(\X))\,\M\circ\S^{-1}(dx),&t < \e,\\ \infty,& t \geq \e,
\end{cases}
\end{equation}
Finally, the diffusion on natural scale
$(\S(I)\ni x\mapsto\Q_{\S^{-1}(x)}\circ\S(\X)^{-1})$
with speed measure $\M\circ\S^{-1}$
is obtained from the diffusion\footnote{Viewed under $(J\ni x\mapsto\P_x)$.} $\s(\X)$ considered above in this proof
via the subordination inverse to~\eqref{eq:200922a1} (\cite[p.~177]{itokean74}).
This concludes the proof.
\end{proof}

\subsection{Separating Times for It\^o Diffusions}
The separating time for It\^o diffusions was studied in \cite{cherUru,MU12}. Thanks to some time-change and symmetrization arguments we can reduce certain steps in the proof of Theorem \ref{theo: main1} to the It\^o diffusion setting. The purpose of this short subsection is to recall the result from \cite{cherUru,MU12} for later reference. We pose ourselves in the setup from Example~\ref{ex: ito diffusion}, i.e., we assume that \(J^\circ = (l, r)\) and 
\[
\s (x) \triangleq \int^x \exp \Big( -  \int^z \frac{2 b (y) }{\sigma^2(y)} dy\Big) dz, \qquad \S (x) \triangleq \int^x \exp \Big( -  \int^z \frac{2 \tilde{b} (y)}{\tilde{\sigma}^2(y)} dy\Big) dz,
\]
and 
\[
\m (dx) \triangleq \frac{dx}{\s' (x) \sigma^2(x)}, \qquad \M (dx) \triangleq \frac{dx}{\S'(x) \tilde{\sigma}^2(x)},
\]
where \(b, \tilde{b}, \sigma, \tilde{\sigma} \colon J^\circ \to \mathbb{R}\) are Borel functions such that
\[
\sigma^2, \tilde{\sigma}^2 > 0\;\;\text{everywhere on }J^\circ, \qquad \frac{1 + |b|}{\sigma^2}, \frac{1 + |\tilde{b}|}{\tilde{\sigma}^2} \in L^1_\textup{loc}(J^\circ).
\]
Furthermore, we suppose that \(\m (\{l\}) \equiv \m (\{r\}) \equiv \M(\{l\}) \equiv \M(\{r\}) \equiv \infty\) in case the points are accessible.
We now recall \cite[Theorem~5.5]{MU12} (see also \cite[Theorem~5.1]{cherUru}):

\begin{lemma}\label{lem: sep ito diff}
Theorem \ref{theo: main1} holds for the above setting.
\end{lemma}

\subsection{Facts on Semimartingales}
We start by recalling the semimartingale occupation time formula, the generalized It\^o formula and some useful facts about the semimartingale local time process.\footnote{Regarding the normalizations in the literature, we stress that the local time in \cite{KaraShre} is half of the local time in \cite{kallenberg,LeGall,RY,RW2} which we use in this paper.}
The following lemma contains \cite[Theorem 29.5]{kallenberg}, \cite[Exercise VI.1.23]{RY} and \cite[Exercise 9.19]{LeGall}.

\begin{lemma} \label{lem: occ smg}
Let \(Y\) be a continuous semimartingale with the continuous in time and c\`adl\`ag in space local time process \(\{L^x_t (Y) \colon (t, x) \in \mathbb{R}_+ \times \mathbb{R}\}\). 
\begin{enumerate}
\item[\textup{(i)}]
Almost surely, simultaneously for all Borel functions \(\f \colon \mathbb{R}\to \mathbb{R}_+\) it holds that 
\[
\int_0^t \f (Y_s) d \langle Y, Y\rangle_s = \int \f (z) L^z_t (Y) dz, \quad t \in \mathbb{R}_+,
\]
where $\langle Y,Y\rangle$ denotes the quadratic variation process of~$Y$.

\item[\textup{(ii)}]
Almost surely, simultaneously for all functions \(\f \colon \mathbb{R} \to \mathbb{R}\) that are differences of two convex functions we have 
\begin{equation}\label{eq:020423c1}
\f (Y_t) = \f (Y_0) + \int_0^t \Big(\frac{d^- \f}{dx}\Big) (Y_s) d Y_s + \frac{1}{2} \int L^y_t (Y) \f'' (dy), \quad t\in\mathbb R_+,
\end{equation}
where \(\f''(dx)\) denotes the second derivative measure of \(\f\) defined by 
\begin{equation}\label{eq:020423c2}
\f '' ((x, y]) = \Big(\frac{d^+ \f}{dx}\Big) (y) - \Big(\frac{d^+ \f}{dx}\Big) (x), \quad x \leq y.
\end{equation}
In particular, $\f(Y)$ is a continuous semimartingale.

\item[\textup{(iii)}]
Let \(\f \colon \mathbb{R} \to \mathbb{R}\) be a strictly increasing function which is a difference of two convex functions.  Almost surely, simultaneously for all $(t,z)\in\mathbb R_+\times\mathbb R$ it holds
\[
L_t^{\f (z)} (\f (Y)) = \Big(\frac{d^+ \f}{dx}\Big) (z) L^z_t (Y).
\]

\item[\textup{(iv)}]
Almost surely, simultaneously for all \(a \in \mathbb{R}\) it holds \(\on{supp}(d_t L^a_t (Y)) \subset \{t \in \mathbb{R}_+ \colon Y_t = a\}\).
\end{enumerate}
\end{lemma}

\begin{remark}\label{rem:021022a1}
In part~(ii) in the formula for $\f(Y_t)$ one can replace the left-hand derivative $d^-\f/dx$ by the right-hand derivative $d^+\f/dx$ when one takes the left-continuous (in the space variable) local time process.
Similarly, the analogue of part~(iii) for the left-continuous local time contains the left-hand derivative~$d^-\f/dx$.
\end{remark}

There is a delicate point related to part~(ii) of Lemma~\ref{lem: occ smg}.
It is tempting to write~\eqref{eq:020423c1} also in case when, say, $Y$ is a $[0,\infty)$-valued semimartingale and $\f\colon[0,\infty)\to\mathbb R$ is a difference of two convex continuous functions $[0,\infty)\to\mathbb R$.
But this would no longer be true in general.
A counterexample is given by $Y=|W|$ with a Brownian motion $W$ and $\f=\sqrt{\, \cdot\,}$ (indeed, $\sqrt{|W|}$ is not a semimartingale).
We, therefore, present also the generalized It\^o formula for semimartingales taking values in half-open intervals.

\begin{lemma} \label{lem: extended gen Ito formula}
	Let \(I = [a, b)\) be a half-open (possibly unbounded) interval and let \(Y\) be an \(I\)-valued continuous semimartingale. Furthermore, let \(\f \colon I \to \mathbb{R}\) be a continuous function such that \(d^+ \f / dx\) exists everywhere on \([a, b)\) as a right-continuous function with locally finite variation
(recall the equivalent conditions \textup{(a$'$)}, \textup{(a$''$)} and \textup{(b$'$)} preceding Lemma~\ref{lem: generator}).
Then, a.s.\ for all \(t \in \mathbb{R}_+\),
	\begin{equation}\label{eq:020423c3}
	\f (Y_t) = \f (Y_0) + \int_0^t \Big( \frac{d^+ \f}{dx} \Big) (Y_s) d Y_s + \frac{1}{2} \int_{(a, b)} L^{x-}_t (Y)\f'' (dx),
	\end{equation}
	where \(\{L_t^x (Y) \colon (t, x) \in \mathbb{R}_+ \times \mathbb{R}\}\) denotes the continuous in time and c\`adl\`ag in space local time process of the process \(Y\), and $\f''(dx)$ is the measure defined on the \emph{open} interval $(a,b)$ by \eqref{eq:020423c2} with $a\le x\le y<b$.
In particular, if \(\f \in C^1 ([a, b); \mathbb{R})\) with absolutely continuous derivative, then a.s.\ for all $t\in\mathbb R_+$
	\begin{equation}\label{eq:020423c4}
	\f (Y_t) = \f (Y_0) + \int_0^t \f' (Y_s) d Y_s + \frac{1}{2} \int_0^t \f'' (Y_s) d \langle Y, Y\rangle_s,
	\end{equation}
where the function $\f''$ is the second derivative of~$\f$
(which is well-defined almost everywhere with respect to the Lebesgue measure).
\end{lemma}

\begin{proof}
For notational simplicity, we prove the claim for $I=[0,\infty)$.
By hypothesis, there exists a function \(g \colon \mathbb{R} \to \mathbb{R}\) that is the difference of two convex functions such that \(g = \f\) on \([0, \infty)\).
Further, as a.s. \(L^y (Y) = 0\) for all \(y \in (- \infty, 0)\), we have a.s. \(L^{y-} (Y) = 0\) for all \(y \in (- \infty, 0]\). Now, Lemma~\ref{lem: occ smg} and Remark~\ref{rem:021022a1} yield that, a.s.\ for all $t\in\mathbb R_+$,
\begin{align*}
\f (Y_t) = g (Y_t)
&=
g (Y_0) + \int_0^t \Big( \frac{d^+ g}{dx} \Big) (Y_s)\,d Y_s + \frac{1}{2} \int_{- \infty}^\infty L^{x-}_t (Y)\,g'' (dx) 
\\
&=
\f (Y_0) + \int_0^t \Big( \frac{d^+ \f}{dx} \Big) (Y_s)\,d Y_s + \frac{1}{2} \int_{(0, \infty)} L^{x-}_t (Y)\,\f'' (dx),
\end{align*}
which is~\eqref{eq:020423c3}.
It is worth noting that, if we applied \eqref{eq:020423c1} directly (i.e., without Remark~\ref{rem:021022a1}),
we would get a dependence on $d^-g(0)/dx$ in both integral terms, which looks puzzling at first glance, as $d^-g(0)/dx$ is not uniquely defined.
However, using \cite[Theorem~VI.1.7]{RY} one can see that $d^-g(0)/dx$ cancels.

For the second claim, suppose that \(\f \in C^1 ([a, b); \mathbb{R})\) with absolutely continuous derivative. Then, by the occupation time formula (part~(i) of Lemma~\ref{lem: occ smg}), a.s.\ for all $t\in\mathbb R_+$,
\[
\int_{(a, b)} L^{x-}_t (Y) \f'' (dx) = \int_a^b L^{x-}_t (Y) \f'' (x) dx = \int_a^b L^x_t (Y) \f'' (x) dx = \int_0^t \f'' (Y_s) d \langle Y, Y\rangle_s.
\]
This observation establishes~\eqref{eq:020423c4}.
\end{proof}

The following lemma contains classical facts on stability of the local martingale and semimartingale property under time-changes. It is implied by \cite[Corollary~10.12, Theorem 10.16]{Jacod}.

\begin{lemma}\label{lem: change of time}
Let \((\mathcal{G}_t)_{t \geq 0}\) be a right-continuous (complete) filtration and let \((L (t))_{t \geq 0}\) be a finite time-change, i.e., a family of (a.s.) finite \((\mathcal{G}_t)_{t \geq 0}\)-stopping times such that \(t \mapsto L (t)\) is (a.s.) increasing and right-continuous.
\begin{enumerate}
\item[\textup{(i)}]
If \(Y\) is a \((\mathcal{G}_t)_{t \geq 0}\)-semimartingale, then \(Y_L\) is a \((\mathcal{G}_{L (t)})_{t \geq 0}\)-semimartingale.

\item[\textup{(ii)}]
If \(Y\) is a (continuous) local \((\mathcal{G}_t)_{t \geq 0}\)-martingale such that a.s. \(Y\) is constant on every interval \([L(t-), L(t)]\), then \(Y_L\) is a (continuous) local \((\mathcal{G}_{L (t)})_{t \geq 0}\)-martingale.
\end{enumerate}
\end{lemma}

Let us also recall a useful formulation of Girsanov's theorem, which is a version of \cite[Theorem~III.3.24]{JS} for one-dimensional (continuous) semimartingales.

\begin{lemma} \label{lem: Girs}
	We consider a filtered measurable space \((\Sigma, \mathcal{G}, (\mathcal{G}_t)_{t \geq 0})\) (with a right-continuous filtration) and two probability meassures \(\P\) and \(\Q\) defined on it. Suppose that \(\Q\ll_\textup{loc} \P\). If \(Y\) is a (continuous) \(\P\)-semimartingale, then it is also a (continuous) \(\Q\)-semimartingale. Moreover, if \(Y\) is a continuous local \(\P\)-martingale with \(\langle Y, Y\rangle = \int_0^\cdot c_s ds\) (where \(c\) is a non-negative predictable process and the integral is well-defined), then there exists a predictable process \(\beta\) such that \(\Q\)-a.s.
	\[
	\int_0^t \beta^2_s c_s ds < \infty, \quad t \in \mathbb{R}_+, 
	\]
	and 
	\[
	Y - \int_0^\cdot \beta_s c_s ds 
	\]
	is a continuous local \(\Q\)-martingale with quadratic variation \(\int_0^\cdot c_s ds\).
\end{lemma}

\subsection{Semimartingale Functions in a Non-Markovian Setting}
It is well-known (\cite[Theorem~5.5]{CinJPrSha}) that if \(W\) is a Brownian motion and \(\f \colon \mathbb{R} \to \mathbb{R}\) is a Borel function, the process \(\f (W)\) is a semimartingale\footnote{Semimartingales are always assumed to have c\`adl\`ag paths.} if and only if \(\f\) is the difference of two convex functions. The following
theorem
is a variation of this result in the sense that we consider Brownian motion up to a strictly positive stopping time and deduce properties of the function in a neighborhood of the origin.
As a stopped Brownian motion is, in general, no longer Markovian, we cannot use the argument from \cite{CinJPrSha}, which essentially hinges on the Markov property.

\begin{theorem} \label{lem: diff convex}
	Let \((\Sigma, \mathcal{G}, (\mathcal{G}_t)_{t \geq 0}, \P)\) be a filtered probability space which supports a standard \((\mathcal{G}_t)_{t \geq 0}\)-Brownian motion \(W\) (starting in the origin). Furthermore, let \(\tau\) be a \((\mathcal{G}_t)_{t \geq 0}\)-stopping time such that \(\P (\tau > 0) > 0\) and let \(\f \colon \mathbb{R} \to \mathbb{R}\) be a Borel function. Suppose that \(\f(W_{\cdot \wedge \tau})\) is a \((\mathcal{G}_t)_{t \geq 0}\)-semimartingale. 
Then, there exists a \(\delta > 0\) such that
the restriction
\(\f |_{(- \delta, \delta)}\) is the difference of two convex functions
from $(-\delta,\delta)$ into $\mathbb R$
(in particular, $\f$ is continuous on $(-\delta,\delta)$).
\end{theorem}

Before we prove this result, let us stress that we recover the classical global result from~\cite{CinJPrSha}.

\begin{corollary}
	Let \(\f \colon \mathbb{R} \to \mathbb{R}\) be a Borel function and let \(W\) be a standard Brownian motion. If \(\f (W)\) is a semimartingale, then \(\f\) is the difference of two convex functions
from \(\mathbb{R}\) into \(\mathbb{R}\).
\end{corollary}

\begin{proof}
Take \(x \in \mathbb{R}\). By Lemma \ref{lem: change of time}, the process \(\f (W_{\cdot + T_x (W)})\) is a semimartingale.
As \(W_{\cdot + T_x (W)} - x\) is a standard Brownian motion,
Theorem~\ref{lem: diff convex} yields the existence of a \(\delta > 0\) such that \( (- \delta, \delta) \ni y \mapsto \f (y + x)\) is the difference of two convex functions,
i.e., \(\f\) is the difference of two convex functions on \((x - \delta, x+ \delta)\).
Thus, as being the difference of two convex functions is a local property
(use the equivalence between (a) and~(b) before Lemma~\ref{lem: generator}, or, alternatively, apply \cite[(I) on p.~707]{hart}),
it follows that \(\f \colon \mathbb{R} \to \mathbb{R}\) is the difference of two convex functions.
\end{proof}

	Theorem~\ref{lem: diff convex} follows by an adjustment of the argument in \cite[Theorem~2.1]{prokaj}, which relies on \cite[Proposition~6.3]{prokaj}, which provides a useful criterion for a continuous function to be the difference of two convex ones. For convenience, let us restate \cite[Proposition~6.3]{prokaj}:
	
	\begin{lemma} \label{lem: convex difference}
		Let \(\f \colon \mathbb{R} \to \mathbb{R}\) be a continuous function and let \(I\) be an open interval. Then,  \(\f \colon I \to \mathbb{R}\) is the difference of two convex functions if and only if for any compact interval \(K \subset I\)
		\[
		\limsup_{\beta \searrow 0} \frac{1}{\beta} \sum_{x \in K \cap  \beta \mathbb{Z}} \big| \f (x + \beta) + \f(x - \beta) - 2 \f (x)\big| < \infty.
		\]
	\end{lemma}

\begin{proof}[Proof of Theorem~\ref{lem: diff convex}]
We first show that $\f$ is continuous in a sufficiently small neighborhood of zero.
For contradiction, assume that there exists a sequence $(x_n)_{n \in \mathbb{N}}$ of points such that $x_n\to0$ and such that $\f$ is not continuous at each $x_n$.
Since a.s.\ ${T_{x_n}(W)\to0}$, we can take a sufficiently large $N\in\mathbb N$ such that $\P(T_{x_N}(W)<\tau)>0$.
By the strong Markov property, $(W_{T_{x_N}(W)+s}-x_N)_{s\ge0}$ is a Brownian motion, hence has oscillating behavior as $s\searrow0$.
Therefore, a.s.\ on $\{T_{x_N}(W)<\tau\}$ it holds
$$
\limsup_{t\searrow T_{x_N}(W)}\f(W_{t\wedge\tau})
=\limsup_{x\to x_N}\f(x)
>\liminf_{x\to x_N}\f(x)
=\liminf_{t\searrow T_{x_N}(W)}\f(W_{t\wedge\tau}).
$$
This contradicts the right-continuity and hence, the semimartingale property of $\f(W_{\cdot\wedge\tau})$.
Therefore, we can find an $\varepsilon>0$ such that $\f$ is continuous on $(-2\varepsilon,2\varepsilon)$ and
$\P (T_{- \varepsilon} (W) \wedge T_{\varepsilon} (W) < \tau) > 0$.
Now, replacing $\f$ with the continuous function
$\f(-\varepsilon)\1_{(-\infty,-\varepsilon)}+\f\1_{[-\varepsilon,\varepsilon]}+\f(\varepsilon)\1_{(\varepsilon,\infty)}$
and $\tau$ with $\tau\wedge T_{- \varepsilon} (W) \wedge T_{\varepsilon} (W)$,
we see that we can w.l.o.g.\ assume that \(\f \colon \mathbb{R} \to \mathbb{R}\) is continuous.

Next, also notice that we can w.l.o.g.\ assume that \(\P (\tau > 0) = 1\). Indeed, in case \(\P (\tau > 0) < 1\), we can simply replace the probability measure \(\P\) by 
\[
\Q (d \omega) \triangleq \P (d \omega | \tau> 0) = \frac{\P (d \omega \cap \{\tau> 0\})}{\P (\tau > 0)}.
\]
Clearly, as \(\{\tau> 0\}\in \mathcal{G}_0\), \(W\) remains a \((\mathcal{G}_t)_{t \geq 0}\)-Brownian motion under \(\Q\) and, since \(\Q \ll \P\), Girsanov's theorem yields that \(\f (W_{\cdot \wedge \tau})\) is also a \(\Q\)-\((\mathcal{G}_t)_{t \geq 0}\)-semimartingale.
Thus, from now on, we assume that \(\f \colon \mathbb{R} \to \mathbb{R}\) is continuous and \(\P (\tau>0) = 1\).

	Fix \(\beta \in (0, 1)\) and define inductively
	\[
	\tau_0 \triangleq 0, \qquad \tau_n \triangleq \inf (t > \tau_{n - 1} \colon |W_t - W_{\tau_{n - 1}}| \geq \beta), \ \ n \in \mathbb{N}.
	\]
	It is easy to see that \((S_n \triangleq W_{\tau_n})_{n = 0}^\infty\) is a simple random walk starting in the origin with step size \(\beta\). 
	By hypothesis, 
	\[
	\f (W_{\cdot \wedge \tau}) = \f (0) + M + A, 
	\]
	where \(M\) is a continuous local martingale and \(A\) is a continuous process of (locally) finite variation both starting in zero. We denote the variation process of \(A\) by \(V\).
Take some $K>1$ and define
	\[
	\rho_K \triangleq \inf (t \geq 0 \colon|M_t| \vee |V_t| \vee |W_t| \geq K) \wedge K \wedge \tau.
	\]
	Next, take \(I = (a, c)\) with \(a < 0 < c\) small enough such that 
	\[
	\inf_{x \in I} \E \big[ L^x_{\rho_K} (W) \big] > 0,
	\]
	where \(\{L^x_t (W) \colon (t, x) \in \mathbb{R}_+ \times \mathbb{R}\}\) denotes the jointly continuous modification of the local time process of~\(W\).

Let us prove that such an interval \(I\) exists.
Thanks to the generalized It\^o formula (Lemma~\ref{lem: occ smg}), we have \[\E [L^x_{\rho_K} (W)] = \E[ |W_{\rho_K} - x|] - |x|,\]
and hence, the map \(x \mapsto \E [L^x_{\rho_K} (W)]\) is continuous.
Lemma~\ref{lem: pos LT} implies that $\E[L^0_{\rho_K}(W)]>0$.
Now, for a continuous function which is strictly positive in zero, we can find a sufficiently small neighborhood $I$ of the origin where this function is bounded away from zero.
	
	We set
	\[
	l (x, t) \triangleq \sum_{m = 0}^\infty \1 \{S_m = x, \tau_m < t\}, \quad (x,t) \in \mathbb{R} \times \mathbb{R}_+.
	\]
For \(n \in \mathbb{Z}_+\),
on the event \(\{\tau_n  < \rho_K\}\)
we get
\begin{align*}
	\big|\tfrac{1}{2} \big( \f (S_n + \beta) + \f(S_n - \beta) \big) - \f(S_{n}) \big|
	&= \big| \E\big[ \f (S_{n + 1}) - \f (S_{n}) \big| \mathcal{G}_{\tau_{n}}\big] \big|
	\\&= \big|\E\big[\f (W_{\tau_{n + 1}}) - \f(W_{\tau_{n + 1} \wedge \rho_K}) \big| \mathcal{G}_{\tau_n} \big] + \E\big[ A_{\tau_{n + 1} \wedge \rho_K} - A_{\tau_{n}} \big| \mathcal{G}_{\tau_n}\big] \big|
	\\&\leq \E\big[ |\f (W_{\tau_{n + 1}}) - \f(W_{\tau_{n + 1} \wedge \rho_K})| \big| \mathcal{G}_{\tau_n} \big] + \E\big[ |A_{\tau_{n + 1} \wedge \rho_K} - A_{\tau_{n}}| \big| \mathcal{G}_{\tau_n}\big].
\end{align*}
Consequently, we obtain
	\begin{align*}
	\sum_{k \colon k \beta \in I} & \tfrac{1}{2}  \big| \f ((k + 1) \beta) + \f((k - 1)\beta) - 2 \f(k \beta)\big| \E\big[ l (k \beta, \rho_K)\big] 
	\\&\hspace{0.45cm} = \E \Big[ \sum_{m = 0}^\infty\;\sum_{\substack{k \colon\! k \beta \in I}} \tfrac{1}{2} \big| \f (S_{m} + \beta) + \f(S_m - \beta) - 2 \f(S_m)\big| \1 \{S_m = k \beta, \tau_m < \rho_K\} \Big]
		\\&\hspace{0.45cm} \leq \E \Big[ \sum_{m = 0}^\infty\big(\E\big[ |\f (W_{\tau_{m + 1}}) - \f(W_{\tau_{m + 1} \wedge \rho_K})| \big| \mathcal{G}_{\tau_m} \big] 
		+ \E\big[ |A_{\tau_{m + 1} \wedge \rho_K} - A_{\tau_{m}}| \big| \mathcal{G}_{\tau_m}\big]\big) \1 \{\tau_m < \rho_K\} \Big]
		\\&\hspace{0.45cm} =  \sum_{m = 0}^\infty \E \Big[\big(\E\big[ |\f (W_{\tau_{m + 1}}) - \f(W_{\tau_{m + 1} \wedge \rho_K})| \big| \mathcal{G}_{\tau_m} \big] 
		+ \E\big[ |A_{\tau_{m + 1} \wedge \rho_K} - A_{\tau_{m}}| \big| \mathcal{G}_{\tau_m}\big]\big) \1 \{\tau_m < \rho_K\} \Big]
		\\&\hspace{0.45cm} =  \sum_{m = 0}^\infty \E \Big[\big( |\f (W_{\tau_{m + 1}}) - \f(W_{\tau_{m + 1} \wedge \rho_K})| +  |A_{\tau_{m + 1} \wedge \rho_K} - A_{\tau_{m}}|\big) \1 \{\tau_m < \rho_K\} \Big]
		\\&\hspace{0.45cm}\leq \E\Big[ \sum_{m = 0}^\infty |\f (W_{\tau_{m + 1}}) - \f(W_{\tau_{m + 1} \wedge \rho_K})| \1 \{ \tau_m < \rho_K < \tau_{m + 1} \} + V_{\rho_K} \Big]
		\\&\hspace{0.45cm}\leq \sup_{|x| \leq K + 1} 2|f (x)| + K.
	\end{align*}
	We now estimate the expectation \(\E[l (k \beta, \rho_K)]\). 
	Notice that, for all \(n \in \mathbb{Z}_+\), on the event \(\{\tau_n < \rho_K\}\)
	\begin{align} \label{eq: small observation}
	|W_{\tau_{n + 1} \wedge \rho_K} - k \beta| - |W_{\tau_n} - k \beta| = \begin{cases} W_{\tau_{n + 1} \wedge \rho_K} - W_{\tau_n}& \text{if } W_{\tau_n} \geq (k + 1) \beta,\\
	W_{\tau_{n}} - W_{\tau_{n +1} \wedge \rho_K}&\text{if }W_{\tau_n} \leq (k - 1) \beta,\\
	|W_{\tau_{n + 1} \wedge \rho_K} - k \beta|&\text{if } S_n = W_{\tau_n} = k \beta.\end{cases}
	\end{align}
For every \(n \in \mathbb{N}\), the generalized It\^o formula (Lemma~\ref{lem: occ smg}) yields that \(\P\)-a.s.
		\begin{align*}
		 \big(|W_{\tau_{n + 1} \wedge \rho_K} &- k \beta| - |W_{\tau_n} - k \beta| \big)\1 \{\tau_n < \rho_K\} 
		 \\&\quad = \Big(L^{k \beta}_{\tau_{n + 1} \wedge \rho_K} (W) - L^{k \beta}_{\tau_n} - \int_{\tau_n}^{\tau_{n + 1} \wedge \rho_K} \on{sgn} (W_s - k \beta) d W_s\Big) \1 \{\tau_n < \rho_K\}.
		 \end{align*}
		 Computing the conditional expectation \(\E[ \hspace{0.075cm}\cdot\hspace{0.075cm} | \mathcal{G}_{\tau_n}]\) of both sides and taking \eqref{eq: small observation} into account, we obtain from the martingale property of Brownian motion and the stochastic integral that \(\P\)-a.s.
		 \begin{align*}
		 \E\big[ |W_{\tau_{n + 1} \wedge \rho_K} &- k \beta| \1 \{S_n = k \beta, \tau_n < \rho_K\} | \mathcal{G}_{\tau_n}\big] 
		 = \E \big[ \big(L^{k \beta}_{\tau_{n + 1} \wedge \rho_K} - L^{k \beta}_{\tau_n}\big) \1 \{\tau_n < \rho_K\} | \mathcal{G}_{\tau_n} \big].
		 \end{align*}
		 Consequently, taking the expectation of both sides, we obtain that
		 \[ 
		  \E \big[ \big(L^{k \beta}_{\tau_{n + 1} \wedge \rho_K} - L^{k \beta}_{\tau_n}\big) \1 \{\tau_n < \rho_K\} \big] \leq \beta \P (S_n = k \beta, \tau_n < \rho_K).
		 \]
		 Summing over \(n\) from \(0\) to \(\infty\) and using Fubini's theorem yields that 
		 	\[
		 \E \big[ L^{k \beta}_{\rho_K} (W) \big] \leq \beta \E\big[ l (k \beta, \rho_K)\big].
		 \]	
	Putting all pieces together, we conclude that 
	\[
	\sum_{k \colon k \beta \in I} \tfrac{1}{\beta} \big| \f ((k + 1) \beta) + \f((k - 1)\beta) - 2 \f(k \beta)\big| \leq \frac{\max_{|x| \leq K + 1} 4 | \f (x) | + 2 K}{\inf_{x \in I} \E[ L^x_{\rho_K} (W)]} < \infty.
	\]
	Finally, we conclude from Lemma \ref{lem: convex difference} that the function \(\f \colon I \to \mathbb{R}\) is the difference of two convex functions.
\end{proof}

\subsection{A useful identity for natural filtrations}

For a process \(Y = (Y_t)_{t \geq 0}\) with paths in \(\Omega\), we define its natural filtration by 
\[
\mathcal{F}^Y_t \triangleq \bigcap_{\varepsilon > 0} \sigma (Y_{s}, s \leq t + \varepsilon), \quad t \in \mathbb{R}_+.
\]
The following technical observation is useful in some of our arguments.

\begin{lemma} \label{lem: sigma field stopped id}
	Let \(\tau\) be an \((\mathcal{F}_t)_{t \geq 0}\)-stopping time and let \(Y\) be a process with paths in \(\Omega\). Then, \(\rho \triangleq \tau \circ Y\) is an \((\mathcal{F}^Y_t)_{t \geq 0}\)-stopping time and 
	\begin{align} \label{eq: stopped sigma id}
		Y^{-1} (\mathcal{F}_\tau) = \mathcal{F}^Y_\rho.
	\end{align}
\end{lemma}
\begin{proof}
	The fact that \(\rho\) is an \((\mathcal{F}^Y_t)_{t \geq 0}\)-stopping time is part~(a) of \cite[Proposition 10.35]{Jacod}. We now prove identity~\eqref{eq: stopped sigma id}. Recall from \cite[Problem~1, p.~88]{itokean74} that 
	\begin{align*}
		\mathcal{F}_\tau &= \bigcap_{\varepsilon > 0} \sigma (\X_{t \wedge (\tau + \varepsilon)}, t \geq 0),\qquad
		\mathcal{F}^Y_\rho = \bigcap_{\varepsilon > 0} \sigma (Y_{t \wedge (\rho + \varepsilon)}, t \geq 0).
	\end{align*}
	Thus, for every \(\varepsilon > 0\), we have
	\[
	Y^{-1} (\mathcal{F}_\tau) \subset Y^{-1} (\sigma (\X_{t \wedge (\tau + \varepsilon)}, t \geq 0)) = \sigma (Y_{t \wedge (\rho + \varepsilon)}, t \geq 0), 
	\]
	which implies that \(Y^{-1} (\mathcal{F}_\tau) \subset \mathcal{F}^Y_\rho\). 
	To establish the converse inclusion, take \(A \in \mathcal{F}^Y_\rho\). Then, for every \(n \in \mathbb{N}\), \(A \in \sigma (Y_{t \wedge (\rho + 1/n)}, t \geq 0) = Y^{-1} (\sigma (\X_{t \wedge (\tau + 1/n)}, t \geq 0))\), which means there exists a set \(B_n \in \sigma (\X_{t \wedge (\tau + 1/n)}, t \geq 0)\) such that \(A = Y^{-1} (B_n)\).
	Set \(G \triangleq \bigcap_{n \in \mathbb{N}} \bigcup_{m \geq n} B_m\) and notice that \(G \in \mathcal{F}_\tau\). Finally, as \(Y^{-1} (G) = \bigcap_{n \in \mathbb{N}} \bigcup_{m \geq n} Y^{-1} (B_m) = A\), we conclude that \(A \in Y^{-1} (\mathcal{F}_\tau)\), which finishes the proof of~\eqref{eq: stopped sigma id}.
\end{proof}

\section{Generalized Density and Differentiation of Measures}\label{app: C}

\subsection{Generalized Density}\label{subsec:GenDen}
Let $(\Omega,\mathcal F)$ be a measurable space and let $\mu,\nu$ be $\sigma$-finite measures on $(\Omega,\mathcal F)$.
\emph{Generalized Radon--Nikodym derivative} (or \emph{generalized density})
$\partial\mu/\partial\nu$
is an $\mathcal F$-measurable mapping $\Omega\to[0,\infty]$ such that
\begin{equation}\label{eq:260922a1}
\mu,\nu\text{-a.e.}\quad
\frac{\partial\mu}{\partial\nu}
=
\frac{d\mu/d\gamma}{d\nu/d\gamma},
\end{equation}
where $\gamma$ is any $\sigma$-finite measure on $(\Omega,\mathcal F)$ such that $\mu\ll\gamma$ and $\nu\ll\gamma$
($d\mu/d\gamma$ and $d\nu/d\gamma$ denote the corresponding Radon--Nikodym derivatives).
To define a version of $\partial\mu/\partial\nu$, we can simply take the right-hand side of~\eqref{eq:260922a1} with any convention for how to understand $0/0$, e.g., $0/0\triangleq1$.\footnote{In fact, as $(\mu+\nu)(d\mu/d\gamma=d\nu/d\gamma=0)=0$, it does not matter which convention to choose for $0/0$.}
We emphasize that this definition does not depend on the choice of the dominating measure $\gamma$ because it holds that $\mu+\nu\ll\gamma$,
\begin{align*}
\gamma\text{-a.e., hence, }\mu,\nu\text{-a.e.}\quad
\frac{d\mu}{d\gamma}
&=
\frac{d\mu}{d(\mu+\nu)}
\frac{d(\mu+\nu)}{d\gamma},
\\
\gamma\text{-a.e., hence, }\mu,\nu\text{-a.e.}\quad
\frac{d\nu}{d\gamma}
&=
\frac{d\nu}{d(\mu+\nu)}
\frac{d(\mu+\nu)}{d\gamma},
\end{align*}
and
$(\mu+\nu)(d(\mu+\nu)/d\gamma=0)=0$, thus,
$$
\mu,\nu\text{-a.e.}\quad
\frac{d\mu/d\gamma}{d\nu/d\gamma}
=
\frac{d\mu/d(\mu+\nu)}{d\nu/d(\mu+\nu)},
$$
where the latter expression does not depend on $\gamma$.
We also notice that in the case $\mu\ll\nu$ the generalized Radon--Nikodym derivative is nothing else but the standard one (take $\gamma=\nu$ in~\eqref{eq:260922a1}).

The generalized Radon--Nikodym derivative appears without a special name
in \cite[Theorem~A.17]{FollmerSchied} and in \cite[Theorem~7.1]{Jacod},
and it appears under the name \emph{Lebesgue derivative} in formula~(29) of \cite[Section~III.9]{Shiryaev}.
This object
is a convenient tool to concisely describe the mutual arrangement between $\mu$ and $\nu$ from the viewpoint of absolute continuity and singularity. We illustrate this by the following lemma. Its proof is straightforward and therefore omitted.

\begin{lemma}\label{lem:260922a1}
There exist pairwise disjoint sets $E,S_\mu,S_\nu\in\mathcal F$ such that
$$
\Omega=E\uplus S_\mu\uplus S_\nu,
\quad
\mu\sim\nu\text{ on }\mathcal F\cap E,
\quad
\mu(S_\nu)=\nu(S_\mu)=0
$$
(in particular, $\mu\perp\nu$ on $\mathcal F\cap E^c$, where $E^c\triangleq\Omega\setminus E$).
Such sets $E,S_\mu,S_\nu$ are $\mu,\nu$-a.e.\ unique.
Moreover, $\mu,\nu$-a.e.\ it holds
$$
E=\Big\{\frac{\partial\mu}{\partial\nu}\in(0,\infty)\Big\},
\quad
S_\mu=\Big\{\frac{\partial\mu}{\partial\nu}=\infty\Big\},
\quad
S_\nu=\Big\{\frac{\partial\mu}{\partial\nu}=0\Big\}.
$$
\end{lemma}

\begin{corollary}\label{cor:260922a1}
\begin{enumerate}
\item[\textup{(i)}]
$\mu$-a.e.\ $\partial\mu/\partial\nu$ is $(0,\infty]$-valued.

\item[\textup{(ii)}]
$\mu\ll\nu$ if and only if $\mu$-a.e.\ $\partial\mu/\partial\nu<\infty$.

\item[\textup{(iii)}]
$\mu\perp\nu$ if and only if $\mu$-a.e.\ $\partial\mu/\partial\nu=\infty$.
\end{enumerate}
\end{corollary}

Another convenient feature of the generalized density is its symmetry in $\mu$ and~$\nu$ (regardless of the mutual arrangement between $\mu$ and~$\nu$):
$\mu,\nu$-a.e.\ we, clearly, have
$$
\frac{\partial\nu}{\partial\mu}=\frac1{\partial\mu/\partial\nu}.
$$
For an illustration, we now rewrite Corollary~\ref{cor:260922a1},
which views the generalized density $\partial\mu/\partial\nu$ only under~$\mu$,
into a similar result  that views $\partial\mu/\partial\nu$ only under~$\nu$.

\begin{corollary}\label{cor:260922a2}
\begin{enumerate}
\item[\textup{(i)}]
$\nu$-a.e.\ $\partial\mu/\partial\nu$ is $[0,\infty)$-valued.

\item[\textup{(ii)}]
$\nu\ll\mu$ if and only if $\nu$-a.e.\ $\partial\mu/\partial\nu>0$.

\item[\textup{(iii)}]
$\mu\perp\nu$ if and only if $\nu$-a.e.\ $\partial\mu/\partial\nu=0$.
\end{enumerate}
\end{corollary}

We, finally, remark that the unique (Lebesgue) decomposition
$\mu=\mu^{ac}+\mu^s$, where $\mu^{ac}$ and $\mu^s$ are $\sigma$-finite measures on $(\Omega,\mathcal F)$ such that $\mu^{ac}\ll\nu$ and $\mu^s\perp\nu$, can be easily expressed with the help of the generalized Radon--Nikodym derivative,
namely
$$
\mu^{ac}(A)=\int_A \frac{\partial\mu}{\partial\nu}\,d\nu,
\quad
\mu^s(A)=\mu\Big(A\cap\Big\{\frac{\partial\mu}{\partial\nu}=\infty\Big\}\Big),
\quad
A\in\mathcal F,
$$
which easily follows from Lemma~\ref{lem:260922a1}
(alternatively, see \cite[Theorem~A.17]{FollmerSchied} or \cite[(29) in Section~III.9]{Shiryaev}).
In particular, viewed under $\nu$ (but not under $\mu$!), $\partial\mu/\partial\nu$ is a version of the Radon--Nikodym derivative $d\mu^{ac}/d\nu$ of the absolutely continuous part $\mu^{ac}$ of $\mu$ with respect to~$\nu$.

\subsection{Generalized density process and Jessen's theorem}
With the concept of generalized density we can in a natural way generalize Jessen's theorem (Corollary~\ref{cor:030323a1} below) and provide a very simple proof.

We consider a filtered space $(\Omega,\mathcal F,(\mathcal F_n)_{n\in\mathbb N})$ with a discrete-time filtration such that $\mathcal F=\bigvee_{n\in\mathbb N}\mathcal F_n$
(i.e., $\bigcup_{n\in\mathbb N}\mathcal F_n$ generates $\mathcal F$).
In this subsection, for any probability measure $\P$ on $(\Omega,\mathcal F)$, we use the notation
$$
\P_n\triangleq\P|_{\mathcal F_n}
$$
for the restriction of $\P$ to $\mathcal F_n$.

Let $\P$ and $\Q$ be two arbitrary probability measures on $(\Omega,\mathcal F)$.
The process $(\partial\P_n/\partial\Q_n)_{n\in\mathbb N}$ is called the \emph{generalized density process}, where $\partial\P_n/\partial\Q_n$ denote the corresponding generalized Radon-Nikodym derivatives.

\begin{theorem}\label{th:030323a1}
$\P,\Q$-a.s.\ it holds that
$$
\frac{\partial\P_n}{\partial\Q_n}
\to
\frac{\partial\P}{\partial\Q},\quad
n\to\infty.
$$
\end{theorem}

\begin{proof}
Set $\QQ\triangleq\frac{\P+\Q}2$, notice that $\P\ll\QQ$ and consider the process $(Z_n)_{n\in\mathbb N}$, $Z_n\triangleq d\P_n/d\QQ_n$, where $d\P_n/d\QQ_n$ denotes the (standard) Radon-Nikodym derivative.
We have $Z_n=\E^\QQ[d\P/d\QQ|\mathcal F_n]$, $n\in\mathbb N$,
and consequently, the following:
\begin{enumerate}
\item[-]
$(Z_n)_{n \in \mathbb{N}}$ is a uniformly integrable $\QQ$-martingale.

\item[-]
$\QQ$-a.s.\ and in $L^1(\QQ)$, $Z_n\to Z_\infty$, $n\to\infty$, for some random variable $Z_\infty\in L^1(\QQ)$.

\item[-]
For all $n_0\in\mathbb N$ and $A\in\mathcal F_{n_0}$,
$$
\int_A Z_\infty\,d\QQ
=
\lim_{n_0\le n\to\infty}\int_A Z_n\,d\QQ
=
\P(A)
=
\int_A \frac{d\P}{d\QQ}\,d\QQ.
$$

\item[-]
By a monotone class argument, for all $A\in\bigvee_{n\in\mathbb N}\mathcal F_n=\mathcal F$,
$$
\int_A Z_\infty\,d\QQ
=
\int_A \frac{d\P}{d\QQ}\,d\QQ.
$$
This implies that $\QQ$-a.s.\ $Z_\infty=d\P/d\QQ$.
\end{enumerate}
Concluding, $\QQ$-a.s.\ it holds
$$
\frac{d\P_n}{d\QQ_n}
\to
\frac{d\P}{d\QQ}
\quad\text{and}\quad
\frac{d\Q_n}{d\QQ_n}
\to
\frac{d\Q}{d\QQ},
\quad
n\to\infty
$$
(the second claim holds by symmetry).
Now the result follows from~\eqref{eq:260922a1}.
\end{proof}

Let $\P^{ac}$ denote the absolutely continuous part of $\P$ with respect to $\Q$.
Similarly, let $\P_n^{ac}$ be the absolutely continuous part of $\P_n$ with respect to $\Q_n$.
Notice that, in general, $\P_n^{ac}$ is \emph{not} the restriction of $\P^{ac}$ to $\mathcal F_n$.

Recalling the discussion after Corollary~\ref{cor:260922a2}, we obtain Jessen's theorem
(see \cite[Theorem~5.2.20]{stroock})
as an immediate consequence of Theorem~\ref{th:030323a1}.

\begin{corollary}[Jessen's theorem]\label{cor:030323a1}
$\Q$-a.s.\ it holds that
$$
\frac{d\P_n^{ac}}{d\Q_n}
\to
\frac{d\P^{ac}}{d\Q},\quad
n\to\infty.
$$
\end{corollary}

\subsection{Differentiation of Measures}\label{subsec:DiffMeas}
Let \(I \subset \mathbb{R}\) be an open interval and let \(\mu\) and \(\nu\) be locally finite measures on \((I, \mathcal{B}(I))\).
For \(x \in \mathbb{R}\) and \(r > 0\), we write \(B(x, r)\) for the open ball with center \(x\) and radius~\(r\).
A sequence \((A_j)_{j = 1}^\infty \subset \mathcal{B}(I)\) is said to \emph{converge \(\nu\)-measure-metrizably} to \(x \in I\) if the following conditions hold:
\begin{enumerate}
	\item[\textup{(a)}]
	For every \(j \in \mathbb{N}\), there exists an \(r_j > 0\) such that \(A_j \subset B(x, r_j) \subset I\);
	\item[\textup{(b)}]
	\(\lim_{j \to \infty} r_j = 0\);
	\item[\textup{(c)}]
	there exists an \(\alpha \in (0, 1]\) such that for every \(j \in \mathbb{N}\), \(\nu (A_j) \geq \alpha \nu(B(x, r_j))\).
\end{enumerate}
We now define two closely related \emph{derivatives of \(\mu\) w.r.t.\ \(\nu\)} at a point \(x\in I\).
The first one will be denoted by $D_\nu(\mu)(x)$ and the second one by $\overline D_\nu(\mu)(x)$.
If there exists a number \(z \in [0,\infty]\) such that
\[
\lim_{j \to \infty} \frac{\mu (A_j)}{\nu(A_j)} = z,
\]
for every sequence \((A_j)_{j = 1}^\infty\) which converges \(\nu\)-measure-metrizably
(resp., \(\mu\)-measure-metrizably and \(\nu\)-measure-metrizably)
to \(x\), then we write \(D_\nu(\mu)(x)=z\) (resp., \(\overline D_\nu(\mu)(x)=z\)).
If the limit 
\[
\lim_{r \searrow 0} \frac{\mu (B(x, r))}{\nu(B(x, r))}
\]
exists in $[0,\infty]$, it is said to be the \emph{symmetric derivative of \(\mu\) w.r.t.\ \(\nu\)} at \(x\) and we denote it by \(D^{\on{sym}}_\nu(\mu)(x)\).
We notice the obvious relations between the notions:
\begin{align}
D_\nu(\mu)(x)\text{ exists}\quad
&\Longrightarrow\quad
\overline D_\nu(\mu)(x)\text{ exists and equals }D_\nu(\mu)(x);
\\
\overline D_\nu(\mu)(x)\text{ exists}\quad
&\Longrightarrow\quad
D^{\on{sym}}_\nu(\mu)(x)\text{ exists and equals }\overline D_\nu(\mu)(x).
\label{eq:260922a2}
\end{align}
The following is a partial restatement of \cite[Theorem~8.4.4, Corollary~8.4.5]{benedetto}.

\begin{lemma} \label{lem: benedetto}
\begin{enumerate}
\item[\textup{(i)}] \(D_\nu (\mu)\) exists \(\nu\)-a.e.\ and \(D_\nu (\mu) \in L^1_{\textup{loc}} (I, \mathcal{B}(I), \nu)\).

\item[\textup{(ii)}] \(\mu \ll \nu\) if and only if \(\mu (A) = \int_A D_\nu (\mu) d \nu\) for all \(A \in \mathcal{B}(I)\).
\end{enumerate}
\end{lemma}

This observation allows us to infer that $\overline D_\nu(\mu)$ is a version of the generalized Radon--Nikodym derivative~$\partial\mu/\partial\nu$.

\begin{theorem}\label{th:260922a1}
For all locally finite measures $\mu,\nu$ on \((I, \mathcal{B}(I))\), the derivative \(\overline D_\nu (\mu)\) exists \(\mu,\nu\)-a.e.\ on $I$ and is a version of the generalized Radon--Nikodym derivative \(\partial \mu/\partial \nu\).
In particular, the same claims hold also for the symmetric derivative $D^{\on{sym}}_\nu(\mu)$.
\end{theorem}

\begin{proof}
By Lemma~\ref{lem: benedetto}, the derivatives $D_{\mu+\nu}(\mu)$ and $D_{\mu+\nu}(\nu)$ exist $(\mu+\nu)$-a.e.\ and are versions of the Radon--Nikodym derivatives $d\mu/d(\mu+\nu)$ and $d\nu/d(\mu+\nu)$.
For a point $x\in I$ such that $D_{\mu+\nu}(\mu)(x)$ and $D_{\mu+\nu}(\nu)(x)$ exist and 
$$
\text{neither }
D_{\mu+\nu}(\mu)(x)
=
D_{\mu+\nu}(\nu)(x)
=
0
\text{ nor }
D_{\mu+\nu}(\mu)(x)
=
D_{\mu+\nu}(\nu)(x)
=
\infty,
$$
we have $\overline D_\nu(\mu)(x)=D_{\mu+\nu}(\mu)(x)/D_{\mu+\nu}(\nu)(x)$
(notice here that, if a sequence \((A_j)_{j = 1}^\infty\) converges \(\mu\)-measure-metrizably and \(\nu\)-measure-metrizably to~\(x\), then it converges \((\mu+\nu)\)-measure-metrizably to~\(x\)).
The claims for $\overline D_\nu(\mu)$ now follow from~\eqref{eq:260922a1} with $\gamma=\mu+\nu$.
The claims for $D^{\on{sym}}_\nu(\mu)$ are trivial consequences of~\eqref{eq:260922a2}.
\end{proof}

\bibliographystyle{abbrv}
\bibliography{references}

\end{document}